\documentclass[10pt,letterpaper]{article}
\usepackage[text={5.5in, 8.25in},centering]{geometry}
\usepackage{fancyhdr}
\usepackage{amssymb}
\usepackage{amsfonts}
\usepackage{amsmath}
\usepackage{amsthm}

\usepackage{listings}
\usepackage{matlab-prettifier}
\usepackage{graphicx}
\usepackage{placeins}
\usepackage[skip=6pt]{caption}
\usepackage{epsfig,subfigure}
\usepackage{mathrsfs}
\usepackage{multicol,multirow}
\usepackage{placeins}
\usepackage{setspace}
\usepackage{xspace}
\usepackage{appendix}
\usepackage{scalerel} 
\usepackage{ulem}

\usepackage{caption}
\usepackage{url}
\usepackage{todonotes}
\usepackage{xfrac}
\usepackage[alphabetic]{amsrefs}
\usepackage[colorlinks=true, linkcolor=blue, citecolor=black]{hyperref}%
\usepackage[capitalize]{cleveref}
\newcommand\myurl[1]{\url{#1}}

\BibSpec{webpage}{%
  +{}{\PrintAuthors} {author}
  +{,}{ \textit} {title}
  +{}{ \parenthesize} {date}
  +{,}{ \myurl} {myurl}
  +{,}{ } {note}
  +{.}{ } {transition}
}

\newtheorem{theorem}{Theorem}
\newtheorem{lemma}[theorem]{Lemma}

\newtheorem{proposition}[theorem]{Proposition}
\newtheorem{definition}[theorem]{Definition}

\newtheoremstyle{restated}{}{}{\itshape}{}{\bfseries}{.}{.5em}{\thmnote{#3}}
\theoremstyle{restated}

\theoremstyle{remark}

\newtheorem{remark}{Remark}

\crefname{definition}{Definition}{Definitions}
\crefname{lemma}{Lemma}{Lemmas}
\crefname{theorem}{Theorem}{Theorems}
\crefname{proposition}{Proposition}{Propositions}
\crefname{figure}{Figure}{Figures}

\Crefname{definition}{definition}{definitions}
\Crefname{lemma}{lemma}{lemmas}
\Crefname{theorem}{theorem}{theorems}
\Crefname{proposition}{proposition}{propositions}
\Crefname{figure}{figure}{figures}

\usepackage{relsize}
\usepackage{bm} 
\providecommand{\mathbold}[1]{\bm{\mathsf{#1}}}
\newcommand{\R}{\mathbb{R}}

\newcommand{\E}{\mathbb{E}}
\newcommand{\vct}[1]{\bm{#1}}

\newcommand{\mtx}[1]{\mathbold{#1}}

\usepackage{stmaryrd}

\newcommand\scalemath[2]{\scalebox{#1}{\mbox{\ensuremath{\displaystyle #2}}}}
\newcommand{\norm}[1]{\left\|#1\right\|}
\newcommand{\sigmaMin}[1]{\sigma_{\min}\left(#1\right)}
\newcommand{\sigmaMax}[1]{\sigma_{\max}\left(#1\right)}
\newcommand{\sigmaSub}[2]{\sigma_{#1}\left(#2\right)}

\makeatletter
\newsavebox\myboxA
\newsavebox\myboxB
\newlength\mylenA
\newcommand*\xoverline[2][0.75]{%
    \sbox{\myboxA}{$\m@th#2$}%
    \setbox\myboxB\null
    \ht\myboxB=\ht\myboxA%
    \dp\myboxB=\dp\myboxA%
    \wd\myboxB=#1\wd\myboxA
    \sbox\myboxB{$\m@th\overline{\copy\myboxB}$}
    \setlength\mylenA{\the\wd\myboxA}
    \addtolength\mylenA{-\the\wd\myboxB}%
    \ifdim\wd\myboxB<\wd\myboxA%
       \rlap{\hskip 0.5\mylenA\usebox\myboxB}{\usebox\myboxA}%
    \else
        \hskip -0.5\mylenA\rlap{\usebox\myboxA}{\hskip 0.5\mylenA\usebox\myboxB}%
    \fi}
\makeatother

\makeatletter
\newcommand*\xunderline[2][0.75]{%
    \sbox{\myboxA}{$\m@th#2$}%
    \setbox\myboxB\null
    \ht\myboxB=\ht\myboxA%
    \dp\myboxB=\dp\myboxA%
    \wd\myboxB=#1\wd\myboxA
    \sbox\myboxB{$\m@th\underline{\copy\myboxB}$}
    \setlength\mylenA{\the\wd\myboxA}
    \addtolength\mylenA{-\the\wd\myboxB}%
    \ifdim\wd\myboxB<\wd\myboxA%
       \rlap{\hskip 0.5\mylenA\usebox\myboxB}{\usebox\myboxA}%
    \else
        \hskip -0.5\mylenA\rlap{\usebox\myboxA}{\hskip 0.5\mylenA\usebox\myboxB}%
    \fi}
\makeatother

\makeatletter
\newcommand*\xwidehat[2][0.8]{%
    \sbox{\myboxA}{$\m@th#2$}%
    \setbox\myboxB\null
    \ht\myboxB=\ht\myboxA%
    \dp\myboxB=\dp\myboxA%
    \wd\myboxB=#1\wd\myboxA
    \sbox\myboxB{$\m@th\widehat{\copy\myboxB}$}
    \setlength\mylenA{\the\wd\myboxA}
    \addtolength\mylenA{-\the\wd\myboxB}%
    \ifdim\wd\myboxB<\wd\myboxA%
       \rlap{\hskip 0.5\mylenA\usebox\myboxB}{\usebox\myboxA}%
    \else
        \hskip -0.5\mylenA\rlap{\usebox\myboxA}{\hskip 0.5\mylenA\usebox\myboxB}%
    \fi}
\makeatother

\DeclareFontFamily{U}{mathx}{\hyphenchar\font45}
\DeclareFontShape{U}{mathx}{m}{n}{
      <5> <6> <7> <8> <9> <10>
      <10.95> <12> <14.4> <17.28> <20.74> <24.88>
      mathx10
      }{}
\DeclareSymbolFont{mathx}{U}{mathx}{m}{n}
\DeclareFontSubstitution{U}{mathx}{m}{n}
\DeclareMathAccent{\widecheck}{0}{mathx}{"71}
\DeclareMathAccent{\wideparen}{0}{mathx}{"75}

\makeatletter
\newcommand*\xwidecheck[2][0.8]{%
    \sbox{\myboxA}{$\m@th#2$}%
    \setbox\myboxB\null
    \ht\myboxB=\ht\myboxA%
    \dp\myboxB=\dp\myboxA%
    \wd\myboxB=#1\wd\myboxA
    \sbox\myboxB{$\m@th\widecheck{\copy\myboxB}$}
    \setlength\mylenA{\the\wd\myboxA}
    \addtolength\mylenA{-\the\wd\myboxB}%
    \ifdim\wd\myboxB<\wd\myboxA%
       \rlap{\hskip 0.5\mylenA\usebox\myboxB}{\usebox\myboxA}%
    \else
        \hskip -0.5\mylenA\rlap{\usebox\myboxA}{\hskip 0.5\mylenA\usebox\myboxB}%
    \fi}
\makeatother

\newcommand{\sktext}{\mathrm{sk}}
\newcommand{\sk}[1]{
    {#1}^{\sktext}
}
\newcommand{\pre}[1]{
    {#1}^{\mathrm{pre}}
}

\newcommand{\trun}[1]{
	{#1}
}

\newcommand{\Roo}{\mtx{A}}
\newcommand{\Rot}{\mtx{B}}
\newcommand{\Rtt}{\mtx{C}}

\newcommand{\roundoff}{\mathsf{u}\xspace}

\newcommand{\orthtol}{\epsilon_{\mathrm{tol}}}
\newcommand{\distortion}{\delta}
\usepackage{lmodern}

\newcommand{\fmtx}[1]{
\mtx{#1}
}

\DeclareMathOperator{\range}{range}
\DeclareMathOperator{\rank}{rank}
\DeclareMathOperator{\diag}{diag}

\DeclareMathOperator{\cond}{\kappa}

\newcommand{\RandLAPACK}{\textsf{RandLAPACK}\xspace}
\newcommand{\RBLAS}{\textsf{RandBLAS}\xspace}
\newcommand{\RandBLAS}{\RBLAS}
\newcommand{\BLAS}{\textsf{BLAS}\xspace}
\newcommand{\BLASlev}[1]{\textsf{Level~#1 BLAS}\xspace}
\newcommand{\LAPACK}{\textsf{LAPACK}\xspace}

\newcommand{\code}[1]{\texttt{#1}}

\newcommand{\leadslicespace}{\hspace{0.125em}}
\newcommand{\trailslicespace}{\hspace{0.025em}}
\newcommand{\fslice}{\leadslicespace{}{:}\trailslicespace{}}  
\newcommand{\lslice}[1]{{1}{:}{#1}}  
\newcommand{\tslice}[1]{{#1}{:}}  
\newcommand{\trans}{*}

\usepackage{algorithm}
\usepackage{algorithmicx}   
\usepackage[noend]{algpseudocode}  
\algdef{SE}[SUBALG]{Indent}{EndIndent}{}{\algorithmicend\ }
\algtext*{Indent}
\algtext*{EndIndent}
\algdef{SE}[DOWHILE]{Do}{doWhile}{\algorithmicdo}[1]{\algorithmicwhile\ #1}%
\newcommand{\codecomment}[1]{\textcolor{gray}{\texttt{//} #1}}
\usepackage[capitalize]{cleveref}
\usepackage{overpic}
\usepackage{xcolor}
\usepackage{stackengine}
\usepackage[graphicx]{realboxes}

\usepackage{lineno}

\newcommand{\footremember}[2]{%
    \footnote{#2}
    \newcounter{#1}
    \setcounter{#1}{\value{footnote}}%
}
\newcommand{\footrecall}[1]{%
    \footnotemark[\value{#1}]%
} 

\newcommand\blfootnote[1]{%
  \begingroup
  \renewcommand\thefootnote{}\footnote{#1}%
  \addtocounter{footnote}{-1}%
  \endgroup
}

\title{
\vspace{-1cm}
CholeskyQR with Randomization and Pivoting for Tall Matrices (CQRRPT)}
\author{
    Maksim Melnichenko\footremember{ICL}{Innovative Computing Laboratory, University of Tennessee, Knoxville}
    \and
    Oleg Balabanov\footremember{Berkeley}{University of California, Berkeley}\footremember{ICSI}{International Computer Science Institute (ICSI)}
    \and 
    Riley Murray\footremember{Sandia}{Sandia National Laboratories}
    \and
    James Demmel\footrecall{Berkeley}
    \and
    Michael W. Mahoney\footrecall{ICSI} \footremember{LBNL}{Lawrence Berkeley National Laboratory} \footrecall{Berkeley}
    \and
    Piotr Luszczek\footremember{MITLL}{MIT Lincoln Laboratory}\footrecall{ICL}
}

\begin{document}

\maketitle


\begin{abstract}
This paper develops and analyzes a new algorithm for QR decomposition
with column pivoting (QRCP) of rectangular matrices with many more rows than columns.
The algorithm carefully combines methods from randomized numerical linear algebra
to accelerate pivot
decisions for the input matrix \textit{and} the process of decomposing
the pivoted matrix into the QR form.
The source of the latter improvement is CholeskyQR with randomized preconditioning.
Comprehensive analysis is provided in both exact and finite-precision
arithmetic to characterize the algorithm's rank-revealing properties
and its numerical stability granted probabilistic assumptions of the
sketching operator.
An implementation of the proposed algorithm is described and made
available inside the open-source \RandLAPACK library, which itself
relies on \RandBLAS.
 Experiments with this implementation on an Intel Xeon Gold 6248R CPU demonstrate order-of-magnitude speedups over \LAPACK's standard function for QRCP, and comparable performance to a specialized algorithm
for unpivoted QR of tall matrices, which lacks the strong rank-revealing
properties of the proposed method.    \blfootnote{
        Funding acknowledgments and additional affiliations appear at the end of the paper.
    }\blfootnote{
    Send correspondence to mmelnic1@vols.utk.edu.
    }
\end{abstract}

\section{Introduction}
\label{intro}

The QR factorization, considered one of ``The Big Six Matrix Factorizations''~\cite{Higham:blog:big6}, is fundamental to the field of numerical linear algebra.
Its applications range from linear least squares problems~\cite{bjork1996lsq} and block orthogonalization~\cite{Stath:2002} to contemporary randomized low-rank approximation algorithms~\cite{tropp2023randomized}.
Beyond the basic QR decomposition, QR with column pivoting (QRCP)
has additional benefits for numerically challenging problems, as it can help with rank-revealing tasks.
To describe this problem concretely, let $\mtx{M}$ be a matrix of size $m \times n$ with $m \geq n$.
QRCP is concerned with finding a column permutation matrix $\mtx{\Pi}$ and
a QR decomposition of the product $\mtx{M}\mtx{\Pi}$, i.e.,
\[
    \mtx{M}\mtx{\Pi} = \mtx{Q}\mtx{R}
\]
such that the information on the leading and trailing singular values of
$\mtx{M}$ can be inferred from the spectra of leading and trailing blocks in $2 \times 2$ partitions of $\mtx{R}$.

\begin{remark}
    We use the term \textit{spectrum} in reference to singular values, not eigenvalues.
    Eigenvalues are of interest to us only insofar as they coincide with singular values for positive semidefinite matrices.
\end{remark}

QRCP is considerably more expensive than unpivoted QR from a communication standpoint, even with straightforward pivoting strategies.
For example, Householder QR with Businger and Golub's \textit{max-norm pivoting} requires updating column norms of every partial decomposition of $\mtx{M}$ as the algorithm progresses from left to right across the columns \cite{BG:1965:QRCP}.
These column-norm updates entail matrix-vector (BLAS 2) operations,
which are far less suitable for modern hardware than the matrix-matrix (BLAS 3) operations abundant in classic algorithms for unpivoted QR.
The significance of this problem has been known for decades \cite{Quintana:1998}, and has compounded over time.

We use methods common in Randomized Numerical Linear
Algebra~(RandNLA) to develop a fast and reliable QRCP algorithm
we call \textit{CQRRPT},\footnote{pronounced ``see-crypt''} which stands
for ``CholeskyQR with Randomization and Pivoting for Tall matrices.''
As we show below, CQRRPT can outperform not only other QRCP algorithms
but also specialized communication-avoiding algorithms for unpivoted QR, such as shifted CholeskyQR3 \cite{FK2020} and TSQR \cite{demmel2012communication}.
The dominant cost of our algorithm is attributed exclusively to BLAS 3 operations and a random projection, rendering it highly efficient from a communication standpoint.
On distributed memory systems, it can be implemented with just two
sum-reduce exchanges.
Compared to shifted CholeskyQR3, CQRRPT requires $\tfrac{4}{3}\times$
fewer data passes and $\tfrac{3}{2}\times$ fewer inter-processor messages.
Furthermore, compared to TSQR, CQRRPT lends itself more easily to
performance optimization on massively parallel architectures, thanks to its simple reduction operator.
The highly efficient computational profile, together with its numerical stability, 
suggests that CQRRPT essentially ``solves'' the algorithm design problem of QRCP for tall matrices.

\subsection{Where is the randomness?}\label{subsec:where_random}

CQRRPT uses randomness only once: as its very first step, it samples a linear
dimension reduction map from a prescribed distribution.
The dimension reduction map is called a \textit{sketching operator} and is said to be sampled from a \textit{sketching distribution}.
Its next step is to apply the sketching operator to the input matrix to produce a \textit{sketch}.
Computing a sketch can be as simple as selecting rows from a matrix and
as complicated as fast Fourier
transforms.
Our implementation of CQRRPT uses structured sparse sketching operators,
since these can provide exceptional speed without sacrificing reliability of the algorithm.

Analysis of sketching-based algorithms often comes in two parts.
In the first part, the algorithm is analyzed as a purely deterministic method, conditional on an event that the sketch retains suitable geometric information from the input matrix.
The second part uses tools from random matrix theory to bound the probability with which the desirable event happens, where the bound is \textit{unconditional on numerical properties of the input matrix}.
Our analysis of CQRRPT follows this tradition.
Our presentation strongly emphasizes the first phase of analysis, since the probabilistic event that the first phase conditions on is extremely well understood.

We refer the reader to the recent monograph \cite{RandLAPACK_Book} for a broad overview of how randomness can be leveraged in numerical linear algebra.

\subsection{The CholeskyQR method and its variants}
\label{subsec:cholqr}

The idea behind CholeskyQR is simple.
Given the $m \times n$ matrix $\mtx{M}$ with $m \geq n$, compute
the Gram matrix $\mtx{G} = \mtx{M}^{\trans}\mtx{M}$.
If we can
factor $\mtx{G} = \mtx{R}^{*}\mtx{R}$ where $\mtx{R}$ is an invertible upper-triangular matrix, then we can obtain the QR decomposition's orthogonal factor -- hereafter, the ``Q-factor'' -- as
$\mtx{Q} = \mtx{M}\mtx{R}^{-1}$.
Note that this procedure only works if $\mtx{M}$ has rank $n$.

\paragraph{Computational profile.}
Computing $\mtx{G}$ and $\mtx{R}$ in CholeskyQR can be carried out by
the BLAS function \code{SYRK} and the \LAPACK function \code{POTRF},
respectively.
The Q-factor can be formed explicitly with the \BLASlev{3}
function \code{TRSM}.
The \textit{flop count} for CholeskyQR using these three functions is
roughly $2mn^{2} + n^{3}/3$ \cite[Page 120]{LAWN41:1994}, which is close
to the $2mn^2 - (2/3)n^{3}$ flop count of Householder QR computed by
\LAPACK's \code{GEQRF} function \cite[Page 122]{LAWN41:1994} when
$m \geq n$.

Notably, CholeskyQR returns an explicit Q-factor, while \code{GEQRF} outputs an implicit representation of its Q-factor.
The implicit representation is preferable for manipulating the full $m \times m$ orthogonal operator in a ``non-economic'' QR decomposition.
Nevertheless, many common numerical methods, such as subspace projection methods or singular value decomposition of tall matrices, require computing the Q-factor explicitly.
Translating \code{GEQRF}'s implicit factor into an explicit form (using \LAPACK{}'s \code{ORGQR}) takes an additional 
$2mn^2 - (2/3)n^3$~flops~\cite[Page 122]{LAWN41:1994}.
Consequently, when $r := m/n > 1$, CholeskyQR requires a fraction $(6r+1)/(12r+4)$ of the flops of the traditional method of obtaining a fully explicit QR factorization.

\paragraph{Limitations.}
Despite its simplicity and speed, CholeskyQR is rarely used in practice.
One clear shortcoming is that it fails if the computed Gram matrix is numerically rank-deficient, which can happen even if the input matrix is numerically full-rank.
More generally, if $\cond(\mtx{M})$ denotes the condition number of
$\mtx{M}$ in the spectral norm and $\roundoff$ denotes the unit roundoff in working
precision, then rounding errors can lead to significant \textit{orthogonality loss}, in the sense that $\|\mtx{Q}^{\trans}\mtx{Q} - \mtx{I}\|_2$ can be as large as $\mathcal{O}(\roundoff{}\cond(\mtx{M})^2)$.

The orthogonality loss of CholeskyQR can be mitigated by a variety of methods.\footnote{
Indeed, this paper adds CQRRPT to the list of orthogonality-loss mitigation methods, even though CQRRPT is concerned with pivoted rather than unpivoted QR.}
For example, \textit{Jacobi preconditioning} normalizes all columns of $\mtx{M}$ to have unit norm and brings the condition number of $\mtx{M}$ to within a factor $\sqrt{n}$ of the best-possible diagonal preconditioner~\cite[Theorem 3.5]{V1969}; similar results are available for normalizing column blocks via a block-diagonal preconditioner \cite{demmel23}.
Recently, it has been suggested to obtain a preconditioner from an LU decomposition of $\mtx{M}$ \cite{TOO2020} or from CholeskyQR on a regularized version of $\mtx{M}$ \cite{FK2020}.
Another approach to mitigating orthogonality loss involves reorthogonalization (essentially running CholeskyQR twice) resulting in ``CholeskyQR2'' \cite{FNYY2014,yamamoto2015roundoff}.
However, this method can still fail when $\cond(\mtx{M}) > \roundoff{}^{-1/2}$.
Finally, a mixed precision approach computes the Gram matrix in higher precision whereby the problematic conditioning can be better controlled~\cite{yamazaki2014cholqr}.

\subsection{Randomized preconditioning for CholeskyQR}
\label{subsec:rand_cholqr}

The following algorithm serves as the starting point of our work.
Given the matrix $\mtx{M}$, it sketches $\mtx{M}$ and then uses the triangular factor of the sketch's QR decomposition as a preconditioner in CholeskyQR.
Like standard (unpreconditioned) CholeskyQR, it has a hard requirement that $\mtx{M}$ is full rank in exact arithmetic.

\begin{enumerate}
    \item Compute $\sk{\mtx{M}} = \mtx{S}\mtx{M}$ using a sketching operator $\mtx{S} \in \R^{d \times m}$ with $n \leq d \leq m$.
    \item Compute 
      $[\sk{\mtx{Q}}, \sk{\mtx{R}}] = \code{qr}(\sk{\mtx{M}})$ and discard $\sk{\mtx{Q}}$.
    \item Form $\pre{\mtx{M}} = \mtx{M}  (\sk{\mtx{R}})^{-1}$ explicitly, using back substitution.
    \item Compute $\mtx{G} = (\sk{\mtx{M}})^{\trans}(\sk{\mtx{M}})$ and $\pre{\mtx{R}} = \code{chol}(\mtx{G}, \text{``upper''})$.
    \item Set  $\mtx{Q} = \pre{\mtx{M}} (\pre{\mtx{R}})^{-1}$ and $\mtx{R}=\pre{\mtx{R}} \sk{\mtx{R}}$, then return $\mtx{Q}$ and $\mtx{R}$.
\end{enumerate}

This method was first introduced~\cite{FGL:2021:CholeskyQR} without the context of the broader RandNLA literature; it was first studied in detail with the RandNLA literature in mind in \cite{Balabanov:2022:cholQR} and subsequently in \cite{HSBY:2023:rand_chol_qr}.
Notably, the possibility of first computing a sketch-orthonormal Q-factor $\pre{\mtx{M}}$, and subsequently retrieving the
$\ell_2$-orthonormal Q-factor via a CholeskyQR on $\pre{\mtx{M}}$ was
originally proposed
elsewhere~\cite{BG:2021:GramSchmidt,https://doi.org/10.48550/arxiv.2111.14641}.
Furthermore, the idea of using $\sk{\mtx{R}}$ as the preconditioner traces back to the \textit{sketch-and-precondition} paradigm for overdetermined least squares~\cite{RT:2008:SAP,AMT:2010:Blendenpik,MSM:2014:LSRN}.\footnote{In this context, $\pre{\mtx{M}}$ was only formed implicitly and was accessed as a linear operator.}
If a well-studied condition called a \textit{subspace embedding
property} holds between $\mtx{S}$ and
$\range(\mtx{M})$, then $\pre{\mtx{M}}$ will be nearly-orthonormal \cite{DMM06,Sarlos:2006,DMMS07_FastL2_NM10,Mah-mat-rev_BOOK}.

For any reasonable sketching operator, $\mtx{S}$, the flop count of the
algorithm outlined above will be $\mathcal{O}(mn^2)$.
Whether or not practical speedup is observed depends greatly on the
sketching distribution and on the method used to compute the product
$\mtx{S} \mtx{M}$.
If a fast sketching operator is used with $n \leq d \leq m$, then the leading term in the resulting algorithm's flop count will be $3mn^2$, which is only $mn^2$ more than the standard CholeskyQR.

\subsection{Randomization in QR with column pivoting}
\label{subsec:earlier_rand_qrcp}

Substantial efforts have been devoted to the design of randomized algorithms for QRCP of general matrices, i.e., rectangular matrices of any aspect ratio of numbers of rows and columns.
These efforts originate with independent contributions by
Martinsson~\cite{Martinsson:2015:QR} as well as Duersch and
Gu~\cite{DG:2017:QR}, with subsequent extensions by Martinsson
\textit{et al.}~\cite{MOHvdG:2017:QR} and also by Xiao, Gu, and
Langou~\cite{XGL:2017:RandQRCP}.
While there are many variations on these methods, they share a common structure that we outline here.

These methods take in integers $(b,s)$ where $b > 0$, $s \geq 0$, and $n \geq b + s$.
They form a sketch 
$\mtx{Y} = \mtx{S} \mtx{M}$ with $b + s$ rows, and then proceed with a three-step iterative loop.
\begin{enumerate}
    \item[1.] Use any QRCP method to find 
    $P_{\text{block}}$, a length-$b$ vector containing column indices for the first
    $b$ pivots for the wide matrix $\mtx{Y}$.
    \item[2.] Process the tall matrix $\mtx{M}[\fslice{},P_{\text{block}}]$
    by QRCP or unpivoted QR.
    \item[3.] Suitably update $\mtx{M}$ and $\mtx{Y}$, then return to Step 1.
\end{enumerate}
The update to $\mtx{M}$ at Step 3 can be handled by standard methods, such as those in blocked unpivoted Householder QR.
The update to $\mtx{Y}$ is more subtle.
If done appropriately, then the leading term in the algorithm's flop count can match that of Householder QR~\cite{DG:2017:QR}.

A core limitation of these methods is that their best performance is attained with small block sizes.
For example, $b = 64$ and $s = 10$ is used in both shared- and
distributed-memory settings~\cites{MOHvdG:2017:QR,XGL:2017:RandQRCP}.
Much larger block sizes would be needed for the updating operations to
be performed near the limit of machines' peak performance rates.
One of our motivations for developing CQRRPT has been to provide an
avenue to extend earlier randomized QRCP algorithms to block sizes on
the order of thousands.

\subsection{Our contributions and outline}
\label{sxn:our_contributions}

CQRRPT takes the method from \cref{subsec:rand_cholqr} and
replaces the call to a QR function on $\sk{\mtx{M}}$
with a call to a QRCP function on $\sk{\mtx{M}}$.
It uses the results from this call to QRCP to provide: the pivot
indices, an estimate of $\mtx{M}$'s numerical rank, and the information
for our preconditioner in the form of an upper-triangular
matrix~$\sk{\mtx{R}}$.
The numerical rank estimate is needed to deal with the possibility that $\mtx{M}$ might be rank-deficient.

CQRRPT's discovery was reported in two independent
works~\cite{Balabanov:2022:cholQR} and \cite{RandLAPACK_Book}.
This paper is a joint effort by the authors of these works to (1)
conduct thorough theoretical analysis of this algorithm, and (2)
demonstrate experimental results of a high-performance implementation of
the algorithm using \RandBLAS and \RandLAPACK.

\cref{sec:introduce_CQRRPT} formally introduces CQRRPT as \cref{alg:CQRRPT_full}.
Our formalism emphasizes how the effect of randomness on CQRRPT's behavior can be isolated in the choice of distribution used for the sketching operator.
The core operations of CQRRPT are presented in \cref{alg:CQRRPT}, which can be analyzed as a purely deterministic algorithm.
\cref{thm:correctness} establishes the correctness of CQRRPT's output, and \cref{thm:cond_of_A_pre} characterizes the spectrum of the preconditioned input matrix.
The remaining results in the section, \cref{thm:cqrrpt_RRQR,thm:cqrrpt_SRRQR}, concern the quality of the pivots obtained for the input matrix.

\cref{sec:finite_precision} concerns CQRRPT's numerical stability.
Its main result, \cref{cor:simple_stability}, says that CQRRPT can be implemented to provide numerical stability that is \textit{unconditional} on numerical properties of the input matrix.
Specifically, with an appropriate criterion for determining numerical rank, CQRRPT can produce a decomposition with relative error and orthogonality loss on the order of machine precision, as long as the unit roundoff and problem dimensions satisfy $\roundoff \leq m^{-1}F(n)^{-1}$ for some low-degree polynomial $F(n)$.
After outlining the proof of this result we explain how numerical rank selection can be performed in practice.
\cref{app:numerical_stability_proofs} states and proves a more technical version of \cref{cor:simple_stability} in the form of \cref{thm:stabprecond}.

\cref{sec:empirical_pivots} empirically investigates pivot quality.
It shows how easy-to-compute metrics of pivot quality compare when running the \LAPACK default function (\code{GEQP3}) versus when running CQRRPT (based on applying \code{GEQP3} to $\sk{\mtx{M}}$).
The results show the interaction between \textit{coherence} -- which essentially represents a condition number for the act of sampling -- and parameter choices for sparse sketching operators.

\cref{sec:speed_experiments} provides performance experiments with a \RandLAPACK implementation of CQRRPT.
The experiments use a machine with dual 24-core Xeon Gold 6248R CPUs, where Intel MKL and \RandBLAS are the underlying linear algebra libraries.
We first compare CQRRPT with alternative methods for both unpivoted and pivoted QR factorizations.
Our algorithm handily beats competing methods for large matrices, and it has the advantage of a simpler representation of $\mtx{Q}$.
We then investigate how CQRRPT's performance might be improved.
Runtime profiling results which show that computing QRCP of the sketch via \code{GEQP3} can exceed the cost of CholeskyQR several times over.
This motivates an experiment where QRCP on the sketch is handled by HQRRP \cite{MOHvdG:2017:QR} (one of the randomized algorithms for QRCP of general matrices described in \cref{subsec:earlier_rand_qrcp}) instead of~\code{GEQP3}.

Concluding remarks are given in \cref{sec:conclusion}.

\subsection{Definitions and notation}\label{subsec:def_and_notations}

\subsubsection{The mundane}

Matrices appear in boldface sans-serif capital letters.
The transpose of a matrix $\mtx{X}$ is given by $\mtx{X}^{\trans}$, its $\ell_2$ condition number is $\cond(\mtx{X})$, and its elementwise absolute value is $|\mtx{X}|$.
Numerical vectors appear as boldface lowercase letters, while index vectors appear as uppercase letters.
The $k \times k$ identity matrix is denoted by $\mtx{I}_k$.
We enumerate components of matrices and vectors with indices starting from one, rather than starting from zero.
We extract the leading $k$ columns of $\mtx{X}$ by writing $\mtx{X}[\fslice{},\lslice{k}]$, while its trailing $n-k$ columns are extracted by writing $\mtx{X}[\fslice{},\tslice{k+1}{n}]$.
The $(i,j)^{\text{th}}$ entry of $\mtx{X}$ is $\mtx{X}[i,j]$.
Similar conventions apply to extracting the rows of a matrix or components of a vector.

The matrix we ultimately aim to decompose is denoted by $\mtx{M}$ and has dimensions $m \times n$.
We use the letters $k$ and $\ell$ as integer indices between 1 and $n$.
Their precise meaning is determined by context; we are free to redefine them at any time.

\subsubsection{Some necessary evils}

Virtuous notation is readable and intuitive.
Despite our best efforts, we have been unable to devise wholly virtuous notation for this manuscript.
We have settled instead for notation that is readable and \textit{internally consistent} for matrices with similar structures.

\paragraph{Leading columns of matrices subject to pivoting and QR.}
Suppose $\mtx{X}$ is a matrix with $n$ columns that we aim to decompose via QR after column pivoting by $J$ (a permutation vector of $\{1,\ldots,n\}$).
For each $k \in \{1,\ldots,n\}$, we define the truncated, pivoted matrix
\[
    \mtx{X}_k := \mtx{X}[\fslice{},J[\lslice{k}]].
\]
Taking $k = n$ lets us refer to a pivoted matrix without any truncation.

    We use this notation from the outset with the $m \times n$ matrix $\mtx{X} = \mtx{M}$ and the $d \times n$ matrix $\mtx{X} = \sk{\mtx{M}}$.
    \cref{sec:finite_precision} extends this convention to a matrix denoted by ``$\pre{\mtx{M}}$,'' which is decomposed via \textit{unpivoted} QR.

\paragraph{Column-pivoted QR decompositions, and their factors.}

A QR decomposition consists of two conformable matrices: an orthonormal matrix called the \textit{Q-factor} and an upper-trapezoidal matrix called the \textit{R-factor}.
We equip these factors with special indexing rules.
Specifically, the first $k$ columns of a Q-factor ``$\mtx{Q}$'' are denoted by $\mtx{Q}_k$, and the first $k$ rows of an R-factor ``$\mtx{R}$'' are denoted by $\mtx{R}_k$.

    This manuscript features three matrices with important QR factors: a $d \times n$ pivoted matrix $\sk{\mtx{M}}_n$, a matrix $\pre{\mtx{M}}$ of dimensions $m \times k$ ($k \leq n$), and an $m \times n$ pivoted matrix $\mtx{M}_n$.
    The QR factors of these matrices are distinguished from one another with superscripts, or the absence thereof (e.g., $\sk{\mtx{R}}$, $\pre{\mtx{R}}$, and $\mtx{R}$).

\begin{definition}\label{def:qrcp_alg}
Let $\mtx{X}$ denote a rank-$k$ matrix with $n$ columns.
A \emph{column-pivoted QR decomposition} of $\mtx{X}$ consists of QR factors $(\mtx{Q},\mtx{R})$ and a length-$n$ permutation vector $J$ such that the pivoted matrix $\mtx{X}_n = \mtx{X}[\fslice{},J]$ satisfies $\mtx{X}_n = \mtx{Q}_k\mtx{R}_k$.
\end{definition}

\paragraph{Partitions of R-factors.}
We partition R-factors with notation that is similar to the celebrated paper of Gu and Eisenstat \cite{GE:1996}.
Specifically, to any $k \times n$ upper-trapezoidal matrix $\mtx{R}$ and any $\ell \leq k$, we associate the distinguished submatrices
\begin{equation*} 
    \Roo_{\ell} = \mtx{R}[\lslice{\ell},\ \lslice{\ell}], \quad \Rot_{\ell} = \mtx{R}[\lslice{\ell},\ \tslice{\ell+1}n],\quad\text{and}\quad
    \Rtt_{\ell} = \mtx{R}[\tslice{\ell+1}k,\ \tslice{\ell+1}n].
\end{equation*}
For fixed $\ell$, these matrices let us cleanly express $\mtx{R}$ as a $2 \times 2$ block matrix 
\[
    \mtx{R} = \begin{bmatrix}
                    \Roo_{\ell} & \Rot_{\ell} \\
                    \null  & \Rtt_{\ell}
                \end{bmatrix}.
\]
This notation helps in analyzing our algorithm's rank-revealing properties and numerical stability.
We use it with R-factors from the QR decompositions of $\sk{\mtx{M}}_n$, $\pre{\mtx{M}}$, and $\mtx{M}_n$.

\section{Getting to know our algorithm}
\label{sec:introduce_CQRRPT}

\cref{subsec:where_random} mentioned a two-phase approach to analysis of randomized algorithms.
Here we describe CQRRPT in a way that puts this approach front and center.
We do this by delineating between the full algorithm (which includes a random sampling step) and the algorithm's core (which is purely deterministic).

The full algorithm appears below.
Its speed and reliability are affected decisively by the sketching distribution, and it has two parameters that control this distribution.
The higher-level parameter is a \textit{distribution family};
this is an association of matrix dimensions to probability distributions over matrices with those dimensions.
The lower-level parameter, called the \textit{sampling factor}, sets the size of the sketch in proportion to $n$.

\begin{algorithm}[htb]
\small \caption{CholeskyQR with randomization and pivoting for tall matrices}
\label{alg:CQRRPT_full}
\begin{algorithmic}[1]

\vspace{0.15em}
\Statex \textbf{Required inputs.} An $m \times n$ matrix $\mtx{M}$.

\vspace{0.45em}

\Statex \textbf{Optional inputs.}  A distribution family $\mathscr{F}$ and a sampling factor $\gamma  \geq 1$.

\vspace{0.45em}

\Statex \textbf{Outputs.} QR factors of dimensions $m \times k$ and $k \times n$, and a length-$n$ permutation vector.
These comprise a column-pivoted QR decomposition of $\mtx{M}$ in the sense of \cref{def:qrcp_alg} if and only if $\rank(\mtx{S}\mtx{M}) = \rank(\mtx{M})$; see \cref{thm:correctness}.

\vspace{0.35em}

\setstretch{1.25}
\State \textbf{function} $[\mtx{Q},\mtx{R},J] = \code{cqrrpt}(\mtx{M},\gamma,\mathscr{F})$
\vspace{-4pt}
\Indent
    \State If $\mathscr{F}$ is not provided, set it to a default family of sparse sketching distributions. \label{line:operator}
    \State If $\gamma$ is not provided, set $\gamma = 1.25$.
    \State Set $d = \lceil \gamma n \rceil$, and randomly sample a $d \times m$ matrix $\mtx{S}$ from $\mathscr{F}_{d,m}$.
    \State Compute $[\mtx{Q},\mtx{R},J] = \code{cqrrpt\_core}(\mtx{M},\mtx{S})$
    \State \textbf{return}
\EndIndent    
\end{algorithmic}
\end{algorithm}

\FloatBarrier

CQRRPT's core appears in \cref{alg:CQRRPT}. We characterize its behavior in \cref{subsec:determ_analysis,subsec:rank_revealing} using the concept of \textit{restricted singular values}.
The restricted singular values of a $d \times m$ matrix $\mtx{S}$ on a subspace $L \subset \R^m$ are the singular values of $\mtx{S}\mtx{U}$ where $\mtx{U}$ is any orthonormal matrix with range $L$.\footnote{Note that these singular values do not depend on the choice of orthonormal basis.}
The \textit{restricted condition number} of $\mtx{S}$ on $L$, denoted $\cond(\mtx{S}|L)$, is the ratio of its largest to smallest restricted singular values.

\cref{subsec:sketching_in_CQRRPT} explains how the distribution family and sampling factor affect the probabilistic behavior of \cref{alg:CQRRPT_full}.
By the end of this section, it will be clear that standard results from random matrix theory can be used to probabilistically yet rigorously bound the restricted singular values of $\mtx{S}$ on $\range(\mtx{M})$.
This will show that there are many ways of using CQRRPT to achieve different trade-offs between speed and reliability.

\begin{remark}
    The analysis in this section is conducted in exact arithmetic.
    Such analysis may seem strange in the context of a CholeskyQR algorithm, but it plays a valuable role in plotting the course for our finite-precision analysis.
\end{remark}

\subsection{CQRRPT's deterministic core}\label{subsec:determ_analysis}

The algorithm below relies on functions called ``$\code{qrcp}$'' and ``$\code{rank}$.''
For now, we only assume that \code{qrcp} always returns column-pivoted QR decompositions in the sense of \cref{def:qrcp_alg}.
The details of \code{rank} become important when we consider finite-precision computations in \cref{sec:finite_precision}. 
For now, it is just a black-box that computes the exact rank of its input.

\FloatBarrier

\begin{algorithm}[htb]
\small \setstretch{1.2}
\caption{ : \code{cqrrpt\_core} }
\label{alg:CQRRPT}
\begin{algorithmic}[1]
\Statex \textbf{Input:} A matrix $\mtx{M} \in \mathbb{R}^{m \times n}$, and a sketching operator $\mtx{S} \in \R^{d \times m}$ where $n \leq d \leq m$.
\setstretch{1.3}
\State \textbf{function} $\code{cqrrpt\_core}(\mtx{M},\mtx{S})$
\vspace{-4pt}
\Indent
    \State Sketch $\sk{\mtx{M}} = \mtx{S}\mtx{M}$\label{line:form_sk}
    \State Decompose $[\sk{\mtx{Q}}, \sk{\mtx{R}}, J] = \code{qrcp}(\sk{\mtx{M}})$ \label{line:qrcp_sk} 
    \Statex ~~~ \codecomment{\textasciicircum{} we return this $n$-vector $J$ in full, regardless of subsequent steps.}
    \State Determine $k = \code{rank}(\sk{\mtx{R}})$  \label{line:k_def}
    \State Precondition $\pre{\mtx{M}} = \mtx{M}_k(\sk{\Roo}_k)^{-1}$ \label{line:M_pre_def}
    \Statex ~~~ \codecomment{\textasciicircum{} recall our notation that $\mtx{M}_k := \mtx{M}[\fslice{},J[\lslice{k}]]$ and $\sk{\Roo}_{k} := \sk{\mtx{R}}[\lslice{k},\lslice{k}]$}
    \State Compute $\mtx{G} = (\pre{\mtx{M}})^{\trans}(\pre{\mtx{M}})$ \label{line:get_gram}
    \State Decompose $\pre{\mtx{R}} = \code{chol}(\mtx{G})$ \label{line:get_R_pre}
    \State Set $\mtx{Q}_k = (\pre{\mtx{M}})(\pre{\mtx{R}})^{-1}$ \label{line:get_Q_k}
    \State Undo preconditioning $\mtx{R}_k = \pre{\mtx{R}}\sk{\mtx{R}}_k$\label{line:r_def}
    \Statex ~~~ \codecomment{\textasciicircum{} recall our notation that $\sk{\mtx{R}}_k := \sk{\mtx{R}}[\lslice{k},\fslice{}]$.}
    \State \textbf{return} $\mtx{Q}_k$, $\mtx{R}_k$, $J$
\EndIndent    
\end{algorithmic}
\end{algorithm}

\FloatBarrier

There are two pressing questions for \cref{alg:CQRRPT}.
\begin{enumerate}
    \item Under what conditions does it actually obtain a decomposition of $\mtx{M}$?
    \item Just how ``safe'' is its use of CholeskyQR on Lines \ref{line:get_gram} through \ref{line:get_Q_k}?
\end{enumerate}
\cref{sec:finite_precision} answers these questions in finite-precision arithmetic.
Here are simpler answers under the assumption of computation in exact arithmetic.
When interpreting them, it can be informative to use the fact that $\kappa(\mtx{S}|\range(\mtx{M}))$ is finite if and only if $\rank(\mtx{S}\mtx{M}) = \rank(\mtx{M})$.

\begin{theorem}\label{thm:correctness}
If $\kappa(\mtx{S}|\range(\mtx{M}))$ is finite, then $[\mtx{Q}_k, \mtx{R}_k, J] = \normalfont\code{cqrrpt\_core}(\mtx{M},\mtx{S})$ define a column-pivoted QR decomposition of $\mtx{M}$ in the sense of \cref{def:qrcp_alg}.
\end{theorem}
\begin{proof}
    It is clear that $\mtx{Q}_k$ is orthonormal and $\mtx{R}_k$ is upper-trapezoidal. 
    We need to show that if $\rank(\mtx{S}\mtx{M}) = \rank(\mtx{M})$ then $\mtx{\Delta} = \mtx{M}_n - \mtx{Q}_k\mtx{R}_k$ is zero.
    As a step towards this, note the identity $\mtx{Q}_k\mtx{R}_k = \pre{\mtx{M}}\sk{\mtx{R}}_k$, which implies $\range(\mtx{\Delta}) \subset \range(\mtx{M})$.
    Next, use the assumption that \code{qrcp} produces column-pivoted QR decompositions in sense of \cref{def:qrcp_alg} to find
    \begin{align*}
        \mtx{S}\mtx{\Delta} 
            &= \mtx{S}\mtx{M}_n - \mtx{S}\pre{\mtx{M}}\sk{\mtx{R}}_k \\
            &= \sk{\mtx{M}}_n - \sk{\mtx{Q}}_k\sk{\mtx{R}}_k = \mtx{0}.
    \end{align*}
    Since $\rank(\mtx{S}\mtx{M}) = \rank(\mtx{M})$ implies $\ker(\mtx{S})\cap\range(\mtx{M})$ is trivial, we have $\mtx{\Delta} = \mtx{0}$.
\end{proof}

\begin{theorem}\label{thm:cond_of_A_pre}
    Let $\pre{\mtx{M}}$ be as on Line 5 of {\normalfont\cref{alg:CQRRPT}} for inputs $\mtx{M}$ and $\mtx{S}$.
    If $\kappa(\mtx{S}|\range(\mtx{M}))$ is finite, then the singular values of $\pre{\mtx{M}}$ are the inverses of the restricted singular values of $\mtx{S}$ on $\range(\mtx{M})$.
\end{theorem}

\begin{proof}
    Let $\mtx{U}$ be an $m \times k$ orthonormal matrix with the same range as $\mtx{M}$. 
    Define the set 
    $\mathcal{U}_i  = \{ L \subset \R^k \,:\,  L \text{ is a linear subspace and } \dim(L) = i \}$.
    
    Since the singular values of a matrix are the square roots of the eigenvalues of its Gram matrix, and since the square root is monotonic, the classic min-max principle for eigenvalues of Hermitian matrices directly translates to singular values of general matrices.
    This formulation of the min-max principle tells us that
   \[
        \sigmaSub{i}{\pre{\mtx{M}}} = \min_{X \in \mathcal{U}_i} \max_{\vct{x} \in X} \frac{\|\pre{\mtx{M}}\vct{x}\|_2 }{\|\vct{x}\|_2}~~\text{ and }~~\sigmaSub{k-i+1}{\mtx{S}\mtx{U}} = \max_{Y \in \mathcal{U}_i} \min_{\vct{y} \in Y} \frac{\|\mtx{S}\mtx{U}\vct{y}\|_2}{\|\vct{y}\|_2 }.
    \]
    Our goal is to show that $\sigmaSub{i}{\pre{\mtx{M}}} = \left(\sigmaSub{k-i+1}{\mtx{S}\mtx{U}}\right)^{-1}$.
    To do this, we assume the restricted singular values of $\mtx{S}$ on $\range(\mtx{M})$ are all nonzero.
    This ensures that $\range(\pre{\mtx{M}}) = \range(\mtx{M})$ and subsequently that there is an invertible matrix $\mtx{T}$ where $\pre{\mtx{M}} = \mtx{U}\mtx{T}$.
    We also rely on the fact that the $d \times k$ matrix $\mtx{S} \pre{\mtx{M}}$ is orthonormal.
    Combining these observations gives a chain of identities
    \begin{align*}
    \sigmaSub{i}{\pre{\mtx{M}}} &= \min_{X \in \mathcal{U}_i} \max_{\vct{x} \in X} \frac{\|\pre{\mtx{M}}\vct{x}\|_2 }{\|\vct{x}\|_2}
     = 
     \min_{X \in \mathcal{U}_i} \max_{\vct{x} \in X} \frac{\|\pre{\mtx{M}}\vct{x}\|_2 }{\|\mtx{S}\pre{\mtx{M}}\vct{x}\|_2}
     =
     \min_{X \in \mathcal{U}_i} \max_{\vct{x} \in X} \frac{\|\mtx{U}\mtx{T}\vct{x}\|_2 }{\|\mtx{S}\mtx{U}\mtx{T}\vct{x}\|_2}.
     \end{align*}
     Then we apply a change of variables $\vct{y} = \mtx{T}\vct{x}$ to get
     \begin{align*}
        \sigmaSub{i}{\pre{\mtx{M}}} &= \min_{X \in \mathcal{U}_i} \max_{\vct{x} \in X} \frac{\|\mtx{T}\vct{x}\|_2 }{\|\mtx{S}\mtx{U}\mtx{T}\vct{x}\|_2} = \min_{Y \in \mathcal{U}_i} \max_{\vct{y} \in Y} \frac{\|\vct{y}\|_2 }{\|\mtx{S}\mtx{U}\vct{y}\|_2} = \left (\max_{Y \in \mathcal{U}_i} \min_{\vct{y} \in Y} \frac{\|\mtx{S}\mtx{U}\vct{y}\|_2}{\|\vct{y}\|_2 } \right)^{-1}
    \end{align*}
    which completes the proof.
\end{proof}

\paragraph{Arithmetic complexity.}
The arithmetic complexity of \cref{alg:CQRRPT} depends on how we compute and then decompose $\sk{\mtx{M}}$.
Since there are many practical ways to handle the first of these operations, we shall simply say its cost in flops is $C_{\sktext}$.
There are also many ways one might perform the second operation, but the most practical is to use the \LAPACK function \code{GEQP3}.
With this choice, the algorithm's flop count is
\begin{equation}\label{eq:cqrrpt_flop_count}
    2mk^{2} + mk(k + 1) + 4dnk - 2k^{2}(d + n) + 5k^{3}/3 + C_{\sktext}
\end{equation}
plus lower-order terms; see \cref{app: flop} for a derivation of this fact.

If we plug $k = n$ into \eqref{eq:cqrrpt_flop_count} and assume $d,n \in o(m)$, we see that the leading-order term in CQRRPT's flop count is $3mn^2 + C_{\sktext}$.
This compares favorably to the $4mn^2$ flops required to compute a Householder QR decomposition and then explicitly restore the Q-factor with \LAPACK's \code{ORGQR}.
What's more, in certain applications it can suffice to represent $\mtx{Q}$ as a composition of two elementary operators, $\pre{\mtx{M}}$ and $(\pre{\mtx{R}})^{-1}$.
This comes with no sacrifices to the numerical stability (as will be clear from our stability analysis) yet it decreases the leading term in CQRRPT's arithmetic complexity to $2mn^2 + C_{\sktext}$.
 
\paragraph{Communication cost.}
CQRRPT lends itself well to distributed computation.
Here it is prudent to choose the sketching distribution that allows for evaluation of Line \ref{line:form_sk} with the smallest possible value of $d$ while retaining good statistical properties. This can be achieved in theory and practice with Gaussian sketching (see \cref{subsubsec:oblivious_subspace_embeddings}) and  $\gamma = d/n \in [1.25, 5]$; see \cite[\S A.1.1]{RandLAPACK_Book}.

Consider a popular setting where $\mtx{M}$ is distributed block row-wise across $p \leq m / d$ processors, each of which has at least $2dn + n^2$ words of memory.
Here we implement Line \ref{line:form_sk} with local multiplications of blocks of $\mtx{M}$ with the corresponding blocks of columns of $\mtx{S}$ (using $dn$ words of memory) and summing the contributions with an all-reduce operation (using another $dn$ words of memory).
From there, each processor performs QRCP on $\sk{\mtx{M}}$, recovers $\sk{\mtx{A}}_k$ from $\sk{\mtx{R}}$, and applies its inverse to the corresponding local block of $\mtx{M}$.
When using the classical binomial tree version of all-reduce, the overall computation of $\pre{\mtx{M}}$ requires one global synchronization, $2 \log_2 p$ messages, $2dn \log_2 p$ communication volume, and two data passes.
Adding the cost of the CholeskyQR step~\cite{nguyen2015reproducible},
we get the total cost of CQRRPT: two global synchronizations, $4 \log_2 p$ messages, $(2dn+n^2) \log_2 p$ communication volume, and three data passes.
Notably, if we used the recursive-halving version of all-reduce~\cite{thakur2005optimization}, the communication volume can be improved to a mere $2dn+n^2$.

\subsection{How \code{cqrrpt\_core} inherits rank-revealing properties}
\label{subsec:rank_revealing}

Here we explain how CQRRPT inherits pivot quality properties from its underlying \code{qrcp} function.
To describe these properties, we speak in terms of a matrix $\mtx{X}$ consisting of $n$ columns and at least as many rows,\footnote{Which we might take as $\mtx{X} = \mtx{M}$ or $\mtx{X} = \sk{\mtx{M}}$, depending on context.} along with its decomposition $[\mtx{Q},\mtx{R},J] = \code{qrcp}(\mtx{X})$.
We set $k := \rank(\mtx{X})$, and for any $\ell \leq k$ we use  $(\Roo_{\ell},\Rot_{\ell},\Rtt_{\ell})$ to denote submatrices of $\mtx{R}$ using the conventions established in \cref{subsec:def_and_notations}.

We first consider the \textit{rank revealing QR} (RRQR) property.
This concerns how well the spectrum of $\Roo_{\ell}$ approximates the leading singular values of $\mtx{X}$, and how well the spectrum of $\Rtt_{\ell}$ approximates the trailing singular values of $\mtx{X}$.
When $\mtx{X}$ is fixed, we can describe the approximation quality by a sequence of coefficients $f_1,\ldots,f_k$, all at least unity.
Formally, \code{qrcp} has the \emph{RRQR property} for $\mtx{X}$ with coefficients $(f_\ell)_{\ell=1}^{k}$ 
if, for all $\ell \leq k$, we have
\begin{subequations}
\begin{equation}
\label{eq:RRQR_R11}
        \sigmaSub{j}{\Roo_{\ell}} \geq \frac{\sigmaSub{j}{\mtx{X}}}{f_{\ell}} \quad\text{ for all }\quad j \leq \ell,
\end{equation}
and
\begin{equation}
\label{eq:RRQR_R22}
     \sigmaSub{j}{\Rtt_{\ell}} \leq  f_{\ell} \, \sigmaSub{\ell+j}{\mtx{X}} \quad\text{ for all }\quad j \leq k - \ell.
\end{equation}
\end{subequations}

\begin{theorem}
\label{thm:cqrrpt_RRQR}
        Let $c = \kappa(\mtx{S}|\range(\mtx{M}))$. If {\normalfont\code{qrcp}} satisfies the RRQR property for $\mtx{S}\mtx{M}$ with coefficients $(f_\ell)_{\ell=1}^k$, then {\normalfont$\code{cqrrpt\_core}(\cdot,\mtx{S})$} satisfies the RRQR property for $\mtx{M}$ with coefficients $(c f_{\ell})_{\ell=1}^k$.
\end{theorem}

It is common to ask that a QRCP algorithm admit a \textit{function} $f$ where it ensures the RRQR property with coefficients $(f(\ell,n))_{\ell=1}^k$ for any rank-$k$ matrix with $n$ columns.
This is a significant request.
Indeed, QRCP with the max-norm pivot rule does not satisfy the resulting requirements for \textit{any} function $f$, 
as can be seen by taking a limit of Kahan matrices of fixed dimension \cite[Example 1]{GE:1996}.
Still, there are several QRCP algorithms that can ensure the RRQR property, particularly where $f(\ell,n)$ is bounded by a low-degree polynomial in $\ell$ and $n$.
\cref{thm:cqrrpt_RRQR} shows that if CQRRPT uses such an algorithm, then it will satisfy a nearly identical RRQR property.

Next, we consider a \textit{strong RRQR} property.
This concerns our ability to use $\mtx{R}$ to find a well-conditioned basis for an approximate null space of $\mtx{X}_n$.
The particular basis is the columns of the block matrix $\mtx{Y} = [\Roo_{\ell}^{-1}\Rot_{\ell}; -\mtx{I}]$, which satisfies $\|\mtx{X}_n\mtx{Y}\| = \|\Rtt_{\ell}\|$ in every unitarily invariant norm.
For our purposes, we say that \code{qrcp} satisfies the \emph{strong RRQR property} for $\mtx{X}$ with coefficients $(f_{\ell})_{\ell=1}^k$ and $(g_{\ell})_{\ell=1}^k$ when the former coefficients satisfy \eqref{eq:RRQR_R11}-\eqref{eq:RRQR_R22} and the latter coefficients satisfy
\begin{equation}\label{eq:RRQR_R12}
    \| \Roo_{\ell}^{-1}\Rot_{\ell}\|_2 \leq g_{\ell}
\end{equation}
for all $\ell \leq k$.

\begin{theorem}\label{thm:cqrrpt_SRRQR}
    Let $c = \kappa(\mtx{S}|\range(\mtx{M}))$. If {\normalfont\code{qrcp}} satisfies the strong RRQR property for $\mtx{S}\mtx{M}$ with coefficients $(f_{\ell})_{\ell=1}^k$ and $(g_{\ell})_{\ell=1}^k$, then {\normalfont$\code{cqrrpt\_core}(\cdot,\mtx{S})$} satisfies the strong RRQR property for $\mtx{M}$ with coefficients $(c f_{\ell})_{\ell=1}^k$ and $(g_{\ell} + c f_{\ell}^2)_{\ell=1}^k$.
\end{theorem}

As with RRQR, we can ask that a QRCP algorithm be associated with \textit{functions} $f$ and $g$ for which $(f(\ell,n))_{\ell=1}^k$ and $(g(\ell,n))_{\ell=1}^k$ provide strong RRQR coefficients for any rank-$k$ matrix with $n$ columns.
The first algorithm that could ensure this was introduced by Gu and Eisenstat \cite{GE:1996}.\footnote{The Gu-Eisenstat algorithm actually bounds $(\Roo_{\ell})^{-1}\Rot_{\ell}$ \textit{elementwise}. The elementwise bounds readily imply the spectral-norm bounds that we require.}
Setting the tuning parameter of their algorithm to two leads to strong RRQR coefficients
\[
f(\ell,n) = \sqrt{1 + 4 \ell(n-\ell)} \quad\text{and}\quad g(\ell,n) = 2 \sqrt{\ell (n-\ell)},
\]
with a runtime of $\mathcal{O}(d n^2\log n)$ for a $d \times n$ matrix with $d \geq n$.
\cref{thm:cqrrpt_SRRQR} shows that if CQRRPT uses this algorithm for QRCP on $\mtx{S}\mtx{M}$, then it will enjoy an analogous strong RRQR property with some increase to $g$.

\subsection{Probabilistic aspects of CQRRPT}\label{subsec:sketching_in_CQRRPT}

From the perspective of \cref{alg:CQRRPT_full}, the sketching operator $\mtx{S}$ is sampled at random from a probability distribution, and so anything that $\mtx{S}$ affects in \cref{alg:CQRRPT} becomes a random variable with some induced distribution.
In particular, the probabilistic behavior of \cref{alg:CQRRPT_full} is determined by the induced distribution for the condition number of $\pre{\mtx{M}}$.
Therefore to understand CQRRPT's behavior, we must explore the following question.
\begin{quote}
    How can we bound the probability that $\cond(\pre{\mtx{M}})$ stays within a prescribed limit, \textit{no matter the matrix $\mtx{M}$}?
\end{quote}
The answer to this question depends greatly on the distribution family $\mathscr{F}$ and the sampling factor $\gamma$ used in \cref{alg:CQRRPT_full}.
Here we give a range of answers based on different choices for these values.

\subsubsection{Sketching distribution families: examples and intuition}

For a fixed distribution family $\mathscr{F}$, we use ``$\mathscr{F}_{d,m}$'' for the distribution in $\mathscr{F}$ over $d \times m$ matrices.
Here are the structures of samples from $\mathscr{F}_{d,m}$ for three prominent distribution families.

\begin{itemize}
    \item \textit{Gaussian matrices}. The entries of $\mtx{S}$ are iid Gaussian random variables with mean zero and variance $1/d$.
    \item \textit{SASOs} (short-axis-sparse operators).
    The columns of $\mtx{S}$ are independent.
    Each column has exactly $s$ nonzeros (for a tuning parameter $s$) whose locations are chosen uniformly at random and whose values are chosen to be $\pm 1/\sqrt{s}$ with equal probability.\footnote{For the etymology of these sketching operators, see~\cite{RandLAPACK_Book}.}
    In practice, it is common to keep $\ell$ between one and eight, even when $d$ is on the order of tens of thousands.
    \item \textit{SRHTs} (subsampled randomized Hadamard transforms). These are ordinarily only defined when $w = \log_2 m$ is an integer and are extended to general $m$ by zero-padding the input.
    To absorb the zero-padding into the SRHT, define a $2^{\lceil w\rceil} \times m$ matrix $\mtx{D}$ whose upper $m \times m$ block is a diagonal matrix populated by independent Rademacher random variables and whose lower block is all zeros.
    An SRHT is then a composition of three operators:
    $\mtx{S} = \sqrt{m / d}\cdot \mtx{\Pi}[\lslice{d},\fslice{}]\mtx{H}\mtx{D}$, 
    where $\mtx{H}$ is a Hadamard transform of order $2^{\lceil w \rceil}$ and $\mtx{\Pi}$ is a random permutation matrix.
\end{itemize}

\noindent These families share key properties.
First, if $\mtx{S}$ is a random matrix sampled from $\mathscr{F}_{d,m}$, then its expected covariance matrix satisfies $\E[\mtx{S}^{\trans}\mtx{S}] = \mtx{I}_m$.
Second, if $\mtx{U}$ is an $m \times k$ orthonormal matrix and $d / k$ is sufficiently large, then the sketched Gram matrix $ \mtx{U}^{\trans}\mtx{S}^{\trans}\mtx{S}\mtx{U}$ concentrates strongly around $\E[\mtx{U}^{\trans}\mtx{S}^{\trans}\mtx{S}\mtx{U}] = \mtx{I}_k$.
To rephrase this second property: $\mtx{S}\mtx{U}$ should concentrate strongly around the real Stiefel manifold $M_{d, k}$ of $d \times k$ orthonormal matrices.

To see the usefulness of these properties, consider \cref{thm:cond_of_A_pre}.
If $\range(\mtx{U}) = \range(\mtx{M})$ and $\mtx{S}\mtx{U}$ has rank-$k$, then the singular values of $\pre{\mtx{M}}$ will equal those of the Moore-Penrose pseudo-inverse of $\mtx{S}\mtx{U}$.
Therefore the desirable property of $\pre{\mtx{M}}$ concentrating
around $M_{m, k}$ is equivalent to $\mtx{S}\mtx{U}$ concentrating around $M_{d, k}$.

\subsubsection{Oblivious subspace embeddings}
\label{subsubsec:oblivious_subspace_embeddings}

Let $\mtx{U}$ be a matrix whose columns are an orthonormal basis for a linear subspace $L \subset \R^m$, and let $\mtx{S}$ be a $d \times m$ matrix.

\begin{definition}\label{def:subspace_embedding}
We call $\mtx{S}$ a \emph{subspace embedding for $L$} with \emph{distortion} $\distortion \in [0, 1]$ if
\begin{equation}\label{eq:subspace_embedding_condition}
   1-\distortion \leq \sigmaMin{\mtx{S}\mtx{U}}^2 \quad\text{and}\quad \sigmaMax{\mtx{S}\mtx{U}}^2 \leq 1 + \distortion.
\end{equation}
Such a matrix is also called a \emph{$\distortion$-embedding} for $L$.
\end{definition}

\noindent
Note that we always have $\kappa(\mtx{S}|L) \leq \sqrt{(1+\distortion)/(1-\distortion)}$ when $\mtx{S}$ is a $\distortion$-embedding for $L$.
There is a long history of using subspace embeddings in randomized algorithms for least squares problems \cite{Sarlos:2006,DMM06,DMMS07_FastL2_NM10}, including as a tool for finding preconditioners to solve least squares problems to higher accuracy \cite{RT:2008:SAP,AMT:2010:Blendenpik,MSM:2014:LSRN}.

Theory is available on how to choose $d$ so that a sample from
$\mathscr{F}_{d, m}$ will be a $\distortion$-embedding for any fixed linear subspace of a given dimension with high probability.
We can use this theory to select $d$ in practice for very well-behaved distribution families.
For example, here is a representative result for Gaussians.

\begin{remark}
    We state the following results with linear subspaces of dimension ``$n$,'' since the dimension of $\range(\mtx{M})$ is never larger than $n$.
\end{remark}

\begin{theorem}
\label{thrm:gaussian_embedding}
    Fix an $n$-dimensional linear subspace $L \subset \R^m$, along with some $\distortion \in (0,1)$ and $\tau > 0$. If $\mtx{S}$ is a $d \times m$ Gaussian operator with
    \[
        \frac{d}{n} \geq \left(\frac{1 + \tau}{(1+\delta)^{1/2} - 1}\right)^2,
    \]
    then it will be a $\distortion$-embedding for $L$ with probability at least $1 - 2 \exp(-n\tau^2/2)$.
\end{theorem}
\begin{proof}
By rotational invariance of the Gaussian distribution, we can take $L = \range(\mtx{U})$ for $\mtx{U} = \mtx{I}_m[\fslice{},\lslice{n}]$.
The claim then follows from \cite[Theorem 8.4]{MT:2020}.
To see how, set $\theta = (1+\delta)^{1/2} - 1$.
If $d$ is chosen in the way indicated above, then setting $t := \tau \theta/ (1+\tau)$ in the statement of \cite[Theorem 8.4]{MT:2020} ensures that $1-\theta \leq \sigma_{\min}(\mtx{S}\mtx{U})$ and $\sigma_{\max}(\mtx{S}\mtx{U}) \leq 1+\theta$ hold with probability at least $1-2\exp(-n \tau^2 /2)$.
From there, simply note that $1-\delta \leq (1-\theta)^2$ and $(1+\theta)^2 = 1+\delta$ to find that $1 -\delta \leq \sigma_{i}(\mtx{S}\mtx{U})^2 \leq 1+\delta$ for all $i$.
\end{proof}
We give results for SASOs and SRHTs below.
These are often of interest since they are algorithmically more
attractive than Gaussian matrices.
However, the available bounds are quite pessimistic (they suggest that one take $\gamma$ proportional to $\log n$, but taking $\gamma = 1.25$ suffices for practical purposes when $n$ is large).

\begin{theorem}\cite[Theorem 4.2]{Cohen:2016:SJLTs}
\label{thrm:saso_embedding}
    Fix an $n$-dimensional linear subspace $L \subset \R^m$, and any $B > 2$, $t \geq 1$, and $\distortion < 1/2$.  
    There are absolute constants $c_1, c_2$ where, upon taking
    \[
    \frac{d}{n} \geq c_1 t (B \log B) \frac{\log n}{\distortion^2} \quad \text{ and }\quad s \geq c_2 t (\log B)\frac{\log n}{\distortion},
    \]
    sampling a $d \times m$ SASO with $s$ nonzeros per column provides a $\distortion$-embedding for $L$ with probability at least $1 - B^{-t}$.
\end{theorem}

\begin{theorem}\cite[Proposition 3.9]{balabanov2019randomized}\label{thrm:srht_embedding}
    Let $L \subset \R^m$ be a linear subspace of dimension $n$. Let $\distortion, p \in (0,1)$.
    If $\mtx{S}$ is a $d \times m$ SRHT with $d \leq m$ and 
    \[
        \frac{d}{n} \geq 2(\delta^2 - \delta^3/3)^{-1} \left(1 + \sqrt{\frac{8 \log( 6m/p)}{n}}\right)^2 \log (3n/p),
    \]
    then it is a $\distortion$-embedding for $L$ with probability at least $1 - p$.
\end{theorem}
We note that the analysis in \cite[Proposition 3.9]{balabanov2019randomized}, cited above, builds on more fundamental results from~\cite{Tropp:2011,BG:2013}.

\section{Numerical stability}
\label{sec:finite_precision}

This section analyzes \code{cqrrpt\_core} (\cref{alg:CQRRPT}) under the assumption that all algebraic operations in it and in its subroutines are performed in finite precision arithmetic.
Suppose for concreteness that we have
\begin{equation}\label{eq:def_Qk_Rk_J_numerical}
     [\mtx{Q}_k,\mtx{R}_k,J] = \code{cqrrpt\_core}(\mtx{M},\mtx{S})  ,
\end{equation}
where $\mtx{Q}_k$ is $m \times k$ and $\mtx{R}_k$ is $k \times n$ and upper-triangular.
In these terms, our analysis is concerned with bounding the \textit{reconstruction error} $\|\mtx{M}_n - \mtx{Q}_k \mtx{R}_k\|_{\mathrm{F}} / \|\mtx{M}\|_{\mathrm{F}}$ and the \textit{orthogonality loss} $\|\mtx{Q}_k^{\trans}\mtx{Q}_k - \mtx{I}_k \|_2$.

Proving bounds on these quantities requires assumptions on the
\code{qrcp} and \code{rank} subroutines in \cref{alg:CQRRPT}.
So far we have said very little about these points because they are ultimately design questions that have no single answer.
In this section, our goal is to show that \code{qrcp} and \code{rank} can be implemented so that upon conditioning on $\mtx{S}$ being a $\distortion$-embedding for $\range(\mtx{M})$ (for some $\distortion \leq 1/2$), the reconstruction error and orthogonality loss
are bounded above by low-degree polynomials in $(m, n)$.
Such a result will directly imply that CQRRPT's properties introduced in \cref{sec:introduce_CQRRPT} and proven in  \cref{sec:analysis_exact} are preserved under finite precision arithmetic. 

\paragraph{New notation.}
\phantomsection{}
\label{page:numerical_stability_notation}

For purposes of exposition in this section only, we adopt notation where scalars $x$ and $y$ are said
to satisfy $x \lesssim y$ if $x \leq c  y + G(n) \roundoff$, where $c$
is a constant close to $1$, $G(n)$ is a low-degree in the small matrix dimension $n$,
and $\roundoff$ is a unit roundoff.\footnote{ The polynomials we use will never exceed degree four, and are typically degree one or two.}

Up until now the $m \times k$ and $k \times k$ matrices $\pre{\mtx{M}}$ and $\pre{\mtx{R}}$ have been referred to without notational dependence on $k$.
Moving forward, we use $\pre{\mtx{M}}_k$ and $\pre{\mtx{R}}_k$ for these matrices.
We can select submatrices from them using notation consistent with \cref{subsec:def_and_notations}.
For example, we can speak of a parameter $\ell < k$ and use $\pre{\mtx{M}}_{\ell}$ in reference to the first $\ell$ columns of $\pre{\mtx{M}}_k$.
Similarly, we can use $\pre{\Roo}_\ell$ to denote the leading $\ell \times \ell$ submatrix of $\pre{\mtx{R}}_k$.

\subsection{Summary}
\subsubsection{The challenge}

We begin by emphasizing that since we want to decide how \code{rank} should be implemented, the value $k = \rank(\sk{\mtx{M}})$ on Line \ref{line:k_def}
of \code{cqrrpt\_core} is really something we \textit{choose}.
   
That established, let us see how $k$ affects reconstruction error and orthogonality loss.
In the tradition of adding zero and applying the triangle inequality, one can obtain the following bound on reconstruction error committed in the preconditioning step:
\begin{align}
    \small \norm{\mtx{M}_n- \pre{\mtx{M}}\sk{\mtx{R}}_k }_\mathrm{F} \leq  \underbrace{\norm{\mtx{M}_n - \mtx{M}_k(\sk{\Roo}_k)^{-1}\sk{\mtx{R}}_k}_\mathrm{F}}_{\text{truncation error}} \, +\, \norm{(\mtx{M}_k - \pre{\mtx{M}}\sk{\Roo}_k) (\sk{\Roo}_k)^{-1} \sk{\mtx{R}}_k }_\mathrm{F}    .
    \label{eq:reconstruction_error_bound}
\end{align}
Of the two terms in this upper bound, the \textit{truncation error} is more opaque.
To better understand it, suppose $\distortion$ is the distortion of $\mtx{S}$ for the range of $\mtx{M}$.
We claim that under mild assumptions for the accuracy of operations on Lines \ref{line:form_sk} and \ref{line:qrcp_sk}, and an assumption that $\sk{\mtx{A}}_k$ is not too ill-conditioned, one can bound
\begin{equation}\label{eq:trunc_error_bound}
   \small {\norm{ \mtx{M}_n - \mtx{M}_k(\sk{\Roo}_k)^{-1}\sk{\mtx{R}}_k}_\mathrm{F}}/{\norm{\mtx{M}}_\mathrm{F}} \lesssim \sqrt{1+\distortion}\norm{\sk{\mtx{C}}_k}_\mathrm{F}/{\norm{\sk{\mtx{M}}}_\mathrm{F}}.
\end{equation}

Since $\|\sk{\mtx{C}}_k\|_\mathrm{F}$ is decreasing with $k$, this bound suggests that $k$ should be large to keep reconstruction error under control.

Unfortunately, if $\mtx{M}$ is ill-conditioned, then choosing $k$ too large can pose severe problems.
The basic reason for this stems from the fact that $\cond(\sk{\Roo}_n)$ can be as large as $\cond(\mtx{S}{\mtx{M}})$, which in turn can be as large as $\sqrt{\frac{1+\distortion}{1-\distortion}}\cond(\mtx{M})$.
This suggests that $\sk{\Roo}_k$ can be ill-conditioned when $k$ is large, even when $\distortion$ is small.
This matters since larger condition numbers for $\sk{\Roo}_k$ risk larger rounding errors at the step when we form the $m \times k$ matrix $\pre{\mtx{M}}$.
If this step is performed inaccurately, then $\cond(\pre{\mtx{M}})$ might be far from one even if $\distortion$ is small, which risks orthogonality loss in $\mtx{Q}_k$ that cannot be controlled by our choice of sketching distribution for \cref{alg:CQRRPT_full}.

\subsubsection{A high-level result}

The tension described above raises an important question.
Can we be certain that there \textit{exists} a truncation index $k$ so that both reconstruction error and orthogonality loss are kept near machine precision?
The following result says the answer is yes, provided the QRCP of $\sk{\mtx{M}}$ is (sufficiently) strongly rank-revealing.

\begin{theorem}[Simplified version of~\cref{thm:stabprecond} from~\cref{app:numerical_stability_proofs}]\label{cor:simple_stability}
	Consider~\cref{alg:CQRRPT} where Line \ref{line:M_pre_def} is executed with unit roundoff $\roundoff$, and other lines are executed with unit roundoff $\tilde{\roundoff}$. 
	Assume that the $\normalfont\code{qrcp}$ subroutine at Line 3 produces $\sk{\mtx{Q}}$ and $\sk{\mtx{R}}$ by a pivoted Householder QR process or pivoted Givens QR process.
    Additionally, assume its pivots satisfy the strong rank-revealing properties \eqref{eq:RRQR_R11}, \eqref{eq:RRQR_R22}, and \eqref{eq:RRQR_R12} with coefficients $f_\ell, g_\ell \leq 2 \sqrt{n \ell}$ when $\mtx{X} = \sk{\mtx{M}}$.
	Finally, suppose $\mtx{S}$ is a $\distortion$-embedding for $\range(\mtx{M})$ with $\distortion \leq 1/2$. 
	
	There exist low-degree polynomials $G_1,\ldots,G_5$ (that have no dependence on $\mtx{M}$ or $\mtx{S}$) and a truncation index $k$ such that if $\tilde{\roundoff} \leq  m^{-1} G_1(n,d)^{-1} \roundoff $ and $\roundoff \leq G_2(n)^{-1}$, then
	\begin{equation} \label{eq:simple_stability1}
	\small \norm{\mtx{M}_n - \pre{\fmtx{M}} \sk{\fmtx{R}}_k}_\mathrm{F} \leq G_3(n) \roundoff \norm{\mtx{M} }_\mathrm{F}
	\text{~~~and~~~}  \cond{(\pre{\fmtx{M}})} \leq 1.8.
	\end{equation}		
	Furthermore, as a direct consequence of \cref{eq:simple_stability1}, we have
	\begin{equation} \label{eq:simple_stability2}
	\small \norm{\mtx{M}_n - \fmtx{Q}_k \fmtx{R}_k}_\mathrm{F} \leq   G_4(n) \roundoff \norm{\mtx{M} }_\mathrm{F}
	\text{~~~and~~~}	 
	\norm{\fmtx{Q}_k^{\trans}\fmtx{Q}_k - \mtx{I}}_2 \leq  G_5(n) \roundoff.
	\end{equation}
\end{theorem}

We have chosen a computational model with two unit roundoffs to highlight an important property of the preconditioner -- that the dominant Line \ref{line:M_pre_def} can be performed with a unit roundoff not constrained by $m$.
This property may have significant implications for multi- and low-precision arithmetic architectures.
Clearly, the result can also be used in the computational model with a single roundoff.

We note that \cref{cor:simple_stability} is an \textit{existential} result for a suitable truncation rank $k$.
Our proof of this result includes a method for how such $k$ can be identified efficiently, under the stated assumptions on the unit roundoffs.
However, in practice, the unit roundoff is usually a constant that does not change with the matrix dimensions.
This creates a need for a separate and more practical way to choose the truncation rank, which is a topic we address in \cref{subsec:practical_rank_est}.

\subsection{Proof sketch for  \cref{cor:simple_stability}}\label{subsec:stability_cor_proof_overview}

The following lemma, which we prove in \cref{subapp:analysis:stabchol} by straightforward methods, shows that \cref{eq:simple_stability2} in~\cref{cor:simple_stability} directly follows from~\cref{eq:simple_stability1}.

\begin{lemma}\label{thm:stabchol}
	Consider Lines 5 to 7 of~\cref{alg:CQRRPT} where the computations are performed with unit roundoff $\tilde{\roundoff} \leq 0.00022 m^{-1}n^{-1}$. If $\pre{\fmtx{M}}_k$ and $\sk{\trun{\fmtx{R}}}_k$ satisfy 
    \begin{equation} \label{eq:CholQRprop0}
        \small \norm{\mtx{M}_n - \pre{\fmtx{M}}_k \sk{\trun{\fmtx{R}}}_k}_2 \leq 0.01 \norm{\mtx{M}}_2 \text{ and } \cond(\pre{\fmtx{M}}_k) \leq 6,
    \end{equation}
    then the output factors $\fmtx{Q}_k$ and $\fmtx{R}_k$ satisfy
	\begin{flalign}
     &\small \norm{\mtx{M}_n- \fmtx{Q}_k \fmtx{R}_k}_\mathrm{F} \leq \norm{\mtx{M}_n - \pre{\fmtx{M}}_k \sk{\trun{\fmtx{R}}}_k}_\mathrm{F} + 60 n^2 \tilde{\roundoff} \norm{\mtx{M}}_2,   \label{eq:CholQRprop1} \\
 & \small \norm{\fmtx{Q}_k^{\trans}\fmtx{Q}_k - \mtx{I}}_2 \leq 180 m n \tilde{\roundoff}. \label{eq:CholQRprop2}
	\end{flalign}
\end{lemma}

Taking \cref{thm:stabchol} as given, the question becomes how to choose $k$ so that \cref{eq:simple_stability1} holds.
We call this the \textit{preconditioner stability question}.
In brief, the main idea is to choose $k$ so that
\begin{equation}\label{eq:tau0}
\small
{\|\sk{\Rtt}_{k} \|_\mathrm{F}}/ {\|\sk{\mtx{R}} \|_2} \leq G_0(n) \roundoff \leq {\|\sk{\Rtt}_{k-1} \|_\mathrm{F}}/ {\|\sk{\mtx{R}} \|_2},
\end{equation}
for a low-degree polynomial $G_0(n)$ that does not depend on $\mtx{M}$ or $\mtx{S}$.
The specific polynomial $G_0(n)$ sufficient to ensure~\cref{eq:simple_stability1} is given in \cref{thm:stabprecond}, which is stated and proven in  \cref{subsubsec:thm:stabprecond}.
In what follows, we outline the key steps in the argument for proving that theorem, using the same notation as in our formal proofs.

\paragraph{Truncation ensures full numerical rank.}

Choosing $k$ to satisfy \eqref{eq:tau0} ensures that $\sk{\fmtx{A}}_k$ is numerically full-rank.
Assuming $\sk{\fmtx{R}}$ comes from a column-pivoted QR decomposition of $\sk{\mtx{M}}$ with the theorem's stated strong RRQR properties, such $k$ provides for the existence of low-degree polynomials $G_6$ and $G_7$ where
\begin{equation*}
\small \cond(\sk{\fmtx{A}}_k) = \frac{\|\sk{\fmtx{A}}_k \|_2}{\sigmaMin{\sk{\fmtx{A}}_k}} 
\leq \frac{\|\sk{\fmtx{R}} \|_2}{\|\sk{\fmtx{C}}_{k-1}\|_\mathrm{F}} G_6(n) \leq G_7(n) \roundoff{}^{-1}.
\end{equation*}
This bound can be used to show that $\mtx{M}_k$ is also numerically full-rank, in that 
\begin{equation} \label{eq:simple_stability3}
\small \cond( \mtx{M}_k) \leq \sqrt{\frac{1+\distortion}{1-\distortion}} \cond(\mtx{S} \mtx{M}_k) \lesssim \cond(\sk{\fmtx{A}}_k) \leq G_7(n) \roundoff{}^{-1}.
\end{equation}
The first inequality in \eqref{eq:simple_stability3} holds due to the $\distortion$-embedding property of $\mtx{S}$.
The second inequality can be shown through three steps: $\cond(\sk{\fmtx{A}}_k) \gtrsim \cond(\sk{\fmtx{Q}}_k \sk{\fmtx{A}}_k) \gtrsim \cond(\sk{\fmtx{M}}_k) \gtrsim \cond(\mtx{S} \mtx{M}_k)$, employing standard rounding bounds for matrix operations, with the fact that the left-hand-side matrix is numerically full rank (in steps two and three).

\paragraph{Preconditioning and reconstruction error for leading columns.}

The next step in the proof is to look at $\pre{\fmtx{M}}_k \sk{\fmtx{A}}_k$ as an unpivoted ``sketched CholeskyQR'' decomposition of $\mtx{M}_k$, in the sense of \cite{Balabanov:2022:cholQR}.
According to \cite[Theorem 5.2]{Balabanov:2022:cholQR}, if $\cond( \mtx{M}_k) \lesssim  G_7(n) \roundoff{}^{-1}$, then there is a low-degree polynomial $G_8$ for which
\begin{equation} \label{eq:uncholQR}
	\small\|\mtx{M}_k - \pre{\fmtx{M}}_k \sk{\fmtx{A}}_k \|_\mathrm{F} \leq  G_8(n) \roundoff \|\mtx{M}_k \|_\mathrm{F}
	\text{~~~and~~~}
	\cond(\pre{\fmtx{M}}_k) \lesssim \sqrt{{\frac{1+\distortion}{1-\distortion}}}.
\end{equation}
This establishes the condition number bound needed in \eqref{eq:simple_stability1}.
It also controls reconstruction error for the first $k$ columns of $\mtx{M}_n$.

\paragraph{Reconstruction error for trailing columns.}

To establish \eqref{eq:simple_stability1}, it remains to bound the reconstruction error of the trailing $n - k$ columns of $\mtx{M}$.
Toward this end, we adopt the notation that 
\[
    \small \xoverline{\mtx{M}}_k := \mtx{M}[\fslice,J[\tslice{k+1}n]] \quad\text{and}\quad \sk{\xoverline{\mtx{M}}}_k:= \sk{\mtx{M}}[\fslice,J[\tslice{k+1}n]].
\]
By using the reconstruction error bound for $\mtx{M}_k$ from~\cref{eq:uncholQR}, the $\distortion$-embedding property, standard bounds for matrix operations, and the reconstruction error bound for $\sk{\mtx{M}}_k$, one finds that
\begin{equation} \label{eq:simple_stability5}	
\begin{split}
    \small \frac{\|\xoverline{\mtx{M}}_k - \pre{\fmtx{M}}_k \sk{\fmtx{B}}_k \|_\mathrm{F}}{\|\fmtx{M}\|_\mathrm{F}} &\lesssim \frac{\|\xoverline{\mtx{M}}_k - \fmtx{M}_k (\sk{\fmtx{A}}_k)^{-1} \sk{\fmtx{B}}_k \|_\mathrm{F}}{{\|\fmtx{M}\|_\mathrm{F}}} \leq 
    \sqrt{\frac{1+\distortion}{1-\distortion}} \frac{\|\mtx{S} \xoverline{\mtx{M}}_k - \mtx{S} \fmtx{M}_k (\sk{\fmtx{A}}_k)^{-1} \sk{\fmtx{B}}_k \|_\mathrm{F}}{\| \mtx{S} {\fmtx{M}}\|_\mathrm{F}} \\
    &\lesssim \sqrt{\frac{1+\distortion}{1-\distortion}} \frac{\|\sk{\xoverline{\mtx{M}}}_k - \sk{\fmtx{M}}_k (\sk{\fmtx{A}}_k)^{-1} \sk{\fmtx{B}}_k \|_\mathrm{F}}{\|\sk{\fmtx{M}}\|_\mathrm{F}}  \lesssim \sqrt{\frac{1+\distortion}{1-\distortion}} \frac{\|\sk{\xoverline{\mtx{M}}}_k - \sk{\fmtx{Q}}_k \sk{\fmtx{B}}_k \|_\mathrm{F}}{\|\sk{\fmtx{M}}\|_\mathrm{F}}.
\end{split}
\end{equation}
Notably, the strong RRQR property of $\fmtx{R}_k$ plays a key role in~\cref{eq:simple_stability5}.
It ensures that the terms of the form $\|\mtx{E} (\sk{\fmtx{A}}_k)^{-1} \sk{\fmtx{B}}_k\|_\mathrm{F}$, for some error matrices such as $\mtx{E} =\mtx{M}_k - \pre{\fmtx{M}}_k \sk{\fmtx{A}}_k$, are bounded by $g_k\|\mtx{E}\|_\mathrm{F}$, where $g_k \leq 2\sqrt{kn}$.
	
We combine \cref{eq:simple_stability5} with the stability of the \code{qrcp} routine and the criterion \cref{eq:tau0} used for selecting the truncation index $k$ to get
\begin{equation}\label{eq:simple_stability4}
   \small \frac{\|\xoverline{\mtx{M}}_k - \pre{\fmtx{M}}_k \sk{\fmtx{B}}_k \|_\mathrm{F}}{\|{\mtx{M}}\|_\mathrm{F}} \lesssim \sqrt{\frac{1+\distortion}{1-\distortion}} \frac{\|\sk{\xoverline{\mtx{M}}}_k - \sk{\fmtx{Q}}_k \sk{\fmtx{B}}_k \|_\mathrm{F}}{\|\sk{\fmtx{M}}\|_\mathrm{F}} \lesssim \sqrt{\frac{1+\distortion}{1-\distortion}} \frac{\|\sk{\fmtx{C}}_k  \|_\mathrm{F}}{\|\sk{\fmtx{R}}\|_\mathrm{F}} \leq  G_{9}(n) \roundoff.
\end{equation}	
This fulfills our need for a reconstruction error bound for the trailing columns of $\mtx{M}$.

\subsection{Determining numerical rank in practice} 
\label{subsec:practical_rank_est}

We suggest a flexible two-stage approach to numerical rank estimation.
The idea is that since orthogonality loss increases with truncation rank and reconstruction error (typically) decreases with truncation rank, it is reasonable to choose the largest rank where some estimate for the orthogonality loss is below a specified tolerance (say, $\orthtol = 100 \roundoff$).

Our first stage begins by finding a crude upper bound on the numerical rank of $\sk{\mtx{R}}$.
For example, one can set
\begin{subequations}
\begin{equation}\label{eq:initial_numerical_rank}
    k_o = \min\{ \ell \,:\, \|\sk{\mtx{C}}_{\ell}\|_{\mathrm{F}} \leq \roundoff s\} ,
\end{equation}
where $s$ is the maximum entry of $|\sk{\mtx{R}}|$.
This bound is cheap to compute and ensures $\|\sk{\mtx{C}}_{k_o}\|_2 \leq  \roundoff \|\sk{\mtx{R}} \|_2$.
The rest of the first stage consists of forming the preconditioned matrix $\pre{\mtx{M}}_{k_o} = (\mtx{M}_{k_o})(\sk{\mtx{A}}_{k_o})^{-1}$ and computing the Cholesky decomposition of $(\pre{\mtx{M}}_{k_o})^{\trans}(\pre{\mtx{M}}_{k_o})$.

Our second stage uses a function for estimating condition numbers of triangular matrices.
Given such a function, $\code{cond}$, we estimate the orthogonality loss of choosing truncation rank $\ell$ by $\roundoff{}\cdot \left(\code{cond}(\pre{\Roo}_{\ell})\right)^2$.
The motivation for this is that if $\code{cond}$ bounds condition numbers from above, then choosing $\ell$ to satisfy $\code{cond}(\pre{\Roo}_{\ell}) \leq \sqrt{\orthtol /\roundoff}$ bounds orthogonality loss by $\mathcal{O}(\orthtol)$.
Therefore the formal goal of the second stage is to find 
\begin{equation}\label{eq:updated_numerical_rank}
    k = \max\{ \ell \,:\,  \code{cond}(\pre{\Roo}_{\ell}) \leq \sqrt{\orthtol /\roundoff}\, \}.
\end{equation}
When computing \eqref{eq:updated_numerical_rank} we can assume that $\code{cond}(\pre{\Roo}_{\ell+1}) \geq \code{cond}(\pre{\Roo}_{\ell})$.
This assumption holds for the true condition number, as can be seen by applying the eigenvalue interlacing theorem to the Gram matrices of $\pre{\Roo}_{\ell+1}$ and $\pre{\Roo}_{\ell}$ (the latter being a submatrix of the former).
This assumption is useful because it lets us compute \eqref{eq:updated_numerical_rank} by binary search over $\ell$ in $\{1,\ldots,k_o\}$.
Even if this assumption does not hold, the only risk in deciding $k$ by binary search is underestimating numerical rank.
\end{subequations}

One can bound the condition number of a triangular matrix $\mtx{X}$ in $\mathcal{O}(n^2 \log n)$ time by applying Krylov subspace methods to $\|\mtx{X}\|_2$ and $\|\mtx{X}^{-1}\|_2$.
A similar approach with half the complexity could be used to estimate $\tau \geq  \|\mtx{I} - \mtx{X}\|_2$, which could be turned around to bound $\cond(\mtx{X}) \leq (1+\tau)/(1-\tau)$ if $\tau < 1$.
If we relax the requirement that $\code{cond}$ always upper-bounds condition numbers, then we can take $\code{cond}(\mtx{X}) = \cond(\diag(\mtx{X}))$ to estimate $\cond(\mtx{X})$ from below. The last of these three is actually our preferred method, since it works well in practice and is extremely simple to implement.

\begin{remark}[What if Cholesky fails in stage one?]\label{rem:if_cholesky_fails}
    Set $\mtx{G} = (\pre{\mtx{M}}_n)^* (\pre{\mtx{M}}_n)$. Consider an index $k_o$ where $\mtx{G}[\lslice k_o, \lslice k_o]$ is positive definite, but $\mtx{G}[\lslice (k_o+1), \lslice (k_o+1)]$ is not.
    Running \LAPACK's \code{POTRF} function on $\mtx{G}$ will produce an error. 
    However, the leading $k_o$-by-$k_o$ submatrix of the output $\mtx{R}$ factor will be well-formed and satisfy $(\mtx{R}[\lslice k_o, \lslice k_o])^{*}\mtx{R}[\lslice k_o, \lslice k_o] = \mtx{G}[\lslice k_o, \lslice k_o]$.
    Therefore if \code{POTRF} fails at Line~\ref{line:get_R_pre} and returns an error code $k$, then we can take $k_o = k - 1$ as an initial estimate for numerical rank.
\end{remark}

\section{Pivot quality experiments}
\label{sec:empirical_pivots}

This section gives experimental comparisons of pivot quality using the \LAPACK default (\code{GEQP3}) versus those produced by CQRRPT \textit{based on} on that default.
Of course, the results of such a comparison depend heavily on the distribution family $\mathscr{F}$ and the sampling factor $\gamma$ that determines the distribution of our sketching operator.
Therefore we take this as an opportunity to show how one might set $\mathscr{F}$ and $\gamma$ in practice.

\paragraph{Background on leverage scores and coherence}
\cref{subsec:sketching_in_CQRRPT} presented results on how $\gamma$ can be chosen to achieve oblivious subspace embedding properties for various sketching families.
The relevant result for SASOs, \cref{thrm:saso_embedding}, provides worst-case bounds that are valuable in theoretical analysis.
However, some subspaces are ``easier to sketch'' with SASOs than this result would suggest, in the sense that the subspace embedding property can reliably be obtained with a far smaller sampling factor or sparsity parameter.

\textit{Leverage scores} are a useful concept for understanding when a subspace might be easy or hard to accurately sketch with a given distribution.
They quantify the extent to which a low-dimensional subspace aligns with coordinate subspaces~\cite{Mah-mat-rev_BOOK,DM16_CACM,DM21_NoticesAMS}.
\begin{definition}
    Let $\mtx{U}$ be an $m \times n$ orthonormal matrix.
    The \emph{$i^{\text{th}}$ leverage score} of $\range(\mtx{U})$ is the squared row-norm $\|\mtx{U}[i,\fslice{}]\|_2^2$.
\end{definition}
When using sketching operators such as SASOs or SRHTs, it is standard to summarize leverage scores with a concept called \textit{coherence}.
This is essentially a condition number for random sampling algorithms~\cite{Mah-mat-rev_BOOK,DM16_CACM,DM21_NoticesAMS}.
Formally, the coherence of an $m \times n$ matrix $\mtx{M}$ is $m$ times the largest leverage score of $\range(\mtx{M})$.
It is well-documented in the literature that matrices with higher coherence are harder to sketch.

\paragraph{Pivot quality metrics}
We use two pivot quality metrics.
The first is the Frobenius norms of the matrices $\Rtt_{\ell}$ in block 2-by-2 partitions of $\mtx{R}$.
This has the natural interpretation as the norm of a rank-$\ell$ approximation of $\mtx{M}_n - \mtx{Q}_{\ell}\mtx{R}_{\ell}$; we plot this metric as ratios $\|\Rtt_{\ell}^{\text{qp3}}\|_{\mathrm{F}} / \|\Rtt_{\ell}^{\text{ours}}\|_{\mathrm{F}}$.
Our second pivot quality metric is the ratios of $r_{ii} := |\mtx{R}[i,i]|$ to the singular values of $\mtx{M}$.
If $\mtx{R}$ comes from \code{GEQP3} then this ratio can be quite bad in the worst case.
Letting $\sigma_i$ denote the $i^{\text{th}}$ singular value of $\mtx{M}$, this only guarantees that $\phi_i := r_{ii} / \sigma_i$ is between $(n(n + 1)/2)^{-1/2}$ and $2^{n-1}$ \cite{Higham:blog:rrf}.
Since there is a chance for large deviations, we plot $r_{ii} / \sigma_i$ for our algorithm and \code{GEQP3} separately (rather than plotting the ratio $r_{ii}^{\text{qp3}} / r_{ii}^{\text{ours}}$).

\subsection{Example low-coherence matrices}
\label{subsec:low_coherence}
Here we consider tall $m \times n$ matrices with $m = 2^{17} = 131072$ and $n = 2000$ whose spectrum falls into one of the two following categories.
\begin{itemize}
    \item Matrices for which the first ten percent of their singular values are all equal to one and the rest are decaying polynomially down to $1/\cond(\mtx{M}) = 10^{-10}$.
    \item Matrices with a four-step ``staircase-shaped'' spectrum. The first quarter of the singular values are 1, the next quarter are $8 \cdot 10^{-10}$, the quarter after that are $4\cdot 10^{-10}$, and the final quarter are all $10^{-10}$.
\end{itemize}

The singular vectors were generated by orthogonalizing the columns of matrices with iid Gaussian entries.
Generating the left singular vectors in this way is a standard method for generating low-coherence matrices.

\begin{figure}[hbt!]
    \centering
    \begin{overpic}[width=0.8\linewidth]{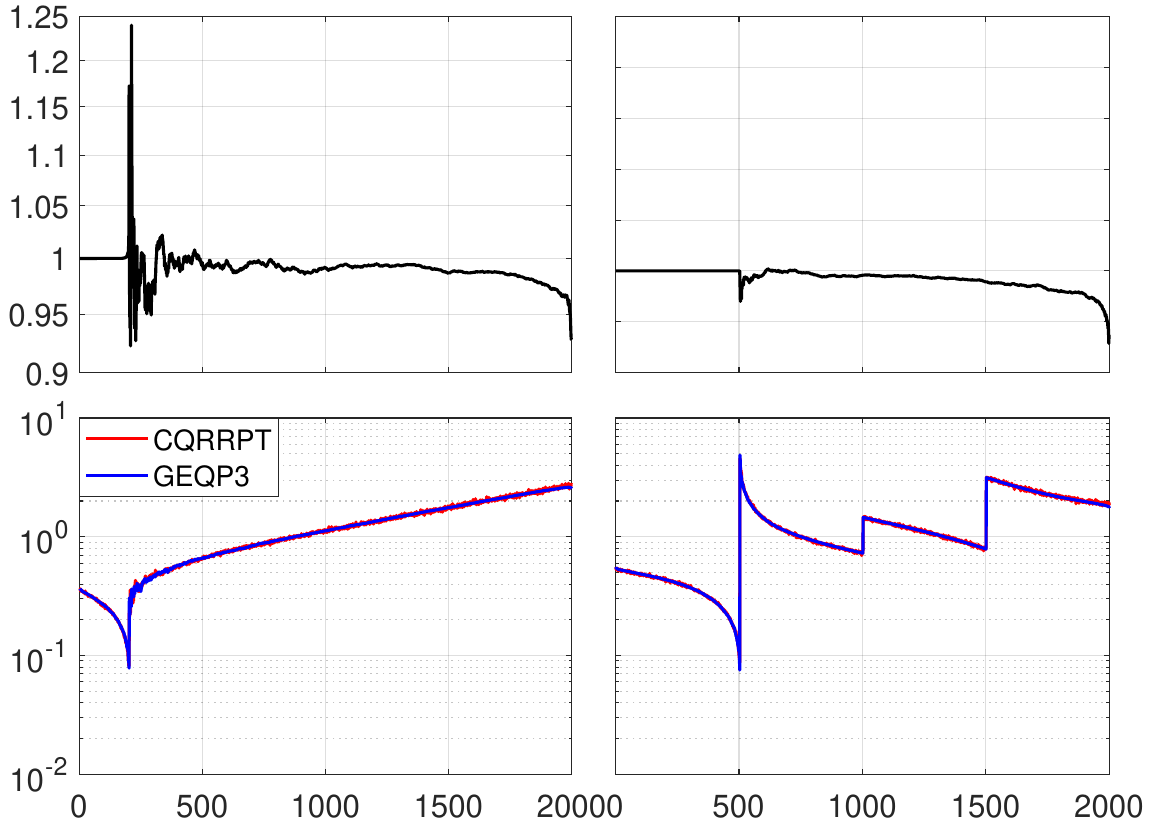}
    \put (15, 72) {polynomial-decay}
    \put (62, 72) {staircase-shaped}
    \put (25, 55) {$\scalemath{1.0}{\displaystyle\frac{\|\mtx{C}^{\text{qp3}}_\ell\|_{\mathrm{F}}}{ \|\mtx{C}^{\text{ours}}_\ell\|_{\mathrm{F}}}}$}
    \put (28, -2) {$\ell$}
    \put (25, 18) {$\scalemath{1.0}{\displaystyle\frac{|\mtx{R}[\ell,\ell]|}{\sigma_{\ell}}}$}
    \put (75, -2) {$\ell$}
    
    \end{overpic}
    \vspace{1ex}
    \small\captionof{figure}{
        \small Pivot quality results for low-coherence matrices with two types of spectral decay.}\label{fig:pivot_qual_combined}
\end{figure}

\cref{fig:pivot_qual_combined} visualizes metrics of CQRRPT's pivot quality for these matrices, when it is configured to use maximally aggressive dimension reduction via SASOs 
($\gamma = 1$ and one nonzero per column).
For both matrices the ratio of \code{GEQP3}'s residual-norm metric to that of CQRRPT is close to 1.
Although, it is noteworthy that the deviations of this ratio from 1 are more pronounced at ranks which coincide with rapid drops in the singular values.
For the second pivot quality metric, observe that the curves for CQRRPT and \code{GEQP3} are extremely similar (in fact, visually coincident) for the two types of matrices.

\FloatBarrier

\subsection{An example high-coherence matrix}
\label{subsec:high_coherence_mat}

Here we consider a matrix $\mtx{M}$ of the same dimensions as before, $(m, n) = (131072, 2000)$, constructed in three steps.
The first step is to vertically stack $c = \lfloor m / n\rfloor$ copies of the $n \times n$ identity matrix and one copy of the first $m - cn$ rows of the $n \times n$ identity.
The second step is to select $n$ rows of $\mtx{M}$ at random and multiply them by $10^{10}$.
Finally, the third step is to multiply $\mtx{M}$ on the right by a random $n \times n$ orthogonal matrix.
The left singular vectors are not explicitly generated in this matrix to ensure its high coherence.

Column one of \cref{fig:high_coherence_combined} shows pivot quality results when applying CQRRPT to $\mtx{M}$ with overly-aggressive dimension reduction; column two of \cref{fig:high_coherence_combined} shows the analogous data for only modestly more expensive dimension reduction.
The figures' combined message is two-fold: 
that there is good reason to use $\gamma > 1$ and more than one nonzero per column in $\mtx{S}$, and that extensive parameter tuning is \emph{not} needed to achieve reliable results from CQRRPT at low computational cost.

\begin{figure}[htb!]
    \centering
    \begin{overpic}[width=0.8\linewidth]{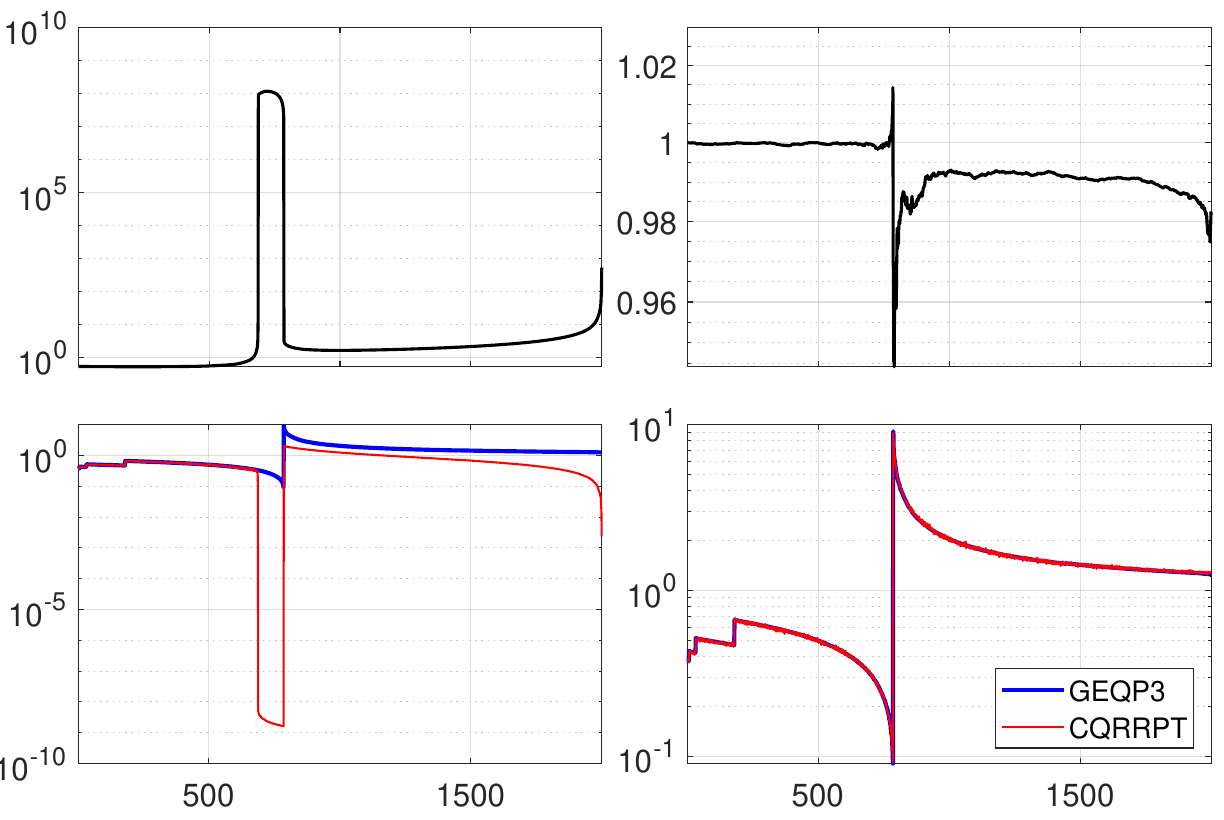}
    
    \put (63, 67) {$\gamma = 1.25$, $\operatorname{nnz/col} \equiv 4$}

    \put (14, 67) {$\gamma = 1.0$, $\operatorname{nnz/col} \equiv 1$}
    
    \put (30, 50) {$\scalemath{1.0}{\displaystyle\frac{\|\mtx{C}^{\text{qp3}}_\ell\|_{\mathrm{F}}}{ \|\mtx{C}^{\text{ours}}_\ell\|_{\mathrm{F}}}}$}
    \put (28, -1) {$\ell$}
    \put (31, 15) {$\scalemath{1.0}{\displaystyle\frac{|\mtx{R}[\ell,\ell]|}{\sigma_{\ell}}}$}
    \put (75, -1) {$\ell$}
    
    \end{overpic}
    \small\captionof{figure}{
        \small 
        Pivot quality results for a high-coherence matrix under two choices of sketching distribution parameters.
        }
        \label{fig:high_coherence_combined}
\end{figure}

\FloatBarrier

\section{Algorithm benchmarking}
\label{sec:speed_experiments}

CQRRPT, together with its testing and benchmarking frameworks, have been implemented in the open-source C++ library called \RandLAPACK.
Together with its counterpart, \RandBLAS, \RandLAPACK provides a platform for developing, testing, and benchmarking of high-performance RandNLA algorithms.
This software was developed as part of the BALLISTIC \cite{BALLISTIC} project and is actively contributed to by the authors of this paper.
All experiments in this section were run using the following version of our software:

\small
\begin{quote}
\url{https://github.com/BallisticLA/RandLAPACK/releases/tag/CQRRPT-benchmark-update}.
\end{quote}

The code for CQRRPT can be found in \code{/RandLAPACK/drivers/}.
The code for constructing and dispatching the experiments can be found in \code{/benchmark/*}.

Experiments in this section are only concerned with algorithm speed.
We measure speed in terms of either an algorithm's canonical flop rate or wall clock time.
The former metric is obtained by dividing the \code{GEQRF} flop count for a given matrix size (see \cite{LAWN41:1994}) by the wall clock time required by a particular run of the algorithm.
This metric helps in understanding how well different algorithms utilize the capabilities of a given machine.

The test matrices were generated with \RandBLAS, by sampling each entry iid from the standard normal distribution.
While these test matrices would not be suitable for assessing pivot quality, they are perfectly appropriate for speed benchmarks.
All of our tests used double-precision arithmetic, all code was compiled with the -O3 flag, and we use \code{OMP\_NUM\_THREADS} = $48$ unless otherwise stated; MKL uses the same number of threads.
We ran the experiments on a single node in UT Knoxville's ISAAC-NG \href{https://oit.utk.edu/hpsc/isaac-open-enclave-new-kpb/system-overview-cluster-at-kpb/}{cluster}, equipped with two Intel Xeon Gold 6248R processors. Detailed machine specifications appear in \cref{table:processor_config}.
\FloatBarrier
\begin{table}[htp]
\small\def\arraystretch{1.3}
\centering
\begin{tabular}{|cc||c|}

\hline
\multicolumn{2}{|c||}{\textbf{}}                                                                                                                                            & \textbf{Intel Xeon Gold 6248R (2x)}     \\ \hline \hline
\multicolumn{2}{|c||}{\textbf{Cores per socket}}                                                                                                                                     & 24                                       \\ \hline 
\multicolumn{1}{|c|}{\multirow{2}{*}{\textbf{Clock Speed}}}                                                                                      & \textbf{\begin{tabular}[c]{@{}c@{}}Base\end{tabular}}   & 3.00 GHz                              \\ \cline{2-3} 
\multicolumn{1}{|c|}{}                                                                                      & \textbf{\begin{tabular}[c]{@{}c@{}}Boost\end{tabular}}  & 4.00 GHz                               \\ \hline
\multicolumn{1}{|c|}{\multirow{3}{*}{\textbf{\begin{tabular}[c]{@{}c@{}}Cache sizes per socket\end{tabular}}}}         & \textbf{L1}                                                   & 1536 KB                                 \\ \cline{2-3} 
\multicolumn{1}{|c|}{}                                                                                      & \textbf{L2}                                                   & 24 MB                                   \\ \cline{2-3} 
\multicolumn{1}{|c|}{}                                                                                      & \textbf{L3}                                                   & 35.75 MB                                \\ \hline
\multicolumn{1}{|c|}{\multirow{2}{*}{\textbf{RAM}}}                                                         & \textbf{Type}                                                 & DDR4-2933                                  \\ \cline{2-3} 
\multicolumn{1}{|c|}{}                                                                                      & \textbf{\begin{tabular}[c]{@{}c@{}}Size \end{tabular}}        & 192 GB                                  \\ \hline
\multicolumn{2}{|c||}{\textbf{\BLAS \& \LAPACK}}                                                                                                                                     & MKL 2023.2                                     \\ \hline
\multicolumn{2}{|c||}{\textbf{\begin{tabular}[c]{@{}c@{}} GFLOPS in GEMM\end{tabular}}}                                                                                     & 1570                                        \\ \hline
\end{tabular}
\caption{\small Key configurations on hardware where testing was performed.}
\label{table:processor_config}
\end{table} 

\FloatBarrier

\subsection{CQRRPT versus other algorithms}
\label{sublec:alg_compare}

Our first round of experiments compares CQRRPT to other algorithms for QR and QRCP.
Specifically, we compare our algorithm to the following.
\begin{itemize}
    \item \code{GEQRF} - standard unpivoted Householder QR.
    \item \code{GEQR} - unpivoted QR, which dispatches either a specialized algorithm for QR of tall and skinny matrices \textit{or} a general-purpose algorithm according to implementation-specific logic.
    The selected algorithm is nominally based on a prediction of which algorithm will be more efficient for the given matrix.
    Intel MKL documentation does not promise that \code{GEQR} makes the optimal decision in this regard.
    \item \code{GEQP3} - standard pivoted QR.
    \item \code{GEQPT} - composed of performing \code{GEQR} on an input matrix to get the factors $\mtx{Q}_{1}$ and $\mtx{R}_{1}$ and then getting data $(\mtx{Q}_{2}, \mtx{R}, J)$ by running \code{GEQP3} on $\mtx{R}_{1}$; the final matrix $\mtx{Q}$ is defined implicitly as the product of $\mtx{Q}_{1}$ and $\mtx{Q}_{2}$.
    \item \code{sCholQR3} - shifted CholeskyQR3, our implementation of \cite[Algorithm 4.2]{FK2020}, using the Frobenius norm on Line $3$ of said algorithm.
\end{itemize}
The performance data for each algorithm is taken as the best of its execution times over twenty consecutive runs.

\paragraph{Varying matrix sizes.}
Our first set of experiments was run on $m \times n$ matrices with $m = 65536$ (\cref{fig:qr_flops}, left) and $m = 131072$ (\cref{fig:qr_flops}, right), with $n$ varying as the powers of two from $512$ to $8192$. The sampling factor, $\gamma$, is set to the default value of $1.25$; the number of nonzeros per column in the sketching operator is set to the default value of $4$.
Note that all matrices in this experiment are out of cache.

\begin{figure}[htb!]
\centering
\begin{minipage}{.5\textwidth}
  \centering
    \begin{overpic}[width=\textwidth]{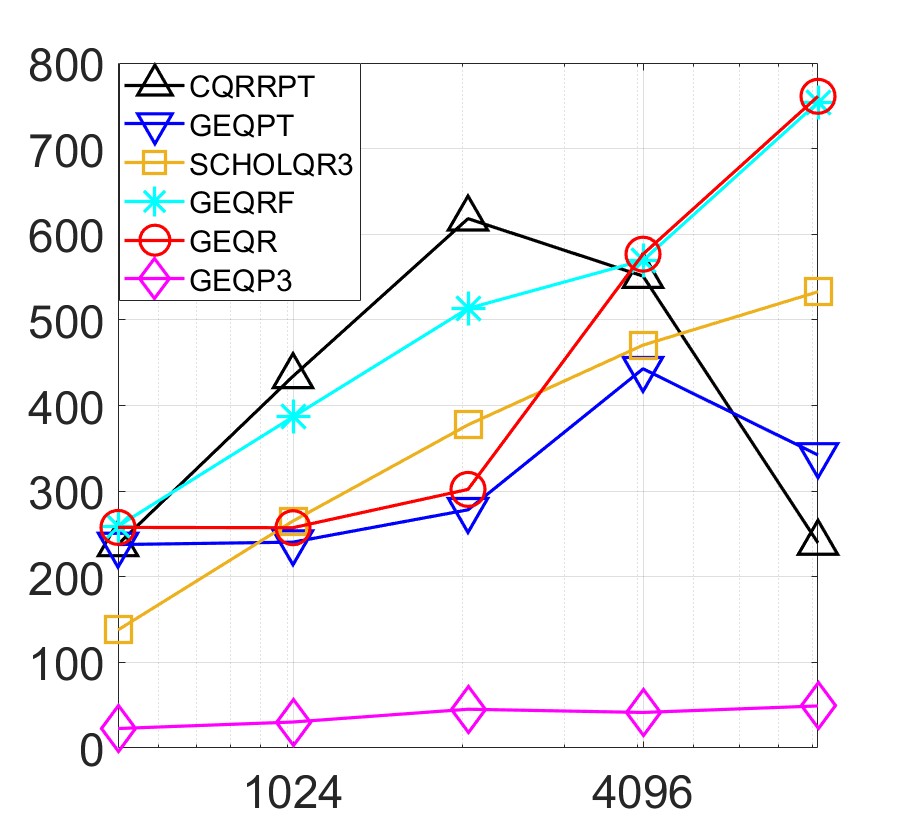}
    \put (-2, 49) {\rotatebox[origin=c]{90}{\textbf{GFLOP/s}}}
    \put (43, 1) {\textbf{columns}}
    \end{overpic}
\end{minipage}
\begin{minipage}{.49\textwidth}
  \centering
    \begin{overpic}[width=\textwidth]{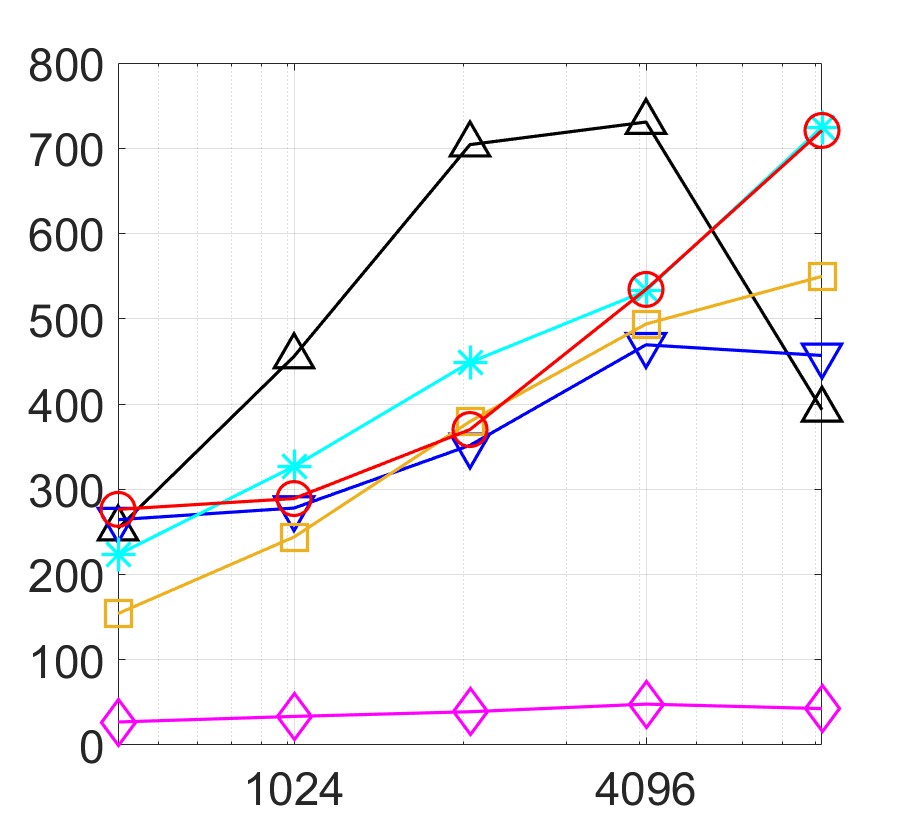}
    \put (43, 1) {\textbf{columns}}
    \end{overpic}
\end{minipage}
\captionof{figure}{\small QR schemes performance comparisons for matrices with fixed numbers of rows ($2^{16} = 65536$, and $2^{17} = 131072$ respectively) and varying numbers of columns ($512, \ldots, 8192$). 
In the case with $131072$ rows, CQRRPT remains the fastest algorithm up until $n = 8192$; at that point, operations on submatrices in underlying \LAPACK routines no longer fit in the cache.
}
\label{fig:qr_flops}
\end{figure}

The results show that with an increase in input matrix size, CQRRPT outperforms alternative pivoted schemes and has either comparable or superior performance to unpivoted schemes. 
We are seeing an order of magnitude acceleration over \code{GEQP3}, and accelerations of up to $2$x over \code{GEQR} and \code{GEQPT}, and $1.6$x over \code{GEQRF}, for matrices with $131072$ rows.
On a plot with $65536$ rows, the performance of CQRRPT declines after $2048$ columns due to the fact that the matrix is not tall enough.

Note that \code{GEQR} fails to match the performance of \code{GEQRF} and consequently causes poor performance in \code{GEQPT}.
This is likely due to the fact that the matrix sizes that we use in our experiments are not considered ``tall enough'' by the internal metric of the MKL implementation of \code{GEQR}.
For experiments with matrices where $m$ is several orders of magnitude larger than $n$, we refer the readers to \cref{sec:more_experiments}.

\paragraph{Varying number of threads.}
\cref{table:threads} depicts the performance scaling of various algorithms with the change in number of threads used. 
For the dimensions of the test matrix used here, CQRRPT remains the fastest algorithm for any number of threads. 
\begin{table}[h!]
\small \centering
\begin{tabular}{ c c c c c c c }
 \hline
 Threads & CQRRPT & GEQPT & sCholQR3 & GEQRF & GEQR & GEQP3 \\ 
 \hline
  8   & 250\phantom{.234} & 159\phantom{.801} & 147\phantom{.178} & 121\phantom{.854} & 166\phantom{.525} & 21\phantom{.3436} \\
  16  & 490\phantom{.108} & 286\phantom{.961} & 276\phantom{.693} & 274\phantom{.553} & 300\phantom{.925} & 25\phantom{.8516} \\
  24  & 577\phantom{.212} & 351\phantom{.439} & 316\phantom{.223} & 343\phantom{.984} & 369\phantom{.317} & 29\phantom{.2671} \\
  36  & 677\phantom{.429} & 372\phantom{.021} & 388\phantom{.738} & 411\phantom{.510} & 393\phantom{.478} & 50\phantom{.7389} \\
  48  & 704\phantom{.264} & 351\phantom{.525} & 379\phantom{.473} & 448\phantom{.884} & 370\phantom{.047} & 39\phantom{.025} \\
\hline
\end{tabular}

\caption{\small Canonical GFLOP/s scaling with the number of threads used, given $\mtx{M} \in \R^{131072 \times 2048}$. At this particular input matrix size, CQRRPT outperforms all other listed algorithms.}
\label{table:threads}
\end{table}

It is worth noting that CQRRPT has another possible performance advantage over the alternative algorithms that come from the output format that CQRRPT uses. Namely, having an explicit representation of a $\mtx{Q}$-factor ensures that the factor can be applied via a fast \code{GEMM} function as opposed to a slower \code{ORMQR}, which is used to apply an implicitly-stored $\mtx{Q}$-factor. 
We acknowledge the existence of applications where having an implicitly stored $\mtx{Q}$-factor is a necessity (for example, when the full square $\mtx{Q}$ is required).
In the context of such applications, the output format of CQRRPT would be disadvantageous.

\subsection{Profiling and optimizing CQRRPT}
\label{subsec:algorithm_engineering}

As seen from Algorithm \ref{alg:CQRRPT}, CQRRPT consists of a handful of subcomponents. 
Understanding how each such subcomponent affects overall runtime for varying input parameters is important for potential future optimizations of~CQRRPT.
To give perspective on this, we present some timing data in~\cref{CQRRPT Inner Speed Fig 1}.
As before, we use an $m \times n$ Gaussian test matrix with $m = 131072$ and $n$ as powers of two from $32$ to $16384$.
We also set the sampling factor, $\gamma$, to the default value of $1.25$.

\FloatBarrier

\begin{figure}[htp]
\minipage{0.6\textwidth}
\centering
    \begin{overpic}[width=0.9\linewidth]{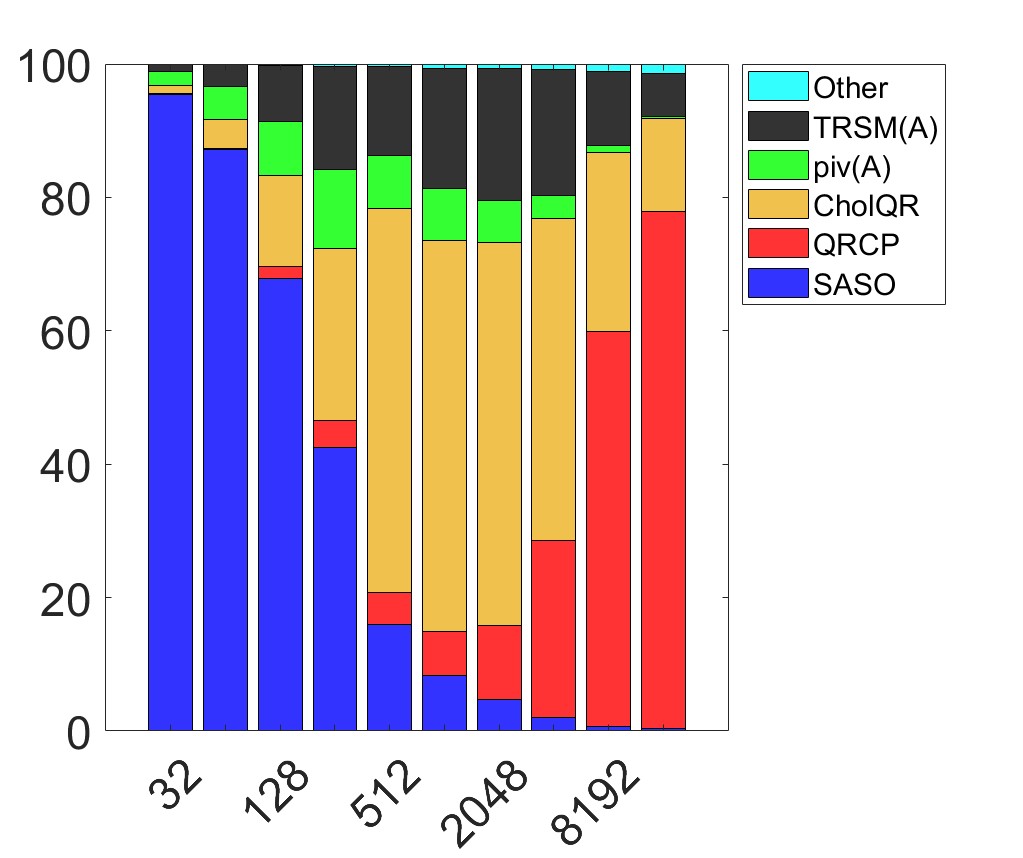}
    \put (-2, 42) {\rotatebox[origin=c]{90}{\textbf{runtime (\%)}}}
    \put (28, -3) {\textbf{columns}}
    \end{overpic}
  \endminipage
  \minipage{0.4\textwidth}
  \captionof{figure}{\small Percentages of CQRRPT's runtime, occupied by its respective subroutines. 
  Note that the cost of sketching becomes negligible for larger matrices. 
  Note also that when $d \geq n$, the cost of applying QRCP to the $d \times n$ sketch grows as $\Omega(n^3)$.
  By contrast, the cost of applying CholeskyQR to the $m \times n$ preconditioned matrix grows as $\mathcal{O}(m n^2)$.
  Therefore, it is reasonable that QRCP consumes a larger fraction of runtime as $n$ increases.
  }
  \label{CQRRPT Inner Speed Fig 1}
  \endminipage
\end{figure}

An immediate observation to be made here is that QRCP and CholeskyQR become the most time-consuming subroutines as the matrix size increases. 
Additionally, QRCP clearly becomes the main computational bottleneck. 
As the embedding dimension parameter greatly affects the timing of QRCP, it is important to understand how the choice of this parameter affects CQRRPT's overall wall clock time.

\FloatBarrier

\subsubsection{Effect of varying the embedding dimension parameter}\label{subsec:embed_effect}

Let's take a look at the effect of varying the embedding dimension in practice. 
We use Gaussian random input matrices of sizes $131072 \times n $ for $n \in \{1024, 2048, 4096\}$ with the ratio $\gamma = d/n$ varying from $1$ to $4$ in steps of $0.5$.
Results are presented in \cref{fig:var_d}.

\begin{figure}[htb!]
\minipage{0.55\textwidth}
\centering
    \begin{overpic}[width=0.9\linewidth]{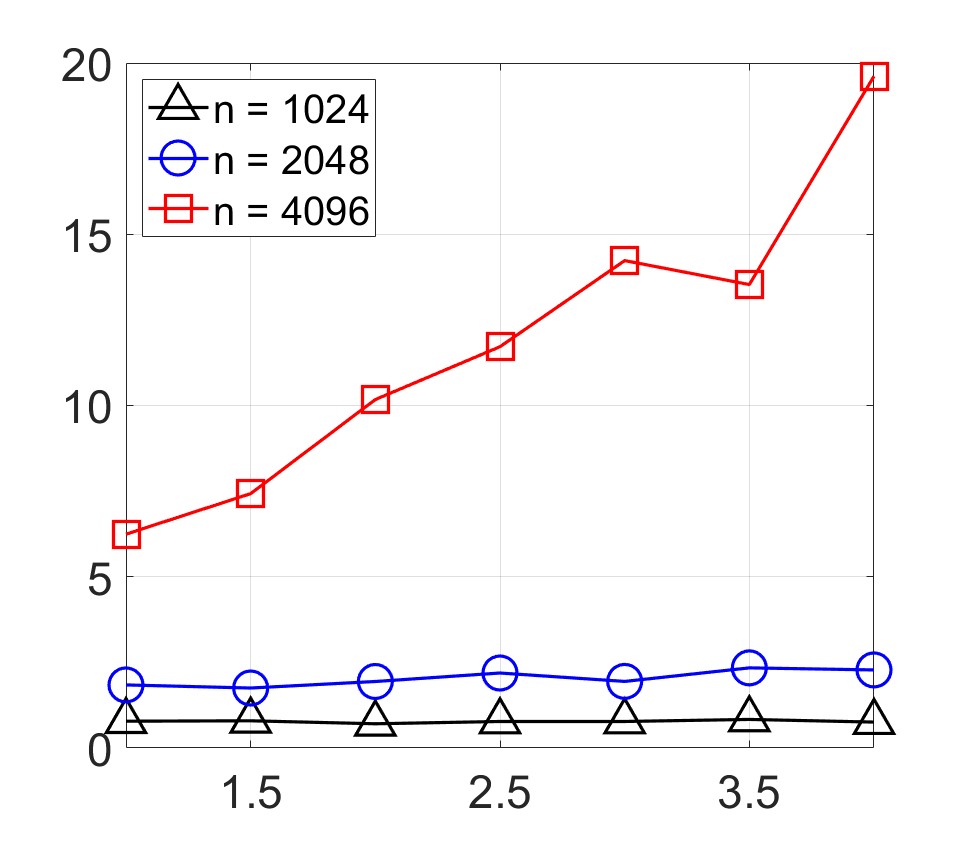}
    \put (-2, 45) {\rotatebox[origin=c]{90}{\textbf{runtime (s)}}}
    \put (47, -2) {\textbf{d/n}}
    \put (17.5, 8) {\textbf{\textcolor{red}{X}}}
    \end{overpic}
  \endminipage
  \minipage{0.4\textwidth}
  \captionof{figure}{\small 
  Effect of varying the sampling factor $\gamma \in \{1, 1.5, 2, ..., 4\}$ for matrices of sizes $131072 \times \{1024, 2048, 4096\}$. Runtime represents the wall clock time for the full execution of CQRRPT. An increase of the embedding dimension has a larger effect on wider matrices, as QRCP becomes more expensive with the increased number of columns, as shown in \cref{CQRRPT Inner Speed Fig 1}.
  Note that our default value of $\gamma = 1.25$ is marked with \textbf{\textcolor{red}{X}}; performance in this case can be inferred by interpolating between the first two data points of each series.
  }
  \label{fig:var_d}
  \endminipage
\end{figure}

\FloatBarrier

Note that for a smaller column size $n$, increasing the parameter $\gamma$ has little to no effect on CQRRPT's overall runtime, as the portion of runtime dedicated to performing the QRCP on a sketch is smaller. 
However, as the figure suggests, for a larger $n=4096$, the effect of increasing $\gamma$ is rather noticeable. 
Recalling our results from \cref{subsec:low_coherence}, using small values of $\gamma$ can reliably produce accurate results for matrices of low coherence. Therefore, a good starting point for the value of this parameter would be somewhere between one and two in such cases.

\subsubsection{Performing HQRRP on the sketch}
\label{subsec:HQRRP_backed_CQRRPT}

By now we have repeatedly seen that CQRRPT's method for QRCP of $\sk{\mtx{M}}$ can decisively impact its runtime when $n$ is large.
Here we explore the possibility of accelerating this operation by replacing CQRRPT's call to \code{GEQP3} with a call to HQRRP (Householder QR Factorization With Randomization for Pivoting, see \cite{MOHvdG:2017:QR}).
As outlined in \cref{subsec:earlier_rand_qrcp}, HQRRP uses low-dimensional random projections to make pivot decisions in small blocks. 

We incorporated an HQRRP implementation into \RandLAPACK and ran experiments comparing HQRRP to \code{GEQP3} directly.
In these experiments, we used square matrices of sizes $n \times n$ for $n \in \{1000, 2000, \ldots, 10000\}$ with varying number of OpenMP threads used (which also sets the number of threads used by Intel MKL).
Another important tuning parameter to consider in HQRRP benchmarking is the block size.
Upon experimenting with different HQRRP block sizes, we have concluded that on our particular system, using the block size of $32$ results in the best performance for HQRRP. 
Using other block sizes in $\{8, 16, 64, 128, 256\}$ resulted in worse performance than what we report in this paper. 
Results are presented in \cref{fig:hqrrpvsgeqp3}.

\begin{figure}[htp]
\centering
\minipage{0.5\textwidth}
        \begin{overpic}[width=\textwidth]{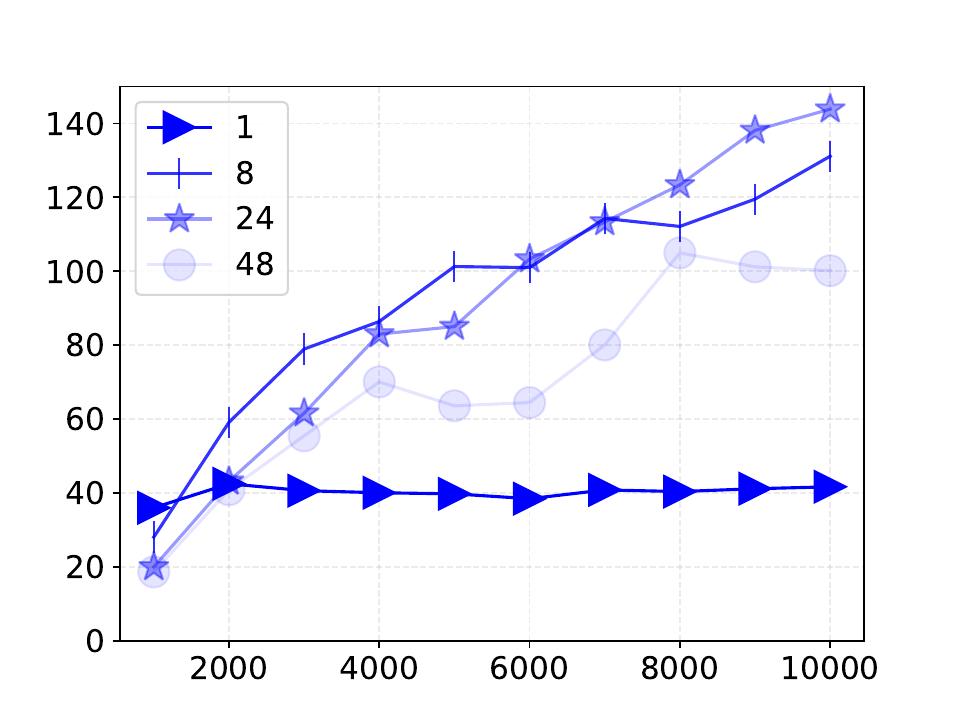}
        \put (42, 68) {\textbf{HQRRP}}
        \put (-1, 35) {\rotatebox[origin=c]{90}{\textbf{GFLOP/s}}}
        \put (42, -1) {\textbf{columns}}
        \end{overpic}
\endminipage
\minipage{0.5\textwidth}
        \begin{overpic}[width=\textwidth]{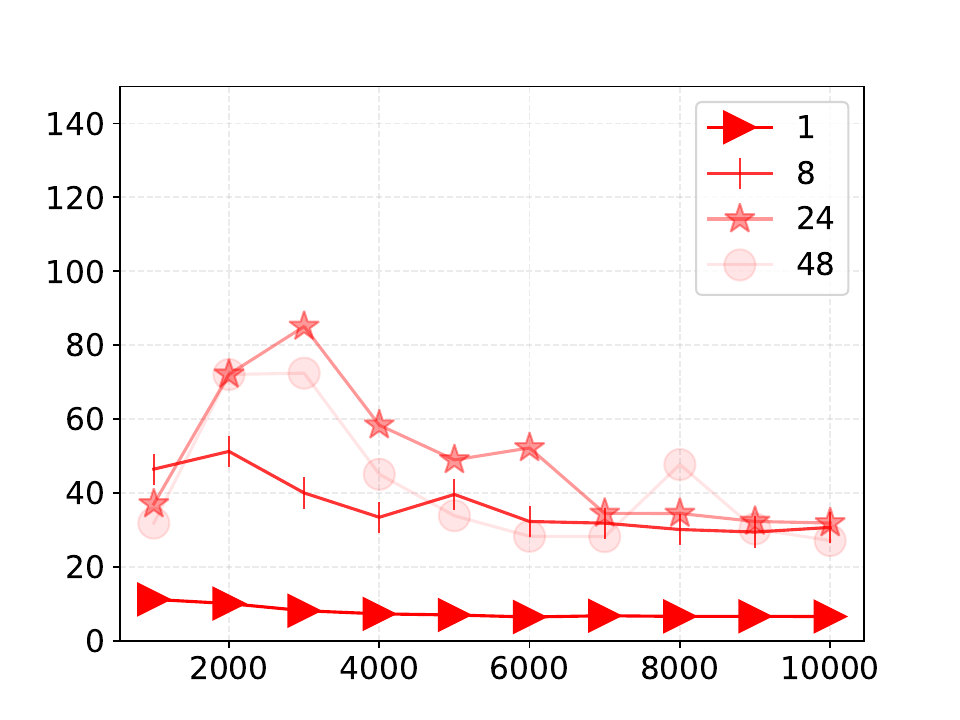}
        \put (42, 68) {\textbf{GEQP3}}
        \put (42, -1) {\textbf{columns}}
        \end{overpic}
\endminipage\vfill

\caption{\small Canonical flop rates for HQRRP vs \code{GEQP3} with \code{OMP\_THREADS} in $\{1, 8, 24, 48\}$. 
The experimental speedup of HQRRP over \code{GEQP3} for the larger matrix sizes is significant, but it does not match the speedup reported in \cite{MOHvdG:2017:QR}; this may be due to a variety of factors (primarily, the machine used for benchmarking). 
}
\label{fig:hqrrpvsgeqp3}
\end{figure}

As seen in \cref{fig:hqrrpvsgeqp3},
HQRRP achieves its peak performance with $24$ threads for large matrix sizes, undoubtedly beating the performance of \code{GEQP3}. 
However, for small input matrices, using a large number of threads results in HQRRP failing to outperform \code{GEQP3}.
When 48 threads are used, the performance of HQRRP stagnates, similar to other QR factorization algorithms; see \cref{sec:more_experiments} for an explanation of why this happens.

\cref{fig:hqrrpvsgeqp3} allows us to conclude that it is preferred to use $8$ threads when running HQRRP on the matrices with $n \leq 6000$; using $24$ threads would be optimal for matrices of larger sizes.
Additionally, using $24$ threads should result in the best performance of \code{GEQP3}. 
Assuming CQRRPT used ideal parameter choices for HQRRP and \code{GEQP3} on our machine, we would expect the HQRRP version to outperform the \code{GEQP3} version when $n \geq 3000$.

\section{Conclusion}
\label{sec:conclusion}

This paper introduces CQRRPT, a randomized algorithm for computing a QRCP factorization of a tall matrix. The algorithm consists of two ingredients: obtaining the pivot order and an approximate R-factor from a small random sketch; and retrieving the complete QRCP factorization by a preconditioned CholeskyQR.
In comparison to standard QR algorithms, such as Householder QR, TSQR, and their pivoted (possibly randomized) derivatives, our CQRRPT requires $\frac{4}{3}$x less flops\footnote{to compute QRCP factorization in explicit form} and is better suited to modern computational architectures due to fewer global synchronization points, fewer data passes, and/or more straightforward reduction operator.
Notably, CQRRPT provides numerical stability even for rank-deficient matrices, in contrast to existing sketched QR, and CholeskyQR methods.
This makes the algorithm an appealing tool for orthogonalization. 

We established rigorous conditions for the random sketch required by CQRRPT to attain the rank-revealing property and numerical stability.
These conditions can be met by employing a sketching distribution that possesses oblivious subspace embedding property.
The robustness and performance advantages of CQRRPT were confirmed through extensive numerical experiments of a \RandLAPACK implementation based on \RandBLAS.
In particular, we observed flop rates improvements of up to $10$x compared to the standard pivoted QR routine $\code{GEQP3}$ and $1.4$x compared to the unpivoted QR routine $\code{GEQRF}$ from \LAPACK.

This paper also highlights several promising research directions.
First, the stability requirements and the efficiency of CQRRPT in certain scenarios can be improved by using more sophisticated sketching operators and schemes for computing the sketch-orthonormal preconditioned matrix.
The performance of CQRRPT in modern computing environments such as GPUs, distributed systems, and multi-(or low-)precision arithmetic systems also remains to be explored.
Second, combining our theoretical results from~\cref{sec:analysis_exact} with the analysis from~\cite{XGL:2017:RandQRCP} may lead to a deeper understanding of the properties of HQRRP and other RandNLA methods existing or yet to emerge.
Third, CQRRPT can be extended to be applicable to matrices with any aspect ratio by integrating ideas from this paper with ideas from~\cite{XGL:2017:RandQRCP,DG:2017:QR,Martinsson:2015:QR,MOHvdG:2017:QR} for HQRRP and the related randomized and communication-avoiding QR techniques.

\subsection*{Acknowledgements}
\small 

The authors would like to thank three anonymous referees for valuable feedback, which greatly improved the clarity of our results and the paper's presentation.

This work was partially funded by an NSF Collaborative Research Framework: Basic ALgebra LIbraries for Sustainable Technology with Interdisciplinary Collaboration (BALLISTIC), a project of the International Computer Science Institute, the University of Tennessee’s ICL, the University of California at Berkeley, and the University of Colorado at Denver (NSF Grant Nos. 2004235, 2004541, 2004763, 2004850, respectively).
MWM would also like to acknowledge the NSF, DOE, and ONR Basic Research Challenge on RLA for providing partial support for this work.

RM was partially supported by Laboratory Directed Research and Development (LDRD) funding from Sandia National Laboratories; Sandia is a multimission laboratory managed and operated by National Technology \& Engineering Solutions of Sandia, LLC, a wholly-owned subsidiary of Honeywell International Inc., for the U.S. Department of Energy’s National Nuclear Security Administration under contract DENA0003525. 

This research was sponsored by the Department of the Air Force Artificial
Intelligence Accelerator and was accomplished under Cooperative
Agreement Number FA8750-19-2-1000.

The views and conclusions contained
in this document are those of the authors and should not be
interpreted as presenting the official policies, either expressed or
implied, of the Department of the Air Force, the Department of Energy, or the U.S.\ Government.
The U.S.\ Government is authorized to reproduce and distribute reprints
for Government purposes notwithstanding any copyright notation herein.

\bibliography{references/refs.bib}

\begin{bibdiv}
\begin{biblist}

\bib{LAWN41:1994}{techreport}{
      author={Anderson, E.},
      author={Dongarra, J.},
      author={Ostrouchov, S.},
       title={{LAPACK} working note 41 installation guide for {LAPACK}},
        type={LAPACK Working Note},
   publisher={Innovative Computing Laboratory, University of Tennessee},
        date={1994},
      number={37996-1301},
}

\bib{AMT:2010:Blendenpik}{article}{
      author={Avron, H.},
      author={Maymounkov, P.},
      author={Toledo, S.},
       title={{Blendenpik: Supercharging {LAPACK}'s Least-Squares Solver}},
        date={2010-01},
     journal={{SIAM} Journal on Scientific Computing},
      volume={32},
      number={3},
       pages={1217\ndash 1236},
         url={https://doi.org/10.1137/090767911},
}

\bib{Balabanov:2022:cholQR}{misc}{
      author={Balabanov, O.},
       title={Randomized {C}holesky {QR} factorizations},
   publisher={arXiv},
        date={2022},
         url={https://arxiv.org/abs/2210.09953},
}

\bib{BG:2013}{article}{
      author={Boutsidis, Christos},
      author={Gittens, Alex},
       title={Improved matrix algorithms via the subsampled randomized
  {Hadamard} transform},
        date={2013-01},
     journal={{SIAM} Journal on Matrix Analysis and Applications},
      volume={34},
      number={3},
       pages={1301\ndash 1340},
         url={https://doi.org/10.1137/120874540},
}

\bib{https://doi.org/10.48550/arxiv.2111.14641}{misc}{
      author={Balabanov, O.},
      author={Grigori, L.},
       title={Randomized block gram-schmidt process for solution of linear
  systems and eigenvalue problems},
   publisher={arXiv},
        date={2021},
         url={https://arxiv.org/abs/2111.14641},
}

\bib{BG:2021:GramSchmidt}{article}{
      author={Balabanov, O.},
      author={Grigori, L.},
       title={Randomized {G}ram--{S}chmidt process with application to
  {GMRES}},
        date={2022},
     journal={SIAM Journal on Scientific Computing},
      volume={44},
      number={3},
       pages={A1450\ndash A1474},
         url={https://doi.org/10.1137/20M138870X},
}

\bib{BG:1965:QRCP}{article}{
      author={Businger, Peter},
      author={Golub, Gene~H.},
       title={Linear least squares solutions by householder transformations},
        date={1965-06},
        ISSN={0945-3245},
     journal={Numerische Mathematik},
      volume={7},
      number={3},
       pages={269–276},
         url={http://dx.doi.org/10.1007/BF01436084},
}

\bib{bjork1996lsq}{book}{
      author={Bj{\"o}rk, {\AA}ke},
       title={Numerical methods for least squares problems},
   publisher={SIAM},
        date={1996},
        note={ISBN 0-89871-360-9},
}

\bib{balabanov2019randomized}{article}{
      author={Balabanov, Oleg},
      author={Nouy, Anthony},
       title={Randomized linear algebra for model reduction. part i: Galerkin
  methods and error estimation},
        date={2019},
     journal={Advances in Computational Mathematics},
      volume={45},
      number={5},
       pages={2969\ndash 3019},
}

\bib{Cohen:2016:SJLTs}{inproceedings}{
      author={Cohen, M.B.},
       title={Nearly tight oblivious subspace embeddings by trace
  inequalities},
        date={2016-12},
   booktitle={Proceedings of the 27th annual {ACM}-{SIAM} {S}ymposium on
  {D}iscrete {A}lgorithms},
   publisher={Society for Industrial and Applied Mathematics},
}

\bib{BALLISTIC}{misc}{
      author={Demmel, J.},
      author={Dongarra, J.},
      author={Langou, J.},
      author={Langou, J.},
      author={Luszczek, P.},
      author={Mahoney, M.W.},
       title={Prospectus for the next {LAPACK} and {ScaLAPACK} libraries:
  {B}asic {AL}gebra {LI}braries for {S}ustainable {T}echnology with
  {I}nterdisciplinary {C}ollaboration {(BALLISTIC)}},
        date={2020},
}

\bib{demmel23}{article}{
      author={Demmel, J.},
       title={Nearly optimal block-{J}acobi preconditioning},
        date={2023-03},
     journal={{SIAM} Journal on Matrix Analysis and Applications},
      volume={44},
      number={1},
       pages={408\ndash 413},
}

\bib{DG:2017:QR}{article}{
      author={Duersch, J.A.},
      author={Gu, M.},
       title={Randomized {QR} with column pivoting},
        date={2017-01},
        ISSN={1095-7197},
     journal={SIAM Journal on Scientific Computing},
      volume={39},
      number={4},
       pages={C263–C291},
         url={http://dx.doi.org/10.1137/15M1044680},
}

\bib{demmel2012communication}{article}{
      author={Demmel, J.},
      author={Grigori, L.},
      author={Hoemmen, M.},
      author={Langou, J.},
       title={Communication-optimal parallel and sequential {QR} and {LU}
  factorizations},
        date={2012},
     journal={SIAM Journal on Scientific Computing},
      volume={34},
      number={1},
       pages={A206\ndash A239},
}

\bib{DM16_CACM}{article}{
      author={Drineas, P.},
      author={Mahoney, M.~W.},
       title={{RandNLA}: Randomized numerical linear algebra},
        date={2016},
     journal={Communications of the ACM},
      volume={59},
       pages={80\ndash 90},
}

\bib{DM21_NoticesAMS}{article}{
      author={Derezinski, M.},
      author={Mahoney, M.~W.},
       title={Determinantal point processes in randomized numerical linear
  algebra},
        date={2021},
     journal={Notices of the AMS},
      volume={68},
      number={1},
       pages={34\ndash 45},
}

\bib{DMM06}{inproceedings}{
      author={Drineas, P.},
      author={Mahoney, M.~W.},
      author={Muthukrishnan, S.},
       title={Sampling algorithms for $\ell_2$ regression and applications},
        date={2006},
   booktitle={Proceedings of the 17th annual {ACM-SIAM} {S}ymposium on
  {D}iscrete {A}lgorithms ({SODA})},
       pages={1127\ndash 1136},
}

\bib{DMMS07_FastL2_NM10}{article}{
      author={Drineas, P.},
      author={Mahoney, M.~W.},
      author={Muthukrishnan, S.},
      author={Sarl\'{o}s, T.},
       title={Faster least squares approximation},
        date={2011},
     journal={Numerische Mathematik},
      volume={117},
      number={2},
       pages={219\ndash 249},
}

\bib{FGL:2021:CholeskyQR}{misc}{
      author={Fan, Y.},
      author={Guo, Y.},
      author={Lin, T.},
       title={A novel randomized {XR}-based preconditioned {CholeskyQR}
  algorithm},
   publisher={arXiv},
        date={2021},
         url={https://arxiv.org/abs/2111.11148},
}

\bib{FK2020}{article}{
      author={Fukaya, T.},
      author={Kannan, R.},
      author={Nakatsukasa, Y.},
      author={Yamamoto, Y.},
      author={Yanagisawa, Y.},
       title={Shifted {Cholesky} {QR} for computing the {QR} factorization of
  ill-conditioned matrices},
        date={2020},
     journal={SIAM Journal on Scientific Computing},
      volume={42},
      number={1},
       pages={A477–A503 (cit. on pp. 2, 19, 20)},
}

\bib{FNYY2014}{inproceedings}{
      author={Fukaya, T.},
      author={Nakatsukasa, Y.},
      author={Yanagisawa, Y.},
      author={Yamamoto, Y.},
       title={Choleskyqr2: A simple and communication-avoiding algorithm for
  computing a tall-skinny {QR} factorization on a large-scale parallel system},
        date={2014},
   booktitle={2014 5th workshop on latest advances in scalable algorithms for
  large-scale systems},
       pages={31\ndash 38},
}

\bib{GE:1996}{article}{
      author={Gu, M.},
      author={Eisenstat, S.C.},
       title={Efficient algorithms for computing a strong rank-revealing {QR}
  factorization},
        date={1996-07},
     journal={{SIAM} Journal on Scientific Computing},
      volume={17},
      number={4},
       pages={848\ndash 869},
         url={https://doi.org/10.1137/0917055},
}

\bib{GvL:2013:MatrixComputationsBook}{book}{
      author={Golub, G.H.},
      author={Van~Loan, C.F.},
       title={Matrix computations},
    language={en},
     edition={4},
      series={Johns Hopkins Studies in the Mathematical Sciences},
   publisher={Johns Hopkins University Press},
     address={Baltimore, MD},
        date={2013},
}

\bib{Higham:2002:book}{book}{
      author={Higham, N.J.},
       title={Accuracy and stability of numerical algorithms},
     edition={Second},
   publisher={Society for Industrial and Applied Mathematics},
     address={Philadelphia, PA, USA},
        date={2002},
        ISBN={0-89871-521-0},
}

\bib{Higham:blog:rrf}{misc}{
      author={Higham, N.J.},
       title={What is a rank-revealing factorization?},
  how={\url{https://nhigham.com/2021/05/19/what-is-a-rank-revealing-factorization/}},
        date={2021},
        note={[Accessed 03-Apr-2023]},
}

\bib{Higham:blog:big6}{misc}{
      author={Higham, N.J.},
       title={The big six matrix factorizations--- nhigham.com},
  how={\url{https://nhigham.com/2022/05/18/the-big-six-matrix-factorizations/comment-page-1}},
        date={2022},
        note={[Accessed 15-Mar-2023]},
}

\bib{HSBY:2023:rand_chol_qr}{misc}{
      author={Higgins, Andrew~J.},
      author={Szyld, Daniel~B.},
      author={Boman, Erik~G.},
      author={Yamazaki, Ichitaro},
       title={Analysis of randomized {H}ouseholder-{C}holesky {QR}
  factorization with multisketching},
        date={2023},
        note={arXiv:2309.05868},
}

\bib{Mah-mat-rev_BOOK}{book}{
      author={Mahoney, M.~W.},
       title={Randomized algorithms for matrices and data},
      series={Foundations and Trends in Machine Learning},
   publisher={NOW Publishers},
     address={Boston},
        date={2011},
}

\bib{Martinsson:2015:QR}{article}{
      author={Martinsson, P.-G.},
       title={Blocked rank-revealing {QR} factorizations: How randomized
  sampling can be used to avoid single-vector pivoting},
        date={2015},
     journal={arXiv preprint arXiv:1505.08115},
}

\bib{RandLAPACK_Book}{techreport}{
      author={Murray, R.},
      author={Demmel, J.},
      author={Mahoney, M.W.},
      author={Erichson, N.B.},
      author={Melnichenko, M.},
      author={Malik, O.A.},
      author={Grigori, L.},
      author={Luszczek, P.},
      author={Derezinski, M.},
      author={Lopes, M.E.},
      author={Liang, T.},
      author={Luo, H.},
      author={Dongarra, J.},
       title={Randomized numerical linear algebra: A perspective on the field
  with an eye to software},
        date={2023},
         url={https://arxiv.org/abs/2302.11474?context=math.NA},
        note={arXiv:2302.11474:v2},
}

\bib{MOHvdG:2017:QR}{article}{
      author={Martinsson, P.-G.},
      author={Quintana-Ort{\'\i}, G.},
      author={Heavner, N.},
      author={van~de Geijn, R.},
       title={Householder {QR} factorization with randomization for column
  pivoting ({HQRRP})},
        date={2017-01},
     journal={{SIAM} Journal on Scientific Computing},
      volume={39},
      number={2},
       pages={C96\ndash C115},
         url={https://doi.org/10.1137/16m1081270},
}

\bib{MSM:2014:LSRN}{article}{
      author={Meng, X.},
      author={Saunders, M.A.},
      author={Mahoney, M.W.},
       title={{LSRN}: A parallel iterative solver for strongly over- or
  underdetermined systems},
        date={2014-01},
     journal={{SIAM} Journal on Scientific Computing},
      volume={36},
      number={2},
       pages={C95\ndash C118},
         url={https://doi.org/10.1137/120866580},
}

\bib{MT:2020}{article}{
      author={Martinsson, P.-G.},
      author={Tropp, J.~A.},
       title={Randomized numerical linear algebra: Foundations and algorithms},
        date={2020},
     journal={Acta Numerica},
      volume={29},
       pages={403–572},
}

\bib{nguyen2015reproducible}{inproceedings}{
      author={Nguyen, Hong~Diep},
      author={Demmel, James},
       title={Reproducible tall-skinny {QR}},
organization={IEEE},
        date={2015},
   booktitle={2015 {IEEE} 22nd symposium on computer arithmetic},
       pages={152\ndash 159},
}

\bib{Quintana:1998}{article}{
      author={Quintana-Ort{\'\i}, G.},
      author={Sun, X.},
      author={Bischof, C.H.},
       title={A {BLAS}-3 version of the {QR} factorization with column
  pivoting},
        date={1998},
     journal={SIAM Journal on Scientific Computing},
      volume={19},
      number={5},
       pages={1486\ndash 1494},
}

\bib{RT:2008:SAP}{article}{
      author={Rokhlin, V.},
      author={Tygert, M.},
       title={{A fast randomized algorithm for overdetermined linear
  least-squares regression}},
        date={2008-09},
     journal={Proceedings of the National Academy of Sciences},
      volume={105},
      number={36},
       pages={13212\ndash 13217},
         url={https://doi.org/10.1073/pnas.0804869105},
}

\bib{Sarlos:2006}{inproceedings}{
      author={Sarlos, T.},
       title={Improved approximation algorithms for large matrices via random
  projections},
        date={2006},
   booktitle={Proceedings of the 47th annual {IEEE} {S}ymposium on
  {F}oundations of {C}omputer {S}cience ({FOCS})},
      series={FOCS '06},
   publisher={IEEE Computer Society},
     address={USA},
       pages={143–152},
}

\bib{V1969}{article}{
      author={Sluis., A. Van~Der},
       title={Condition numbers and equilibration of matrices.},
        date={1969},
     journal={Numerische Mathematik},
      number={14},
       pages={14\ndash 23},
         url={https://doi.org/10.1007/BF02165096},
}

\bib{Stath:2002}{article}{
      author={Stathopoulos, A.},
      author={Wu, K.},
       title={A block orthogonalization procedure with constant synchronization
  requirements},
        date={2002},
     journal={SIAM J.~Sci.~Comput.},
      volume={23},
       pages={2165\ndash 2182},
}

\bib{TOO2020}{article}{
      author={Terao, T.},
      author={Ozaki, K.},
      author={Ogita, T.},
       title={{LU-Cholesky} {QR} algorithms for thin {QR} decomposition},
        date={2020},
        ISSN={0167-8191},
     journal={Parallel Computing},
      volume={92},
       pages={102571},
  url={https://www.sciencedirect.com/science/article/pii/S0167819119301620},
}

\bib{thakur2005optimization}{article}{
      author={Thakur, Rajeev},
      author={Rabenseifner, Rolf},
      author={Gropp, William},
       title={Optimization of collective communication operations in {MPICH}},
        date={2005},
     journal={The International Journal of High Performance Computing
  Applications},
      volume={19},
      number={1},
       pages={49\ndash 66},
}

\bib{Tropp:2011}{article}{
      author={Tropp, J.A.},
       title={Improved analysis of the subsampled randomized {Hadamard}
  transform},
        date={2011-04},
     journal={Advances in Adaptive Data Analysis},
      volume={03},
      number={01n02},
       pages={115\ndash 126},
         url={https://doi.org/10.1142/s1793536911000787},
}

\bib{tropp2023randomized}{misc}{
      author={Tropp, Joel~A.},
      author={Webber, Robert~J.},
       title={Randomized algorithms for low-rank matrix approximation: Design,
  analysis, and applications},
        date={2023},
}

\bib{XGL:2017:RandQRCP}{inproceedings}{
      author={Xiao, J.},
      author={Gu, M.},
      author={Langou, J.},
       title={Fast parallel randomized {QR} with column pivoting algorithms for
  reliable low-rank matrix approximations},
        date={2017},
   booktitle={2017 {IEEE} 24th international conference on high performance
  computing ({HiPC})},
       pages={233\ndash 242},
}

\bib{yamamoto2015roundoff}{article}{
      author={Yamamoto, Y.},
      author={Nakatsukasa, Y.},
      author={Yanagisawa, Y.},
      author={Fukaya, T.},
       title={Roundoff error analysis of the {CholeskyQR2} algorithm},
        date={2015},
     journal={Electron. Trans. Numer. Anal},
      volume={44},
      number={01},
       pages={306\ndash 326},
}

\bib{yamazaki2014cholqr}{article}{
      author={Yamazaki, I.},
      author={Tomov, S.},
      author={Dongarra, J.},
       title={Mixed-precision {Cholesky QR} factorization and its case studies
  on multicore {CPUs} with multiple {GPUs}},
        date={2015},
     journal={SIAM J. Scientific Computing},
      volume={37},
       pages={C307\ndash C330},
}

\end{biblist}
\end{bibdiv}
\bibliographystyle{alpha}

\normalsize

\appendix

\section{Analysis in exact arithmetic}
\label{sec:analysis_exact}

\subsection{Detailed arithmetic complexity analysis}
\label{app: flop}

In an effort to give the reader a perspective on an approximate flop count of \cref{alg:CQRRPT}, 
let us take a closer look at what happens inside this algorithm after computing the sketch.

\begin{enumerate}
    \item Factor $\sk{\mtx{M}}$ by QRCP. The cost of this step depends on the algorithm used for QRCP. 
        By default, we use the method from \LAPACK's \code{GEQP3}. 
        With \LAPACK's \code{GEQP3}, this operation would cost $4dnk - 2k^{2}(d + n) + 4k^{3}/3$ flops \cite[Algorithm 5.4.1]{GvL:2013:MatrixComputationsBook}, where $k = \rank(\sk{\mtx{M}})$ and $d$ is the user-defined embedding dimension of a sketch.
    \item Apply the preconditioner to $\mtx{M}$, to get $\pre{\mtx{M}}$. This step is done via \code{TRSM}, so the cost is $mk^2$ \cite[Page 120]{LAWN41:1994}.
    \item Run CholeskyQR on $\pre{\mtx{M}}$.
        This operation is composed of \code{SYRK}, \code{POTRF}, and \code{TRSM} functions. 
        This operation would cost $mk(k + 1) + k^{3}/3 + k^2/2 + k/6 + mk^2$ flops \cite[Page 120]{LAWN41:1994}.
    \item Combine the triangular factors from QRCP on $\sk{\mtx{M}}$ and CholeskyQR on $\pre{\mtx{M}}$.
        This cost is negligible; an $\mathcal{O}(n^3)$ \code{TRMM}.
\end{enumerate}

This results in a cumulative flop count of $2mk^{2} + mk(k + 1) + 4dnk - 2k^{2}(d + n) + 5k^{3}/3 + k^2/2 + k/6$. Put simply, we have an algorithm with an asymptotic flop count with a leading term $3mn^{2}$ when $d \in o(n)$ and $k = n$.

\subsection{Rank-revealing properties}
\label{subsec:rankrevealing_proofs}

Here we prove pivot quality results claimed in   \cref{thm:cqrrpt_RRQR,thm:cqrrpt_SRRQR}.
Our arguments use the following concept to relate subspace embedding results (common in RandNLA theory) and restricted condition numbers.

\begin{definition}\label{def:effective_distortion}
    The \underline{effective distortion} of a $d \times m$ matrix $\mtx{S}$ for a linear subspace $L \subset \R^m$ is the smallest real number $\delta$ where there exists a $t \geq 0$ for which $t \mtx{S}$ is a $\delta$-embedding for $L$.
\end{definition}
The specific relationship is as follows \cite[Proposition A.1.1]{RandLAPACK_Book}.
\begin{proposition}\label{prop:effective_distortion}
    Suppose the columns of $\mtx{U}$ are an orthonormal basis for $L = \range(\mtx{M})$, and let $\distortion$ denote the effective distortion of $\mtx{S}$ for $L$.
    If $\distortion < 1$, then we have
    \[
        \cond(\mtx{S}\mtx{U}) = \sqrt{\frac{1+\distortion}{1 - \distortion}}.
    \]
    Furthermore, we have $\distortion = 1$ if and only if $\rank(\mtx{S}\mtx{M}) < \rank(\mtx{M})$.
\end{proposition}

Now we fix an arbitrary permutation vector $J$ of $\{1,\ldots,n\}$ and analyze the R-factors from QR decompositions of the pivoted matrices $\mtx{M}_n = \mtx{M}[\fslice{},J]$ and $\sk{\mtx{M}}_n = \sk{\mtx{M}}[\fslice{},J]$.
Set $k = \rank(\mtx{M})$. We assume $\rank(\sk{\mtx{M}}) = k$, since otherwise there is nothing to prove.

\begin{proposition}\label{thm:skRRQR}
    If $\mtx{S}$ has effective distortion $\distortion < 1$ for the range of $\mtx{M}$, then we have 
	\begin{subequations}
	\begin{align}
		\frac{
                \sigmaSub{j}{\Roo_{\ell}}
            }{
                \sigmaSub{j}{\mtx{M}}
            } 
            &\geq \left (\frac{1-\distortion}{1+\distortion} \right )^{1/2} \frac{\sigmaSub{j}{\sk{\Roo}_{\ell}}}{\sigmaSub{j}{\sk{\mtx{M}}}} \quad \text{ for }\quad j \leq \ell \label{eq:skRRQRa} \\ 
		\frac{\sigmaSub{j}{{\Rtt_{\ell}}}}{\sigmaSub{\ell+j}{{\mtx{M}}}} &\leq \left (\frac{1+\distortion}{1-\distortion} \right )^{1/2} \frac{\sigmaSub{j}{{\sk{\Rtt}_{\ell}}}}{\sigmaSub{\ell+j}{{\sk{\mtx{M}}}}} \quad\text{ for }\quad j \leq k-\ell, \label{eq:skRRQRb} \\
		\text{ and } \qquad \|(\Roo_{\ell})^{-1} \Rot_{\ell}\|_2 &\leq \|(\sk{\Roo}_{\ell})^{-1} \sk{\Rot}_{\ell}\|_2+	
		\left (\frac{1+\distortion}{1-\distortion} \right )^{1/2} \sigmaMin{\sk{\Roo}_{\ell}}^{-1} \| \sk{\Rtt}_{\ell} \|_2. \label{eq:skRRQRc}	
	\end{align}
	\end{subequations}
 \end{proposition} 
\begin{proof}		
        First, note that all of our claimed bounds are invariant under scaling of $\mtx{S}$ by positive constants.
        This means that we can assume $\mtx{S}$ is scaled ($\mtx{S} \leftarrow t\mtx{S}$ for some $t \neq 0$) so that its distortion and its effective distortion coincide.
		From here, it follows directly from the min-max principle and the subspace embedding property of $\mtx{S}$ that
		\begin{equation} \label{eq:skRRQR0}
		(1 - \distortion)^{1/2}\sigmaSub{j}{\mtx{M}} \leq \sigmaSub{j}{\mtx{S}\mtx{M}} \leq (1 + \distortion)^{1/2} \sigmaSub{j}{\mtx{M}}  \text{ for } 1 \leq j \leq n.
		\end{equation}
	Consider the matrix $\pre{\mtx{M}}$ computed in \cref{alg:CQRRPT}. 
		 We can interpret \cref{thm:cond_of_A_pre} with the definition of a subspace embedding to find that 
		\begin{equation}\label{eq:A_pre_singval_bounds}
		(1 + \distortion)^{-1/2}  \leq \sigmaSub{j}{\pre{\mtx{M}}} \leq (1 - \distortion)^{-1/2} \text{ for } 1 \leq j \leq k.
		\end{equation}
    Next, partition the matrix $\pre{\mtx{R}}$ from \cref{alg:CQRRPT} into a $2 \times 2$ block matrix in the way analogous to the partitioning of $\mtx{R}$ and $\sk{\mtx{R}}$, and recall the definition $\mtx{R} = \pre{\mtx{R}} \sk{\mtx{R}}$.

    The bounds in \cref{eq:A_pre_singval_bounds} ensure that $\sigma_{\min}\left( \pre{\Roo}_{\ell}\right) \geq  \sigma_{\min}\left( \pre{\mtx{R}}\right) = \sigma_{\min}\left( \pre{\mtx{M}}\right)  \geq  (1+\distortion)^{-1/2} $ and $\left\| \pre{\Rtt}_{\ell}\right\| \leq \left\| \pre{\mtx{R}}\right\| = \left\| \pre{\mtx{M}}\right\| \leq (1-\distortion)^{-1/2} $.
    Meanwhile, rote calculations let us express the blocks of $\mtx{R}$ as $ \Roo_{\ell} = \pre{\Roo}_{\ell} \sk{\Roo}_{\ell} $, $\Rtt_{\ell} = \pre{\Rtt}_{\ell} \sk{\Rtt}_{\ell} $ and $\Rot_{\ell} = \pre{\Roo}_{\ell} \sk{\Rot}_{\ell} + \pre{\Rot}_{\ell} \sk{\Rtt}_{\ell}$.
    Consequently,
		\small 
		\begin{subequations} \label{eq:skRRQR1}
		\begin{align*} 	
		 \sigmaSub{j}{\sk{\Roo}_{\ell}} &\leq  \sigmaMin{\pre{\Roo}_{\ell}}^{-1} \sigmaSub{j}{\pre{\Roo}_{\ell} \sk{\Roo}_{\ell}} =  \sigmaMin{\pre{\Roo}_{\ell}}^{-1} \sigmaSub{j}{\Roo_{\ell}} \leq  (1+\distortion)^{1/2} \sigmaSub{j}{\Roo_{\ell}}, \\
		  \sigmaSub{j}{\sk{\Rtt}_{\ell}} &\geq  \|\pre{\Rtt}_{\ell}\|_2^{-1} \sigmaSub{j}{\pre{\Rtt}_{\ell} \sk{\Rtt}_{\ell}} =  \| \pre{\Rtt}_{\ell}\|_2^{-1} \sigmaSub{j}{\Rtt_{\ell}} \geq    (1-\distortion)^{-1/2}  \sigmaSub{j}{\Rtt_{\ell}}. 
		\end{align*}
		\end{subequations}
		\normalsize 
		We can combine these inequalities with \cref{eq:skRRQR0} to easily get \cref{eq:skRRQRa,eq:skRRQRb}.	
		
            To show \cref{eq:skRRQRc}, we first notice that
		\small 
		\begin{align*}
              \|(\Roo_{\ell})^{-1}\Rot_{\ell}\|_2
                &= \|(\pre{\Roo}_{\ell} \sk{\Roo}_{\ell})^{-1}( \pre{\Roo}_{\ell} \sk{\Rot}_{\ell} + \pre{\Rot}_{\ell} \sk{\Rtt}_{\ell})\|_2 \\
                &\leq  \|(\sk{\Roo}_{\ell})^{-1} \sk{\Rot}_{\ell}\|_2 + \|(\sk{\Roo}_{\ell})^{-1} (\pre{\Roo}_{\ell})^{-1}\pre{\Rot}_{\ell} \sk{\Rtt}_{\ell}\|_2.
          \end{align*}
		\normalsize
		In turn, we have
		\small 
		\begin{align*}
		\|(\sk{\Roo}_{\ell})^{-1} (\pre{\Roo}_{\ell})^{-1}\pre{\Rot}_{\ell} \sk{\Rtt}_{\ell}\|_2 
		&\leq \sigmaMin{\sk{\Roo}_{\ell}}^{-1} \sigmaMin{ \pre{\mtx{R}}}^{-1} \|  \pre{\mtx{R}} \|_2 \| \sk{\Rtt}_{\ell} \|_2  \\
		 &\leq \left (\frac{1+\distortion}{1-\distortion} \right )^{1/2} \sigmaMin{\sk{\Roo}_{\ell}}^{-1} \| \sk{\Rtt}_{\ell} \|_2,
		\end{align*}
		\normalsize
		which finishes the proof.
\end{proof}

It is easy to prove \cref{thm:cqrrpt_RRQR,thm:cqrrpt_SRRQR} from here.
Start by using 
\cref{prop:effective_distortion} to express the bounds from \cref{thm:skRRQR} in terms of the restricted condition number for $\mtx{S}$ on $\range(\mtx{M})$.
Straightforward algebra shows that if \eqref{eq:RRQR_R11} and \eqref{eq:RRQR_R22} hold for $\mtx{X} = \mtx{S}\mtx{M}$, then \eqref{eq:skRRQRa} and \eqref{eq:skRRQRb} imply the claim of \cref{thm:cqrrpt_RRQR}.
Proving \cref{thm:cqrrpt_SRRQR} requires a slight detour to show that $\sigmaMin{\sk{\Roo}_{\ell}}^{-1} \| \sk{\Rtt}_{\ell}\|_2 \leq f_{\ell}^2$ follows from \eqref{eq:skRRQRa} and \eqref{eq:skRRQRb}.
Combine this new bound with \eqref{eq:skRRQRc} to see that if \eqref{eq:RRQR_R12} holds for $\mtx{X} = \mtx{S}\mtx{M}$ then the claim of \cref{thm:cqrrpt_SRRQR} follows.

\section{Technical numerical stability results}\label{app:numerical_stability_proofs}

Here we prove results from \cref{sec:finite_precision}, using notation first defined on page \pageref{page:numerical_stability_notation}.

\subsection{Proof of \cref{thm:stabchol}}\label{subapp:analysis:stabchol}
The relation~\cref{eq:CholQRprop2} follows directly from~\cite{yamamoto2015roundoff}. To show~\cref{eq:CholQRprop1}, notice that 
\begin{equation} \label{eq:CholQRproof1}
\|\mtx{M}_n - \fmtx{Q}_k \fmtx{R}_k\|_\mathrm{F} = \|\mtx{M}_n - \pre{\fmtx{M}}_k \sk{\trun{\fmtx{R}}}_k + \mtx{E}_1 \sk{\trun{\fmtx{R}}}_k + \fmtx{Q}_k \mtx{E}_2\|_\mathrm{F}, 
\end{equation}
where $\mtx{E}_1 = \pre{\fmtx{M}}_k - \fmtx{Q}_k \pre{\fmtx{R} }_k$ and 
$\mtx{E}_2 =  \fmtx{R}_k - \pre{\fmtx{R} }_k \sk{\trun{\fmtx{R}}}_k$. According to~\cite{yamamoto2015roundoff}, it holds that $\|\mtx{E}_1\|_\mathrm{F} \leq 8.4 \|\pre{\fmtx{M}}_k \|_2n^2\roundoff$. By combining this relation with~\cref{eq:CholQRprop2} we deduce that $\|\pre{\fmtx{R} }_k \|_\mathrm{F} \leq 1.01 \|\pre{\fmtx{M}}_k \|_\mathrm{F} $. Furthermore, by standard rounding bounds for matrix multiplication~\cite{Higham:2002:book}, we have $\|\mtx{E}_2\|_\mathrm{F} \leq \frac{n\roundoff}{1-n\roundoff} \|\pre{\fmtx{R} }_k\|_\mathrm{F} \|\sk{\trun{\fmtx{R}}}_k\|_\mathrm{F} $. Next, from~\cref{eq:CholQRprop0} it can be deduced that $\|\pre{\fmtx{M}}_k \sk{\trun{\fmtx{R}}}_k \|_2 \leq 1.01 \|\mtx{M}\|_2$, and therefore
\[
    \|\pre{\fmtx{M}}_k \|_2\|\sk{\trun{\fmtx{R}}}_k\|_2\leq  6 \cond(\pre{\fmtx{M}}_k)^{-1}  \|\pre{\fmtx{M}}_k \|_2\|\sk{\trun{\fmtx{R}}}_k\|_2\leq  6 \sigmaMin{\pre{\fmtx{M}}_k} \|\sk{\trun{\fmtx{R}}}_k\|_2\leq 6.06 \|\mtx{M}\|_2.
\]
By plugging the obtained bounds for $\|\mtx{E}_1\|_\mathrm{F}$, $\|\pre{\fmtx{R} }_k \|_\mathrm{F}$, $\|\mtx{E}_2\|_\mathrm{F}$  and $\|\pre{\fmtx{M}}_k \|_2\|\sk{\trun{\fmtx{R}}}_k\|_2$ to~\cref{eq:CholQRproof1}, we obtain 
\begin{align*}
\|\mtx{M}_n - \fmtx{Q}_k \fmtx{R}_k\|_\mathrm{F} - \|\mtx{M}_n - \pre{\fmtx{M}}_k \sk{\trun{\fmtx{R}}}_k\|_\mathrm{F} 
    &\leq  \|\mtx{E}_1\|_\mathrm{F} \|\sk{\fmtx{R}}_k\|_2+ \|\fmtx{Q}_k\|_2\|\mtx{E}_2\|_\mathrm{F} \\
& \leq 8.4 \|\pre{\fmtx{M}}_k \|_2n^2\roundoff \|\sk{\trun{\fmtx{R}}}_k \|_2+ 1.01 n\roundoff \|\pre{\fmtx{R} }_k\|_\mathrm{F} \|\sk{\trun{\fmtx{R}}}_k\|_\mathrm{F} \\
& \leq 51 n^{2} \roundoff \|\mtx{M}\|_2+ 1.03 n\roundoff \|\pre{\fmtx{M}}_k\|_\mathrm{F} \|\sk{\trun{\fmtx{R}}}_k\|_\mathrm{F} \\
&\leq (50.5 n^{2}\roundoff +6.3 n^{2} \roundoff)\|\mtx{M}\|_2,
\end{align*}
which completes the proof.

\subsection{Preconditioner stability}\label{subsec:stab_precond}

\subsubsection{Assumptions}
To characterize the numerical stability of CQRRPT's preconditioner, we need assumptions on the accuracy of the operations at Lines \ref{line:form_sk} to \ref{line:M_pre_def} of \cref{alg:CQRRPT} under finite precision arithmetic. 
We consider a scenario where Line \ref{line:M_pre_def}, which dominates the preconditioner's computational cost, is computed using unit roundoff $\roundoff$, whereas cheaper Lines \ref{line:form_sk} to \ref{line:k_def} are computed at a higher precision, using a roundoff $\tilde{\roundoff} = m^{-1} F(n)^{-1}\roundoff$, where $F(n)$ is a low-degree polynomial.
We use this computational model to illustrate a crucial property of the preconditioner: the ability to ensure numerical stability with a roundoff $\roundoff$ for dominant operations not constrained by $m$. This property is of interest for integrating CQRRPT into multi- or low-precision arithmetic architectures.
Clearly, the numerical stability of CQRRPT computed with two unit roundoffs also implies its stability when computed with a single roundoff.
Define,
\begin{equation} \label{eq:errmat2}
\begin{split}
\mtx{E}_1  &:= \mtx{S} \mtx{M}_k  - \sk{\fmtx{M}}_k, \\
\mtx{E}_2  &:= \sk{\fmtx{M}} - \sk{\fmtx{Q}}_k \sk{\fmtx{R}}_k  \\
\mtx{E}_3  &:= \mtx{M}_k - \pre{\fmtx{M}}_k \sk{\fmtx{A}}_k \\
\mtx{E}_4  &:= \mtx{S} \xoverline{\mtx{M}}_k - \sk{\xoverline{\fmtx{M}}}_k, 
\end{split}
\end{equation}
where $\xoverline{\mtx{M}}_k := \mtx{M}[\fslice,J[\tslice{k+1}n]] \quad\text{and}\quad \sk{\xoverline{\mtx{M}}}_k:= \sk{\mtx{M}}[\fslice,J[\tslice{k+1}n]]$.

We require that several bounds hold on these matrices.
The choice for these bounds begins with a base assumption that the unit roundoff $\roundoff$ and the dimensions $n$ and $k$ satisfy
\begin{subequations} \label{eq:ass200} 
	\begin{equation} \label{eq:ass20} 
	   \roundoff \leq 0.001 n^{-\frac{3}{2}} k^{-\frac{5}{2}}.
	\end{equation}
    With that, the first round of bounds that we discuss are 
    \begin{align}
    	&\|{\mtx{E}_1}[\fslice{},j]\|_2 \leq 0.01k^{-\frac{1}{2}}\roundoff   \|\mtx{M}_k[\fslice{},j] \|_2 \text{ for $1\leq j \leq k$,}  &  \label{eq:ass21} \\
    	&\|{\mtx{E}_2}[\fslice{},j]\|_2 \leq 0.01  k^{-\frac{1}{2}}\roundoff   \|\sk{\fmtx{M}}[\fslice{},j] \|_2  \text{ for $1\leq j \leq n$, ~~~and } &\|{\sk{\fmtx{Q}}_k}^{\trans} {\sk{\fmtx{Q}}_k} - \mtx{I} \|_\mathrm{F} \leq 0.1 \roundoff.  \label{eq:ass22}
	\end{align}
    It is easy to ensure that these conditions hold.
    To begin, we note that the classical worst-case rounding analysis ensures  
    $|\mtx{E}_1| \leq \frac{m \tilde{\roundoff}}{1-m \tilde{\roundoff}}  | \mtx{S} ||\mtx{M}_k|$.
    This tells us that \cref{eq:ass21} can be achieved if $\tilde{\roundoff} = \mathcal{O}(m^{-1} n^{-1} \roundoff)$ and $\|\mtx{S} \|_{\mathrm{F}} = \mathcal{O}(\sqrt{m})$.
    Meanwhile, the condition~\cref{eq:ass22} can be achieved by using any stable QRCP factorization executed in sufficient precision.
    This for instance includes the Householder QR with unit roundoff $\tilde{\roundoff}=\mathcal{O}(n^{-1}d^{-\frac{3}{2}}\roundoff)$ or Givens QR with unit roundoff $\tilde{\roundoff}=\mathcal{O}(n^{-1}d^{-\frac{1}{2}}\roundoff)$~\cite{Higham:2002:book}.

    Our next round of assumptions is more substantial.
	In particular, assume that the output of the black-box \code{qrcp} function satisfies the rank-revealing properties
	\begin{align}
    	&\sigmaSub{\mathrm{min}}{\sk{\fmtx{A}}_k} \geq 0.5 n^{-\frac{1}{2}} k^{-\frac{1}{2}} \sigmaSub{k}{\sk{\fmtx{M}}},  &\|\sk{\fmtx{C}}_{k-1}\|_2 \leq 2n^\frac{1}{2} k^\frac{1}{2} \sigmaSub{k}{\sk{\fmtx{M}}},  \label{eq:ass25}	\\
    	&\|(\sk{\fmtx{A}}_k)^{-1} \sk{\fmtx{B}}_k \|_\mathrm{F}  \leq 2 n^\frac{1}{2} k^\frac{1}{2}. &  \label{eq:ass26}
    \end{align}
     These conditions can be achieved with the strong rank-revealing QR method from~\cite{GE:1996} with unit roundoff $\tilde{\roundoff}=\mathcal{O}(n^{-1}d^{-\frac{1}{2}}\roundoff)$.
     The method from~\cite{GE:1996} contains an extra parameter $f$ that in our case should be taken as, say, $1.5$.
     Then the $\mathtt{qrcp}$ subroutine on Line \ref{line:qrcp_sk} will take a negligible amount $\mathcal{O}(d n^2\log n)$ of flops and satisfy~\cref{eq:ass25,eq:ass26}.
     It is important to note that in practical applications, the condition \cref{eq:ass25,eq:ass26} is expected to hold also for traditional QRCP with max-norm column pivoting.
     However, there are a few exceptional cases in which the traditional QRCP would not satisfy~\cref{eq:ass26}.
 
	Next, according to \cite[Theorem 8.5]{Higham:2002:book}, we have $\pre{\fmtx{M}}_k[j,\fslice{}] (\sk{\fmtx{A}}_k + \Delta \Roo^{(j)}) = \mtx{M}_k{[j,\fslice{}]} $ with $|\Delta \Roo^{(j)}| \leq 1.1\roundoff{}k |\sk{\fmtx{A}}_k|$, which implies that $|\mtx{E}_3|$ is bounded by $1.1\roundoff{}k |\pre{\fmtx{M}}_k||\sk{\fmtx{A}}_k |$ and leads to the following condition
	\begin{align}
    	|{\mtx{E}_3}[:,j]| \leq 1.1 k \roundoff  | \pre{\fmtx{M}}_k| | \sk{\fmtx{A}}[\fslice{},j]| \text{ for $1\leq j \leq k$.}~~~~~~~~
    	\label{eq:ass23}
    	\\
    	\intertext{Furthermore, by the standard rounding analysis we have $|\mtx{E}_4|\leq \frac{m \tilde{\roundoff}}{1-m \tilde{\roundoff}}  | \mtx{S} ||\xoverline{\mtx{M}}_k|, $ which implies}
    	\|{\mtx{E}_4}[:,j]\|_2 \leq 0.1  n^{-\frac{1}{2}} \roundoff{}\|\xoverline{\mtx{M}}_k[\fslice{},j] \|_2 \text{ for $1\leq j \leq n-k$,}   ~~ \label{eq:ass24}
	\end{align}
	if $\tilde{\roundoff} = \mathcal{O}(m^{-1} n^{-1} \roundoff)$ and $\|\mtx{S} \|_{\mathrm{F}} = \mathcal{O}(\sqrt{m})$.
\end{subequations}

With the computational model based on the assumptions~\cref{eq:ass200}, we can now establish a numerical characterization of the CQRRPT preconditioner.

\subsubsection{Main result}\label{subsubsec:thm:stabprecond}
\cref{thm:stabprecond} is our main numerical stability result.
It guarantees stability of CQRRPT's preconditioner under the condition that the truncation error associated with the preceding index $k-1$ is greater than the threshold value $G(n,k)\roundoff$, where $G(n,k)$ is a low-degree polynomial.
Consequently, by using the largest $k$ that satisfies this condition, we will have a factorization $\pre{\fmtx{M}}_k \sk{\fmtx{R}}_k$, which is within a distance of $G(n,k)\roundoff$ from $\mtx{M}_n$ and for which $\pre{\fmtx{M}}_k$ is well-conditioned.
 
\begin{theorem} \label{thm:stabprecond}
    Consider~\cref{alg:CQRRPT} where Line 5 is executed with unit roundoff $\roundoff \leq 0.001n^{-\frac{3}{2}}k^{-\frac{5}{2}}$. Assume that the other lines are executed with unit roundoff $\tilde{\roundoff} = m^{-1}F(n)^{-1} \roundoff$, where $F(n)$ is some low-degree polynomial, and $\|\mtx{S}\|_\mathrm{F}  =\mathcal{O}(\sqrt{m})$, so that the conditions~\cref{eq:ass200} hold.  Assume that $k$ satisfies
	\[ \small 1000 n^{\frac{3}{2}}k^{\frac{5}{2}}\roundoff{}\leq \frac{\|\sk{\fmtx{C}}_{k-1} \|_\mathrm{F}}{\|\sk{\fmtx{R}} \|_2}.\]
	If $\mtx{S}$ is an $\distortion$-embedding for $\range(\mtx{M})$ with $\distortion \leq 1/2$, then
	\begin{subequations}
		\begin{align} 
		\small 
		\frac{\|\mtx{M}_n - \pre{\fmtx{M}}_k \sk{\fmtx{R}}_k\|_\mathrm{F}}{\|\mtx{M} \|_\mathrm{F}} \leq 2 \frac{\|\sk{\fmtx{C}}_k \|_\mathrm{F}}{\|\sk{\fmtx{R}} \|_2} + 10 n k \roundoff \label{eq:main4} \\ 
		0.8  \leq \sigmaMin{\pre{\fmtx{M}}_k} \leq \sigmaMax{\pre{\fmtx{M}}_k}  \leq 1.44.  \label{eq:main5}
		\end{align}
	\end{subequations}
\end{theorem}

Our proof of \cref{thm:stabprecond} uses two intermediate propositions that provide further insights into the stability of the CQRRPT preconditioner.
First, \cref{thm:stabprecond1} concerns how the permutation obtained on Line \ref{line:qrcp_sk} of CQRRPT effectively bounds the condition number of the submatrix $\mtx{M}_k$.  
Second, \cref{thm:RCholQR} involves examining the factors $\pre{\fmtx{M}}_k$ and $\sk{\fmtx{A}}_k$ as an unpivoted ``sketched CholeskyQR'' factorization of $\mtx{M}_k$, which enables us to leverage a result from~\cite{Balabanov:2022:cholQR}.

\begin{proposition} \label{thm:stabprecond1}
 Consider~\cref{alg:CQRRPT} where Line 5 is executed  with unit roundoff $\roundoff \leq 0.001n^{-\frac{3}{2}}k^{-\frac{5}{2}}$ and other lines are executed with unit roundoff $\tilde{\roundoff} = m^{-1}F(n)^{-1} \roundoff$, where $F(n)$ is some low-degree polynomial, and $\|\mtx{S}\|_\mathrm{F}  =\mathcal{O}(\sqrt{m})$, so that the conditions~\cref{eq:ass200} hold. Assume that $k$ is such that \[
        n^{\frac{3}{2}}k\roundoff{}\leq {\|\sk{\fmtx{C}}_{k-1} \|_\mathrm{F}}\|\sk{\fmtx{R}} \|_2^{-1}.
    \]
    If $\mtx{S}$ is an $\distortion$-embedding for $\mtx{M}_k$ with $\distortion \leq 1/2$, then we have 
    \[
        \cond(\mtx{M}_k) \leq 10 n^{\frac{3}{2}}k \|\sk{\fmtx{C}}_{k-1} \|_\mathrm{F}^{-1} \|\sk{\fmtx{R}} \|_2  \leq 2.5 \roundoff{}^{-1}.
    \]
	\begin{proof}
		 Denote $\|\sk{\fmtx{C}}_{k-1} \|_\mathrm{F}\|\sk{\fmtx{R}}_k \|_2^{-1}$ by $\tau_{k-1}$. Frame the assumptions~\cref{eq:ass200}. Notice that by~\cref{eq:ass25}, 
		\begin{equation} \label{eq:RRRCholeskyQR11}	
		\sigmaMin{\sk{\fmtx{A}}_k} \geq 4^{-1} n^{-1} k^{-1} \|\sk{\fmtx{C}}_{k-1} \|_2  \geq 4^{-1} n^{-\frac{3}{2}} k^{-1} \tau_{k-1} \|\sk{\fmtx{A}}_k \|_2.
		\end{equation}	
		Thus, it is deduced that
		\begin{equation} \label{eq:RRRCholeskyQR112}	
		\cond(\sk{\fmtx{A}}_k) \leq 4 n^{\frac{3}{2}}k \tau_{k-1}^{-1}\leq \roundoff{}^{-1}.
		\end{equation}
		
		Furthermore, by~\cref{eq:ass22} we have 
		\begin{equation} \label{eq:RRRCholeskyQR12}	
		\|\sk{\fmtx{M}}_k\|_2 \leq \|\sk{\fmtx{Q}}_k\|_2 \|\sk{\fmtx{A}}_k\|_2+ \|\mtx{E}_2[:,1\fslice{k}]\|_2 \leq 1.01 \|\sk{\fmtx{A}}_k\|_2,
		\end{equation}
		and 
		\begin{equation} \label{eq:RRRCholeskyQR13}	
		\begin{split}
		\small 
		\sigmaMin{\sk{\fmtx{M}}_k} &\geq \sigmaMin{\sk{\fmtx{Q}}_k} \sigmaMin{\sk{\fmtx{A}}_k} - \|\mtx{E}_2[:,1\fslice{k}]\|_2 \\
		&\geq 0.99 \sigmaMin{\sk{\fmtx{A}}_k} - 0.01u  \|\sk{\fmtx{M}}_k \|_2  
		 \geq  0.99 \sigmaMin{\sk{\fmtx{A}}_k} - 0.011\roundoff \|\sk{\fmtx{A}}_k \|_2 \\
		& \geq  \sigmaMin{\sk{\fmtx{A}}_k} \left (0.99 - 0.011\roundoff{}\cond(\sk{\fmtx{A}}_k) \right ) \geq  0.97  \sigmaMin{\sk{\fmtx{A}}_k}.
		\end{split}
		\end{equation}
		Consequently,
		\begin{equation} \label{eq:RRRCholeskyQR14}	
		\cond(\sk{\fmtx{M}}_k) \leq 4.2 n^{\frac{3}{2}} k \tau_{k-1}^{-1} \leq 1.05 \roundoff{}^{-1}. 
		\end{equation}			
		Next, by~\cref{eq:ass21} we get 	
		\begin{equation} \label{eq:RRRCholeskyQR15}			
		\|\mtx{S} \mtx{M}_k \|_2 \leq \|\sk{\fmtx{M}}_k \|_2 + \|\mtx{E}_1 \|_2 \leq 1.01 \|\sk{\fmtx{M}}_k\|_2, 
		\end{equation}
		which due to the $\distortion$-embedding property of $\mtx{S}$ implies that 
		\begin{equation} \label{eq:RRRCholeskyQR16}			
		\| \mtx{M}_k \|_2 \leq 1.5 \| \sk{\fmtx{M}}_k \|_2.
		\end{equation}			
		We also have by~\cref{eq:ass21,eq:RRRCholeskyQR14},
		\begin{equation} \label{eq:RRRCholeskyQR17}			
		\begin{split}
		\small 
		\sigmaMin{\mtx{S} \mtx{M}_k} &\geq \sigmaMin{\sk{\fmtx{M}}_k} - \|\mtx{E}_1\|_2 \geq  \sigmaMin{\sk{\fmtx{M}}_k} - 0.01u {\|\mtx{M}_k \|_2}  \\ 
		& \geq  \sigmaMin{\sk{\fmtx{M}}_k} -0.015u \|\sk{\fmtx{M}}_k\|_2 
		\geq  0.92 \sigmaMin{\sk{\fmtx{M}}_k}.
		\end{split}
		\end{equation}			
		Consequently, by the $\distortion$-embedding property of $\mtx{S}$ and~\cref{eq:RRRCholeskyQR16,eq:RRRCholeskyQR17},
		\begin{equation*}
		\cond( \mtx{M}_k) \leq \sqrt{\frac{1+\distortion}{1-\distortion}}\cond(\mtx{S} \mtx{M}_k) \leq 1.91 \cond(\sk{\fmtx{M}}_k) \leq  10 n^{\frac{3}{2}} k \tau_{k-1}^{-1} \leq 2.5 u^{-1},
		\end{equation*}
		which finishes the proof.		
	\end{proof} 
\end{proposition}

\begin{proposition}[Corollary of Theorem 5.2 from~\cite{Balabanov:2022:cholQR}] \label{thm:RCholQR} 	
Consider~\cref{alg:CQRRPT} where Line 5 is executed with unit roundoff $\roundoff \leq 0.01 n^{-\frac{3}{2}} \cond(\mtx{M}_k)^{-1}$ and other lines are executed with unit roundoff $\tilde{\roundoff} = m^{-1} F(k)^{-1} \roundoff$, where $F(k)$ is some low-degree polynomial, and $\|\mtx{S}\|_\mathrm{F}  =\mathcal{O}(\sqrt{m})$, so that the assumptions~\cref{eq:ass200} hold. If $\mtx{S}$ is an $\distortion$-embedding for $\mtx{M}_k$ with $\distortion\leq 1/2$ then
	\begin{align}	
	\mtx{M}_k +{\Delta} \mtx{M} &= \pre{\fmtx{M}}_k \sk{\fmtx{A}}_k  \label{eq:main1} \\
	\intertext{with $\|{\Delta} \mtx{M}_k{[\fslice{},j]}\|_2\leq  2.1 k \roundoff \|\mtx{M}{[\fslice{},j]}\|_2 $ for $1 \leq j \leq k$. Furthermore, we have}
	(1+\distortion)^{-1/2} - 4 k^\frac{3}{2} \roundoff  \cond(\mtx{M}_k) \leq \sigmaMin{\pre{\fmtx{M}}_k}&\leq \sigmaMax{\pre{\fmtx{M}}_k}  \leq (1-\distortion)^{-1/2} + 4 k^\frac{3}{2} \roundoff \cond(\mtx{M}_k).\label{eq:main2}
	\end{align}
	\normalsize
\end{proposition}
\normalsize

We are now ready to present the proof of~\cref{thm:stabprecond}.

\begin{proof}[Proof of~\cref{thm:stabprecond}]	
	Denote $\|\sk{\fmtx{C}}_{k} \|_\mathrm{F}\|\sk{\fmtx{R}}_k \|_2^{-1}$ by $\tau_{k}$. By~\cref{eq:ass21,eq:ass22,eq:ass24,eq:ass26} we have	
    \begin{small}
	\begin{equation} \label{eq:RRRCholeskyQR20}
	\begin{split} 
	\|\xoverline{\mtx{M}}_k  &- \pre{\fmtx{M}}_k \sk{\trun{\fmtx{B}}}_k \|_\mathrm{F} \leq \| \xoverline{\mtx{M}}_k  - \mtx{M}_k (\sk{\fmtx{A}}_k)^{-1} \sk{\fmtx{B}}_k\|_\mathrm{F} + \|\mtx{E}_3 (\sk{\fmtx{A}}_k)^{-1} \sk{\fmtx{B}}_k \|_\mathrm{F} \\
	&\leq (1-\distortion)^{-1/2} \|\mtx{S}(\xoverline{\mtx{M}}_k  - \mtx{M}_k (\sk{\fmtx{A}}_k)^{-1} \sk{\fmtx{B}}_k) \|_\mathrm{F} + \|\mtx{E}_3\|_\mathrm{F} \|(\sk{\fmtx{A}}_k)^{-1} \sk{\fmtx{B}}_k \|_2\\ &\leq (1-\distortion)^{-1/2} ( \|\mtx{E}_4 +  \mtx{E}_5 + \mtx{E}_6 + \mtx{E}_{7} \|_\mathrm{F} ) + 2n^\frac{1}{2} k^\frac{1}{2}\|\mtx{E}_3 \|_\mathrm{F},
	\end{split}
	\end{equation}
    \end{small}
	where $\mtx{E}_4:= \mtx{S} \xoverline{\mtx{M}}_k - \sk{\xoverline{\mtx{M}}}_k, \mtx{E}_5: = \sk{\xoverline{\mtx{M}}}_k - \sk{\fmtx{Q}}_k \sk{\fmtx{B}}_k, \mtx{E}_{6}:=(\sk{\fmtx{Q}}_k \sk{\fmtx{A}}_k-\sk{\fmtx{M}}_k) (\sk{\fmtx{A}}_k)^{-1} \sk{\fmtx{B}}_k,$ and $\mtx{E}_7:=( \sk{\fmtx{M}}_k - \mtx{S} \mtx{M}_k ) (\sk{\fmtx{A}}_k)^{-1} \sk{\fmtx{B}}_k$ satisfy the following bounds 
     \begin{small}
	\begin{align*}
	\|\mtx{E}_4 \|_\mathrm{F} &= \|\mtx{S} \xoverline{\mtx{M}}_k - \sk{\xoverline{\mtx{M}}}_k \|_\mathrm{F} \leq 0.1\roundoff{}\|\xoverline{\mtx{M}}_k \|_2  \\
	\|\mtx{E}_5 \|_\mathrm{F} &= \|\sk{\xoverline{\mtx{M}}}_k - \sk{\fmtx{Q}}_k \sk{\fmtx{B}}_k \|_\mathrm{F} \leq \|\sk{\fmtx{Q}}_k\|_2\|\sk{\fmtx{C}}_k \|_\mathrm{F}+ \|\mtx{E}_2 \|_\mathrm{F} \\ &\leq 1.01 {\tau_k} \|\sk{\fmtx{R}}_k\|_2 + \|\mtx{E}_2 \|_\mathrm{F} \leq (1+\distortion)^{1/2}(1.02 {\tau_k} + 0.01\sqrt{n}u )\|\mtx{M}\|_\mathrm{F} \\
	\|\mtx{E}_{6} \|_\mathrm{F} &= \|(\sk{\fmtx{Q}}_k \sk{\fmtx{A}}_k-\sk{\fmtx{M}}_k) (\sk{\fmtx{A}}_k)^{-1} \sk{\fmtx{B}}_k\|_\mathrm{F} \\
	& \leq \|\sk{\fmtx{Q}}_k \sk{\fmtx{A}}_k-\sk{\fmtx{M}}_k\|_\mathrm{F} \|(\sk{\fmtx{A}}_k)^{-1} \sk{\fmtx{B}}_k\|_2  \\ &\leq 2 n^\frac{1}{2} k^\frac{1}{2} \|\mtx{E}_2[:,1\fslice{k}] \|_\mathrm{F} \leq 0.02 n^\frac{1}{2} k^\frac{1}{2}\roundoff{}\|\sk{\fmtx{M}}_k \|_2 \leq 0.03 n\roundoff{}\|\mtx{M}_k \|_2
	\\
	\|\mtx{E}_7 \|_\mathrm{F} & = \|( \sk{\fmtx{M}}_k - \mtx{S} \mtx{M}_k ) (\sk{\fmtx{A}}_k)^{-1} \sk{\fmtx{B}}_k\|_2 
	\leq \|\sk{\fmtx{M}}_k - \mtx{S} \mtx{M}_k  \|_\mathrm{F} \|(\sk{\fmtx{A}}_k)^{-1} \sk{\fmtx{B}}\|_2 \\& \leq 2 n^\frac{1}{2} k^\frac{1}{2} \|\mtx{E}_1 \|_\mathrm{F} \leq 0.02 n\roundoff{}\| \mtx{M}_k \|_2.
	\end{align*}
    \end{small}
	By using the triangle inequality and the fact that $\distortion \leq 1/2$, we obtain
	\begin{equation} \label{eq:RRRCholeskyQR21}
	\|\xoverline{\mtx{M}}_k  - \pre{\fmtx{M}}_k \sk{\trun{\fmtx{B}}}_k \|_\mathrm{F} \leq \frac{4}{3} ( 0.2 n \roundoff + 1.02 {\textstyle\frac{5}{4}}{\tau_k} )\|\mtx{M}\|_\mathrm{F} + 2n^\frac{1}{2} k^\frac{1}{2}\|\mtx{E}_3 \|_\mathrm{F}.
	\end{equation}
	Furthermore, from~\cref{thm:stabprecond1} it follows that 
	\begin{equation} \label{eq:RRRCholeskyQR31}
	\cond(\mtx{M}_k) \leq 10 n^{\textstyle \frac{3}{2}} k {\tau_k}^{-1}.
	\end{equation}
	By looking at $\pre{\fmtx{M}}_k$ and $\sk{\fmtx{A}}_k$ as a sketched CholeskyQR factorization of $\mtx{M}_k$, according to~\cref{thm:RCholQR}, we have 
	\begin{subequations}
		\begin{align}	
		\|\mtx{E}_3 \|_\mathrm{F} = \|\mtx{M}_k - \pre{\fmtx{M}}_k \sk{\fmtx{A}}_k\|_\mathrm{F} &\leq 2.1 k \roundoff \|\mtx{M}_k \|_\mathrm{F} \label{eq:RRRmain1} \\
		(1+\distortion)^{-1/2} - 4 k^\frac{3}{2} \roundoff  \cond(\mtx{M}_k) \leq \sigmaMin{\pre{\fmtx{M}}_k}&\leq \sigmaMax{\pre{\fmtx{M}}_k}  \leq (1-\distortion)^{-1/2} + 4 k^\frac{3}{2} \roundoff  \cond(\mtx{M}_k) \label{eq:RRRmain2}
		\end{align}
	\end{subequations}
	
	By combing~\cref{eq:RRRmain1} with~\cref{eq:RRRCholeskyQR21} and using the triangle inequality we obtain~\cref{eq:main4}. By combining~\cref{eq:RRRmain2} with~\cref{eq:RRRCholeskyQR31} we obtain~\cref{eq:main5} and finish the proof. 	
\end{proof}

\section{More on numerical experiments}\label{sec:more_experiments}

\paragraph{Extremely tall matrices.}

The matrices used in \cref{sec:speed_experiments}'s experiments were tall, but not as tall as some readers might anticipate. 
Therefore here we consider matrices with one million rows and $32 \dots 16384$ columns.
\cref{CQRRPT_tall_speed} presents performance results (in canonical GFLOPs/second) for the algorithms considered before.
It shows that for very thin matrices ($n \leq 256$), CQRRPT may be inefficient compared to alternative algorithms.
This can be understood with performance profiling data in \cref{CQRRPT Inner Speed Fig 2}. Specifically, for these extremely tall and thin matrices, the runtime is dominated by sketching, and the time to pivot $\mtx{M}$ before preconditioning takes nearly as long as CholeskyQR.
As before, our experiments were run on the machine described in \cref{table:processor_config} using $48$ threads.

\begin{figure}[!h]
\minipage{0.425\textwidth}
    \centering
    \begin{overpic}[width=1.1\linewidth]{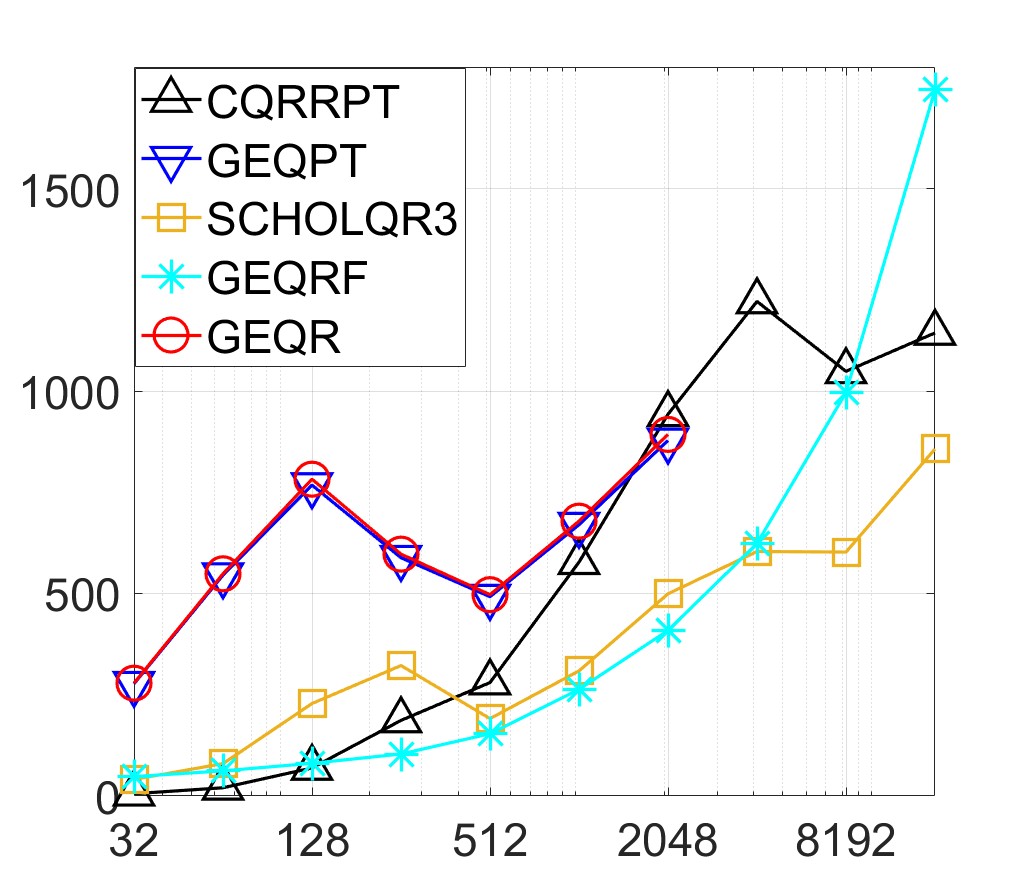}
    \put (-3, 45) {\rotatebox[origin=c]{90}{\textbf{GFLOP/s}}}
    \put (40, -3) {\textbf{columns} ($n$)}
    \end{overpic}
    \vspace{0.1cm}
    \captionof{figure}{\small QR schemes performance comparisons for matrices with one million rows and varying numbers of columns ($32, \ldots, 16384$). The plot does not depict \code{GEQR} and \code{GEQPT} results for the number of columns larger than $2048$, because of an internal overflow of parameters related to the input matrix size and a consequent early termination.}
    \label{CQRRPT_tall_speed}
\endminipage%
\hspace{1cm}
\minipage{0.475\textwidth}
\vspace{1.1cm}
\centering
    \centering
    \begin{overpic}[width=1.1\linewidth]{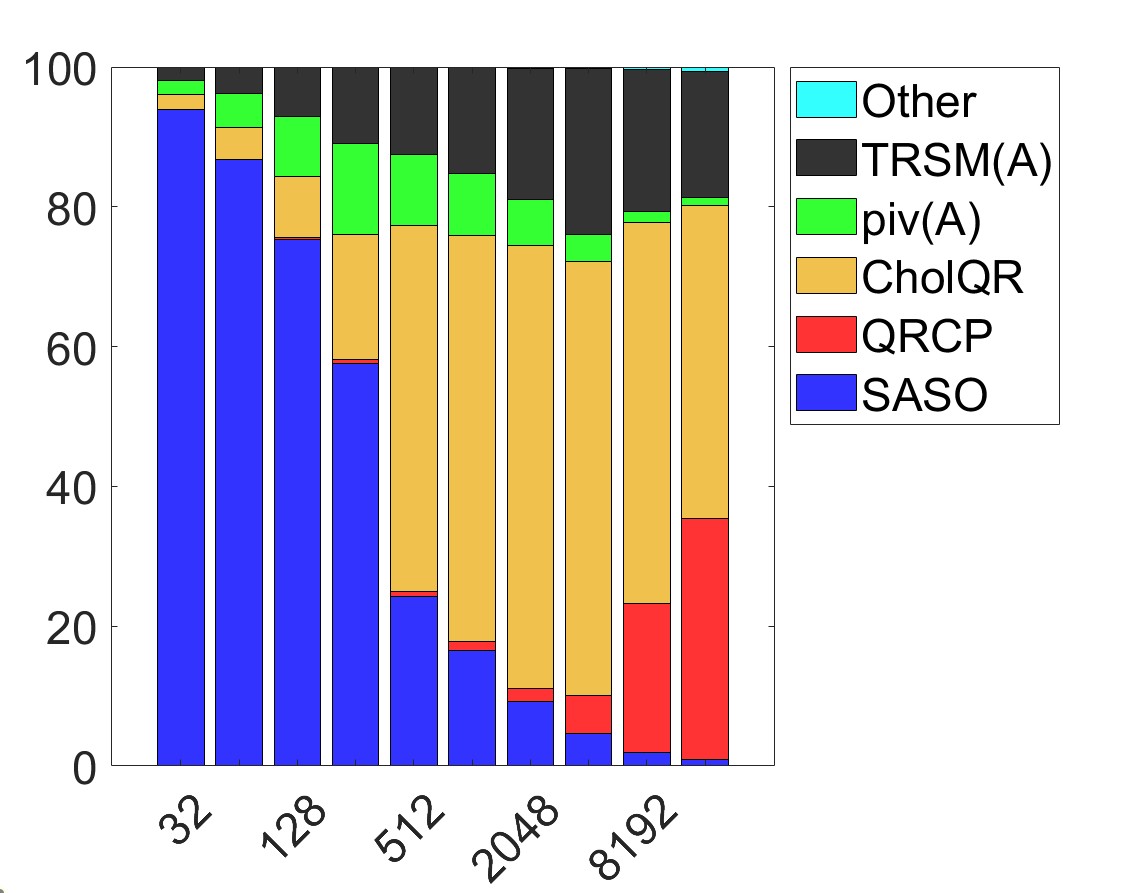}
    \put (-3, 42) {\rotatebox[origin=c]{90}{\textbf{runtime(\%)}}}
    \put (28, -5) {\textbf{columns} ($n$)}
    \end{overpic}
    \vspace{0.2cm}
    \captionof{figure}{\small Percentages of runtime used by CQRRPT's subroutines, for matrices with one million rows and the indicated number of columns.
    Note that although CholQR occupies the dominant portion of runtime across the larger ($n\geq 512$) matrix sizes, its runtime is primarily comprised of vector-vector operations for the smaller matrix sizes. As the matrix size increases, \code{TRSM} routine gradually achieves a \code{GEMM}-like performance level.}
  \label{CQRRPT Inner Speed Fig 2}
\endminipage
\end{figure}
\FloatBarrier

\paragraph{HQRRP performance details.}
\cref{fig:hqrrpvsgeqp3} is our attempt to give a contemporary
refresh of Figures 4 and 5 from the original HQRRP
publication~\cite{MOHvdG:2017:QR}. It is important to note that the
initial HQRRP thread scalability study was performed on Intel Xeon
E5-2695v3 (the Haswell platform) processor with clocked capped at 2.3
GHz featuring ``only'' 14 cores in a single socket with 22nm node size.
Our experiments used dual-socket Intel Xeon Gold 6248R (Cascade Lake
platform) with 24 cores per socket and 48 cores total with 14nm feature
size - a NUMA design.  In essence, the new processor engages a much larger
number of cores in the computation and they have to be coordinated by
on-node interconnect, that is, the design keeps the content of Level 1 caches
coherent through a memory controller state tracking each L1 cache line.
For smaller matrix sizes, lower-level caches, especially Level 3, are very effective in alleviating the demand for cache lines from the main memory and HQRRP behaves as it did on the old Haswell system.
For larger matrix sizes, HQRRP performance drops as its internal pivoted QR function, that is implemented using level 2 \BLAS{} routines, becomes too costly.

\end{document}